\documentclass{article}
\usepackage[utf8]{inputenc}
\usepackage{amsmath,amsthm,amssymb}
\usepackage{graphicx} 
\usepackage{xcolor}
\usepackage{lineno}
\usepackage{enumitem}
\usepackage[margin=3.5cm]{geometry}
\usepackage[normalem]{ulem}
\usepackage{thmtools,thm-restate}
\usepackage{slashbox}
\newcommand{\Z}{\mathbb{Z}_3}
\newcommand{\PP}{\mathcal{P}}
\newcommand{\GG}{\mathcal{G}}
\newcommand{\CC}{\mathcal{C}}
\newcommand\g{\mathfrak{G}}
\newcommand\comp{\mathfrak{C}}
\usepackage{tikz,verbatim}
\usetikzlibrary{shapes.geometric}
\newtheorem{theorem}{Theorem}[section]
\newtheorem{observation}[theorem]{Observation}
\newtheorem{corollary}[theorem]{Corollary}
\newtheorem{lemma}[theorem]{Lemma}
\newtheorem{conjecture}[theorem]{Conjecture}
\newtheorem{definition}[theorem]{Definition}
\newtheorem*{claim*}{Claim}
\newtheorem*{subclaim*}{Subclaim}
\newtheorem*{ack}{Acknowledgments}
\usepackage{pgfplots}

\usetikzlibrary{decorations.markings}
\tikzset{myptr/.style={decoration={markings,mark=at position 1 with %
    {\arrow[scale=3,>=stealth]{>}}},postaction={decorate}}}
\newcommand{\equal}{=}

\tikzset{
blackvertexv2/.style={circle, draw=black!100,fill=black!100,thick, inner sep=0pt, minimum size= 2mm},
dummywhite/.style={circle, draw=white!100,fill=white!100,thick, inner sep=0pt, minimum size= 0.5mm},
ellipsenodev1/.style={ellipse, draw=black!100,fill=none,thick, inner sep=0pt, minimum width= 1.5cm, minimum height = 2.75cm},
ellipsenodev2/.style={ellipse, draw=black!100,fill=none,dashed, inner sep=0pt, minimum width= .75cm, minimum height = 2.25cm},
ellipsenodev3/.style={ellipse, draw=black!100,fill=none,thick, inner sep=0pt, minimum width= 3cm, minimum height = 1cm},
}
\definecolor{hanpurple}{rgb}{0.32, 0.09, 0.98}
\definecolor{SoffiaRed}{RGB}{160,20,20}

\newenvironment{subproof}{%
  \begin{proof}[Proof of claim]%
}{%
  \end{proof}%
}
\newenvironment{subsubproof}{%
  \begin{proof}[Proof of subclaim]%
}{%
  \end{proof}%
}
\usepackage[colorlinks=true,linkcolor=red!80!black,citecolor=blue]{hyperref}

\title{Flow-critical graphs\thanks{Supported by project 22-17398S (Flows and cycles in graphs on surfaces) of Czech Science Foundation.}}
\author{Arnbjörg Soffía Árnadóttir\thanks{University of Iceland, Reykjavík, Iceland, email: \texttt{asoffia@hi.is}. Supported by the Independent Research Fund Denmark, 8021-00249B AlgoGraph and by CNPq (Brazil).}\and Zdeněk Dvořák \thanks{Charles University, Prague, Czechia, email: \texttt{rakdver@iuuk.mff.cuni.cz}.} \and Bernard Lidický \thanks{Iowa State University, Ames, Iowa, United States, email: \texttt{lidicky@iastate.edu}. Research supported by NSF DSM-2152490 and Scott Hanna Professorship.} \and Benjamin Moore \thanks{University of Manitoba, Winnipeg, Canada, email: \texttt{Ben.Moore@umanitoba.ca}. Research partially supported by ERC Starting Grant ``RANDSTRUCT'' No.\ 101076777, and  acknowledges the support of the Natural Sciences and Engineering Research Council of Canada
(NSERC) [RGPIN-2025-07125],  Cette recherche a été financé par le Conseil de recherches en sciences naturelles et en génie du Canada (CRSNG) [RGPIN-2025-07125]} \and Evelyne Smith-Roberge \thanks{Illinois State University, Illinois, United States, email: \texttt{esmithr@ilstu.edu},} \and Robert Šámal \thanks{Charles University, Prague, Czechia, email: \texttt{samal@iuuk.mff.cuni.cz}.}}
\date{\today}

\begin{document}

\maketitle


\begin{abstract}
 Lov\'{a}sz et al. proved that every $6$-edge-connected graph has a nowhere-zero $3$-flow. In fact, they proved a more technical statement which says that there exists a nowhere-zero $3$-flow that extends the flow prescribed on the incident edges of a single vertex $z$ with bounded degree. We extend this theorem of Lov\'{a}sz et al. to  allow $z$ to have arbitrary degree, but with the additional assumption that there is another vertex $x$ with large degree and no small cut separating  $x$ and $z$. Using this theorem, we prove two results regarding the generation of minimal graphs with the property that prescribing the edges incident to a vertex with specific flow does not extend to a nowhere-zero $3$-flow. We use this to further strengthen the theorem of Lov\'{a}sz et al., as well as give a density bound on flow-critical graphs with at most one vertex of degree at least $7$.
\end{abstract}

\tableofcontents

\section{Introduction}
Subsection \ref{subsec:background} sets the stage for our results, an overview of which is given in Subsection \ref{subsec:results}. Subsection \ref{subsec:paperoutline} contains an outline of the rest of the paper.

\subsection{Background and context}\label{subsec:background}
Throughout the paper, graphs are allowed to have multiple edges, but are
loopless.  A \emph{$k$-colouring} of a graph $G$ is a function $f:V(G)
\rightarrow \{0,\ldots,k-1\}$ such that for every edge $e =xy \in E(G)$, we
have $f(x) \neq f(y)$.

We start our story by recalling the famous theorem of
Grötzsch: Every triangle-free planar graph is $3$-colourable
\cite{grotzsch1959}. There are many possible avenues for generalizations of
Grötzsch's Theorem. For example, one could try to add some number of triangles.
Along this direction, we see that one can allow arbitrarily many triangles if
they are far apart \cite{havelpaper}, or up to four triangles assuming certain
structures are avoided \cite{fourtriangles}.

Another possibility would be to
generalize past planarity. Here, when attempting to increase the genus of the
surface, one immediately runs into problems: Even on the projective plane there
are triangle-free graphs which are not $3$-colourable (although, they can be
completely characterized; see \cite{gimbelthomassentheorem}, and a similar
although more complicated situation occurs on the torus
\cite{torodialcharacterization}). For more general surfaces, some structure is
known but there is not a complete characterization, see \cite{DVORAK2024517,
trfree2,trfree3,trfree4,trfree7}.

As generalizing to larger genus surfaces
quickly leads to a messy situation, one might ask if the complicated condition
of planarity can be replaced with a simpler condition. The most natural such
condition would be an edge-density condition -- as in, show all vertex and
edge-minimal graphs with no $4$-colouring have many edges and then deduce
Gr\"{o}tzsch's Theorem from the density condition and Euler's formula. This
almost works, as shown in \cite{Shortproof}; however one cannot quite deduce
Grötzsch's theorem from the density condition given in \cite{Shortproof}
without appealing to planarity, one first needs to preform a routine reduction
of $4$-faces. As there exists infinitely many vertex and edge-minimal
triangle-free graphs with no $3$-colouring with average degree roughly
$\frac{10}{3}$ (see \cite{densityboundtrianglefree}), and triangle-free planar
graphs have average degree arbitrarily close to $4$, this approach is not
likely to completely succeed, so another idea is needed.

A different approach
to removing the planarity condition is to move from the colouring framework to
that of nowhere-zero flows. Indeed, this is the focus of this paper. With
$\deg^+(v)$ and $\deg^-(v)$ denoting the in- and outdegree of a vertex $v$,
respectively, a \emph{nowhere-zero $3$-flow}  is an orientation of $G$ such
that $\deg^+(v) -\deg^-(v) \equiv 0 \pmod 3$ for every $v \in V(G)$. Normally,
this would be called a modulo $3$ orientation of $G$, and one would prove that
this is equivalent to a nowhere-zero $3$-flow, but we simply take this as a
definition (see \cite{circularflows}).  Tutte proved that a planar graph admits a $3$-colouring if and
only if the dual graph admits a nowhere-zero $3$-flow
\cite{tutte1954contribution}, and thus Grötzsch's theorem implies that
$4$-edge-connected planar graphs admit nowhere-zero $3$-flows. Famously, Tutte
conjectured that planarity is not required in the previous statement:

\begin{conjecture}[$3$-flow-conjecture, \cite{tutte1954contribution}]
Every $4$-edge-connected graph admits a nowhere-zero $3$-flow.
\end{conjecture}

There has been much work towards this conjecture; see for example \cite{KOCHOL,torodial3flows,ltwz,weak3flow} for a non-exhaustive list. We highlight the most important progress for this paper below.

\begin{theorem}[Kochol~\cite{KOCHOL}]
\label{Kocholtheorem}
If every $5$-edge-connected graph admits a nowhere-zero $3$-flow, then the $3$-flow conjecture is true.
\end{theorem}

\begin{theorem}[L. M. Lovász et al.~\cite{ltwz}]
\label{6edgeconnected}
Every $6$-edge-connected graph admits a nowhere-zero $3$-flow.
\end{theorem}

Thus in a sense we are very close to proving the $3$-flow conjecture: We know
the conjecture holds for $6$-edge-connected graphs, and further know that in
order to establish the full conjecture, it suffices to prove it for
$5$-edge-connected graphs.

As a possible approach to the $3$-flow conjecture,
one could draw inspiration from proofs of Grötzsch's Theorem. For example, one
can try to mimic the potential method proof of Grötzsch's Theorem given by
Kostochka and Yancey \cite{Shortproof}. To understand their approach, we need a
definition: We say a graph is \textit{$4$-critical} if it is not
$3$-colourable, but all proper subgraphs are. Kostochka and Yancey showed that
$4$-critical graphs have a lot of edges\textemdash more precisely, that if a
graph is $4$-critical, then $|E(G)| \geq \frac{5|V(G)|-2}{3}$. From this, one
can deduce that triangle-free graphs are $3$-colourable by first observing that
a minimal counterexample to Grötzsch's Theorem is $4$-critical and has no faces
of length $4$, and then observing that by Euler's formula, such graphs do not
have enough edges to be $4$-critical.

With this and flow-colouring duality in
mind, Theorem \ref{Kocholtheorem} is the analogue for removing $4$-faces from a
minimal counterexample to the $3$-flow conjecture. Following the general
approach of Kostochka and Yancey, the next step would be to prove that graphs
that are in some way critical for not having a nowhere-zero $3$-flow have few
edges, and deduce the $3$-flow conjecture from this. This strategy has been
proposed before, but we need some definitions to  explain precisely. 

For a
partition $\PP$ of the vertex set of a graph $G$, we let $G /
\PP$ be the graph obtained by identifying the vertices in each part of
$\PP$ to a single vertex and then removing all of the loops. If
$\PP$ contains a part with at least two vertices, then we say
$\PP$ is \textit{non-trivial}, and that $G / \PP$ is a
\textit{contraction} of $G$. By appealing to flow-colouring duality, it is natural
to define flow-criticality in the following way: We say a graph $G$ is
\textit{connected-flow-critical} if $G$ does not admit a nowhere-zero $3$-flow
and for every non-trivial partition $\PP$ of $V(G)$ where each part
induces a connected graph, $G / \PP$ admits a nowhere-zero $3$-flow. We
note that while the above definition is natural, there is really no reason to
enforce the connectivity condition in the parts. If one modifies the definition
above to say $G$ does not admit a nowhere-zero $3$-flow, but $G / \PP$
does for all non-trivial partitions $\PP$, then we obtain
\textit{flow-critical} graphs. This notion will be used throughout the paper,
and is more flexible than the connected version. Nevertheless, in the
literature, more work has been done on connected-flow-critical graphs, and the
following density bound conjectures are stated for connected-flow-critical
graphs.

\begin{restatable}[Li et al.~\cite{li20223}]{conjecture}{conjdensitycrit}
\label{conj:densitycrit1}
For any connected-flow-critical graph $G$ on at least seven vertices, $|E(G)| \leq 3|V(G)|-8$.
\end{restatable}

This unfortunately is not strong enough to deduce the $3$-flow conjecture. However, if true it implies Theorem \ref{6edgeconnected} and there are infinitely many graphs which attain the bound, as shown by the examples in \cite{li20223}. Nevertheless, these examples have many vertices of degree 3, and hence the authors of \cite{li20223} suggest the following, which does imply the $3$-flow conjecture.

\begin{conjecture}[Li et al.~\cite{li20223}]\label{conj:li}
For any connected-flow-critical graph $G$ on at least seven vertices with $n_{3}$ vertices of degree $3$, we have
$$|E(G)| < \frac{5|V(G)|}{2} +n_{3}.$$
\end{conjecture}

Aside from the bounds given in \cite{li20223}, recent progress on lower bounds in \cite{dvořák2024sparsity} and some progress when the genus of the graph is bounded in \cite{bojanzdenek}, Conjecture \ref{conj:densitycrit1} is still wide open. 

\subsection{Results}\label{subsec:results}
We now pivot to an overview of the results in this paper. Along the lines of Conjecture \ref{conj:densitycrit1}, we prove the following.
\begin{restatable}{theorem}{thmintroweakerversion}\label{thm:introweakerversion}
Let $G$ be a flow-critical graph.  If $G$ has at most one vertex of degree at least $7$,
then $|E(G)|\le 3|V(G)|-5$.
\end{restatable}

 One might anticipate a proof of Theorem \ref{thm:introweakerversion} following the proof strategy given in \cite{Shortproof}; in fact it does not use the potential method (the technique used in \cite{Shortproof}) at all. Rather, most of the work  comes from a more technical theorem which generalizes Theorem \ref{6edgeconnected} and relies on a clever inductive argument. This is not particularly new: There are many induction-based proofs of Grötzsch's Theorem (for example \cite{THOMASSEN2003189}), and this idea leads to a proof of Theorem \ref{6edgeconnected}. As our theorem is quite technical even to state, let us motivate some of the following required definitions by analogy to their counterparts for colouring. 
 
 First, in Thomassen's \cite{THOMASSEN2003189} proof of Grötzsch's Theorem he proves a theorem on list-colouring: A generalization of ordinary colouring wherein we prescribe what colour choices are available to each vertex. Similarly, we will want to work with a ``list-colouring" version of nowhere-zero flows. 

\begin{definition}
Let $G$ be a graph and $\beta:V(G) \rightarrow \Z$ a function.  For $A\subseteq V(G)$, we define
$\beta(A):=\sum_{v\in A} \beta(v)$. We say that $\beta$ is a \emph{$\Z$-boundary} for $G$ if for every component $C$ of $G$, we have that $\beta(V(C)) \equiv 0 \pmod 3$. If $\beta$ is a $\Z$-boundary for $G$, we say the pair $\g = (G,\beta)$ is a \emph{$\Z$-bordered graph}.
\end{definition}

With that, we can define a notion that is analogous to list-colouring for flows.

\begin{definition}
Let $\g = (G,\beta)$ be a $\Z$-bordered graph. A \emph{nowhere-zero flow} in $\g$ is an orientation such that for all $v \in V(G)$, we have $\deg^{+}(v) -\deg^{-}(v) \equiv \beta(v) \pmod 3$.
We say that a graph $G$ is $\Z$-connected if $(G,\beta)$ admits a nowhere-zero flow for every $\Z$-boundary $\beta$.
\end{definition}

Observe that if $\beta$ is a $\Z$-boundary of $G$ such that $\beta(v) = 0$ for all $v \in V(G)$, a nowhere-zero flow in $\g$ is a nowhere-zero flow of $G$.

The next useful step in Thomassen's proof of Grötzsch's Theorem is to restrict the lists of some of the vertices, in a way that facilitates the induction. In particular, it is often useful to precolour some vertices\textemdash that is, prescribe their colour before attempting to colour the rest of the graph. Analogously, we introduce \emph{preflows}.  For an edge $uv$, we use $(u,v)$ to denote an arc oriented from $u$ to $v$.

\begin{definition}
A \emph{preflow} $\psi$ in $\g =(G,\beta)$ is a partial orientation of $G$, i.e., a directed graph with the same vertex set as $G$, and such that $uv\in E(G)$ for every $(u,v)\in E(\psi)$.  For a vertex $z$ of $G$, we say that $\psi$ is a \emph{preflow around $z$} if all arcs in $\psi$ are incident to $z$, and further $\deg^{+}(z) - \deg^{-}(z) \equiv \beta(z) \pmod 3$. We say that a preflow around $z$ \emph{extends to a nowhere-zero flow} of $\g$ if there exists a nowhere-zero flow of $\g$ for which the orientation of edges incident to $z$ agrees with the preflow around $z$.
\end{definition}

As we will mostly be interested in preflows around a specific vertex, we make the following definition.

\begin{definition}
Given a $\Z$-bordered graph $\g$, and a vertex $z \in V(G)$, a \emph{canvas} is
a pair $(\g,z)$ and we will refer to $z$ as the \emph{tip} of the canvas. A
\emph{tip preflow} of $(\g,z)$ is a preflow around~$z$.
\end{definition}

The final technical definition (first stated in \cite{weak3flow}) is motivated
by the main tool used for induction. For colouring, a typical reduction
involves identifying two non-adjacent vertices and then applying induction, and
obtaining from the resulting colouring a colouring of the original graph. For
flows, the corresponding tool would be to \textit{split off} an edge: That is,
given vertices $u,v,w$ such that $uv$ and $vw$ are edges, to delete $uv$ and
$vw$, and replace them with an edge $uw$ (possibly creating a multiple edge in
the process, but never a loop: If $u=w$ then we simply delete the pair of
parallel edges). When we split off edges, we may reduce the edge-connectivity
of our graph\textemdash which is not ideal if we are aiming to prove the
$3$-flow conjecture.

Thus it would be nice to replace the edge-connectivity
condition with a simpler condition that more easily facilitates the splitting
off operation. The idea is as follows. We have a $\Z$-bordered graph $\g$, and
we would like to split off the edges incident to a vertex $v$, creating a new graph $G'$, and we want to create a new
$\Z$-boundary $\beta'$ of $G'-v$  such that if we find a nowhere-zero flow of
the $\Z$-bordered graph $(G'-v,\beta')$, we can easily translate this into a
nowhere-zero flow of $\g$. For example, if $\beta(v) = 0$ and further $v$ has
even degree, the approach would be to pair up the neighbours of $v$ and then
split off all of the edges incident to $v$ in accordance to their pairs, and
let $\beta' = \beta|_{G'-v}$. It is easy to see that if $(G'-v,\beta')$ has a
nowhere-zero flow, we can lift it to a nowhere-zero flow of $G$ simply by
taking one of the split off edges, which say has been oriented from $u$ to $w$,
and then orienting the edges $uv,vw$ from $u$ to $v$, and from $v$ to $w$. Of
course, this only works in  the specific case described. If the degree of $v$
was odd and $\beta(v) =0$, then we cannot simply pair up the neighbours of $v$
and split off, as we will always have one edge left over. Thus to ensure that
the process works, instead of splitting off all edges incident to $v$ one
should leave behind three edges, say incident to $u,w,q$, and to compensate for
this modify the boundary of all of $u,w,q$ by $\pm 1$. To extend the flow to
$G$, orient the edges $uv,vw,vq$ either all towards $v$ or all away from $v$
depending on how the boundaries were modified. The following definition
encapsulates the minimum number of edges that would need to be left over in the
above procedure for the extension to work, extended from single vertices to
arbitrary sets of vertices. 

\begin{definition}
Let $\g$ be a $\Z$-bordered graph. For a set $A \subseteq V(G)$, let the degree of $A$, denoted $\deg(A)$, be the number of edges with exactly one endpoint in $A$, and let $\beta(A) = \sum_{v \in A} \beta(v) \pmod 3$. Let $\tau(A)$ be defined by the following chart: 
\begin{center}
\begin{tabular}{c|ccc}
\backslashbox{degree}{$\beta$} & $0$ & $1$ & $-1$\\
\hline
even& $\{0\}$ & $\{-2\}$ & $\{2\}$\\
odd& $\{-3,3\}$ & $\{1\}$ & $\{-1\}$
\end{tabular} \\
\end{center}
We write  $\tau(A)>0$ or $\tau(A)<0$ to mean that $\tau(A)$ contains only a single positive or negative element. Slightly abusing notation, we let $|\tau(A)|$ be defined by the following chart:
\begin{center}
\begin{tabular}{c|cc}
\backslashbox{degree}{$\beta$} & $0$ & $\pm 1$\\
\hline
even& $0$ & $2$\\
odd& $3$ & $1$
\end{tabular}
\end{center}
\end{definition}
To avoid cumbersome notation, if $A = \{v\}$, then we write $\tau(\{v\})$ as $\tau(v)$.
With this, we can state a strengthening of Theorem \ref{6edgeconnected} proven in \cite{ltwz}.

\begin{theorem}[L. M. Lovász et al.~\cite{ltwz}]
\label{thm:lovaszrealtheorem}
Let $(\g,z)$ be a canvas such that $\deg(z) \leq 4 + |\tau(z)|$. If for every non-empty $A \subsetneq V(G) \setminus \{z\}$ we have $\deg(A) \geq 4 + |\tau(A)|$, then every tip preflow extends to a nowhere-zero flow in $\g$. 
\end{theorem}

One of our main theorems is an extension of this theorem where we allow the degree of $z$ to be arbitrarily large. In general this may result in graphs for which a tip preflow will not extend; we will show that if there is another vertex with large degree relative to the degree of $z$, then every tip preflow extends. To make the statement less cumbersome, we introduce a notion of flow-criticality for canvases in the natural fashion.

\begin{definition}
Let $\g=(G,\beta)$ be a $\Z$-bordered graph. For a partition $\PP$ of $V(G)$, we
define the $\Z$-boundary $\beta / \PP:V(G/\PP)\to\Z$ for the graph $G/\PP$ by
letting 
$$(\beta / \PP)(p) = \sum_{ v \in P} \beta(v) \pmod 3$$
for every part $P\in\PP$ and the corresponding vertex $p$ obtained by contracting $\PP$.
We define $\g / \PP$ to be the $\Z$-bordered graph $(G/ \PP, \beta / \PP)$.
As before, we call $\g / \PP$ a \emph{contraction} of $\g$.

Suppose now that $(\g,z)$ is a canvas.
We say that $\PP$ is \emph{tip-respecting} if $\{z\}$ is a part of $\PP$.
In this case, we abuse notation and use $z$ to refer to the vertex of $\g/\PP$ obtained by
contracting the part $\{z\}$ of $\PP$; thus, $(\g/\PP, z)$ is a canvas.
Moreover, we can view any tip preflow in $(\g,z)$ as a tip preflow in $(\g / \PP,z)$.

For a tip preflow $\psi$, we say that the canvas $(\g,z)$ is \emph{$\psi$-critical} if $\psi$ does
not extend to a nowhere-zero flow in $(\g,z)$, but for every non-trivial
tip-respecting partition $\PP$, $\psi$ extends to a nowhere-zero flow
in $(\g / \PP,z)$. More generally, we say a canvas $(\g,z)$ is
\emph{flow-critical} if for every non-trivial tip-respecting partition
$\PP$, there exists a tip preflow that extends to a nowhere-zero flow
in $(\g / \PP,z)$ but does not extend in $(\g,z)$.
\end{definition}

We pause to make an important remark on our definition of flow-criticality.
First observe that unlike in connected-flow-criticality, here we make no
assumption that the parts induce a connected graph. While this connectivity
assumption is natural from the perspective of flow-colouring duality, there
seems to be no real reason to include it in the definition, and thus we do not.
Further, trivially a canvas with at most two vertices is flow-critical, as
there are no non-trivial tip-respecting contractions. We call a canvas with two
vertices \textit{trivial}.

As it will come up frequently, we also define the following important notion.

\begin{definition}
 Let $(\g,z)$ be a canvas. We say $(\g,z)$  is \emph{tame} if for all vertices $v \in V(G) \setminus \{z\}$, we have $\deg(v) \geq 4 + |\tau(v)|$. 
\end{definition}

Now we can finally state one of our main theorems, which allows us to extend Theorem \ref{thm:lovaszrealtheorem} to the situation where we have a vertex of large degree with a preflow.

\begin{theorem}
\label{thm-deg}
If $(\g,z)$ is a non-trivial flow-critical tame canvas, then every vertex other than $z$ has degree at most $\deg(z) -2$.  
\end{theorem}

Of course, one can rephrase this to make the connection to Theorem \ref{thm:lovaszrealtheorem} more obvious.

\begin{theorem}\label{cor-tall}
Let $(\g,z)$ be a canvas such that $\deg(A)\ge 4+|\tau(A)|$ for every non-empty $A\subsetneq V(\g)\setminus\{z\}$. If there exists a vertex $x\in V(\g)\setminus \{z\}$ such that for every $X\subseteq V(\g)\setminus\{z\}$ containing $x$, we have $\deg(X)>\deg(z)-2$,  then every tip preflow extends to a nowhere-zero flow in $\g$.
\end{theorem} 

We show that Theorem \ref{cor-tall} is indeed a rephrasing of Theorem \ref{thm-deg} at the end of the preliminary section. For the purposes of induction, we will actually prove a more complicated statement which requires a definition.

\begin{definition}
 An \emph{easel} is a tuple $(\g,z,x,\psi)$, where
$(\g,z)$ is a canvas, $x$ is a vertex of $\g$ distinct from $z$, and $\psi$ is a tip preflow.
The easel is \emph{tame} if $\deg(v)\ge 4+|\tau(v)|$ for every $v\in V(\g)\setminus\{x,z\}$.
The easel is \emph{tall} if $\deg(z)\le \deg(x)+2$,
and if $\deg(z)=\deg(x)+2$, then there additionally exists an edge in $\psi$ not incident with $x$ directed towards
$z$ and another such edge directed away from $z$.
It is \emph{critical} if the canvas $(\g,z)$ is $\psi$-critical.
\end{definition}

The more complicated theorem can now be stated in the following manner.

\begin{theorem}
\label{thm:tall}
    Tall tame easels are not critical.  
\end{theorem}

To prove Theorem \ref{thm:introweakerversion}, we will actually need more
control over flow-critical graphs than what is given in Theorem \ref{thm-deg}.
In particular, we will want to understand the possible degree sequences of flow-critical tame canvases. As one can generate
flow-critical tame canvases with arbitrarily many vertices of degree four,  we
define the following.

\begin{definition}
Let $(\g,z)$ be a canvas. The \emph{census} $C(\g,z)$ of $(\g,z)$ is the multiset $\{\deg(v) : v \in V(G) \setminus \{z\}, \deg(v) \neq 4\}$. 
\end{definition}

This definition might look strange, however as with most of the work in this paper, the definition can be motivated from results on colouring graphs on surfaces. When $3$-colouring graphs on surfaces, it is known that in a $4$-critical triangle-free graph embedded in a surface with no non-contractible $4$-cycles, there are only a bounded number of faces of length bigger than four \cite{trfree2,trfree3}; however, there can be arbitrarily many $4$-faces. 

Rephrased in terms of censuses, Theorem \ref{thm-deg} says the following.

\begin{corollary}\label{cor-censmax}
If $(\g,z)$ is a non-trivial flow-critical tame canvas, then $$\max_{d \in C(\g,z)} d \le \deg(z)-2.$$  In particular, if $\deg(z)=6$, then $C(\g,z)=\emptyset$.
\end{corollary}

We prove the following strengthening of Theorem \ref{thm-deg} and Corollary \ref{cor-censmax}.

\begin{restatable}{theorem}{thmdegbetter}\label{thm-degbetter}
Let $(\g,z)$ be a non-trivial flow-critical tame canvas.
If $\deg(z)\ge 7$ and $\g$ contains a vertex of degree $\deg(z)-2$, then $C(\g,z)=\{\deg(z)-2\}$.
\end{restatable}
Remember, $C(\g,z)$ is a multiset, so the above theorem says that in the case that $\g$ contains a vertex of degree $\deg(z) -2$, only that vertex and $z$ have degree bigger than $4$. 

Together with Corollary~\ref{cor-censmax}, this has the following consequence.
\begin{corollary}\label{cor-degbetter}
Let $(\g,z)$ be a non-trivial flow-critical tame canvas.
If $\deg(z)=7$, then $C(\g,z)=\{5\}$.
\end{corollary}

The censuses from Corollaries~\ref{cor-censmax} and \ref{cor-degbetter} might seem somewhat familiar:
They resemble the multisets of lengths of internal faces (ignoring faces of length four)
of triangle-free plane graphs with outer face of length $6$ or $7$ and minimal subject to the
property that some precoloring of their outer face does not extend to a 3-coloring of the whole
graph; see~\cite{col8cyc}.  It might seem this is because of the flow-coloring duality; however,
upon closer inspection, it is probably somewhat of a coincidence:  The condition of tameness
implies that the boundary of vertices of degree five is non-zero, and thus the dual to Corollary~\ref{cor-degbetter}
does not say anything about proper coloring of plane graphs with 5-faces.  Moreover, the
two concepts diverge at length $8$, where in the $3$-coloring case, there are only two 5-faces possible,
while in the flow case, there can be four vertices of degree five (for the canvas consisting of $K_4$
and the tip joined to its vertices by double edges; see Figure~\ref{fig:tightnessexample1}).

The flow-critical tame canvases with the censuses described in Theorem~\ref{thm-degbetter} actually exist:
Suppose $(\g,z)$ is any tame canvas containing a vertex $x$ of degree $\deg(z)-2$ not adjacent to $z$ with all other vertices having degree four
(and boundary $0$, by tameness), and suppose $\psi$ is the tip preflow orienting all edges away from $z$.
Then $\psi$ does not extend to a nowhere-zero flow.  Indeed, if it extended to  a nowhere-zero flow,
then for any vertex $v$ of degree four, we would have $\deg^+(v)=\deg^-(v)=2$, and thus 
\begin{align*}
    0 & =\sum_{u\in V(\g)} \deg^+ (u)-\deg^-(u)\\ 
    & = \deg^+ (z)-\deg^- (z)+\deg^+ (x)-\deg^-(x)
\end{align*}

However, the choice of $\psi$ implies $\deg^+(z)-\deg^-(z)=\deg(z) >\deg(x) \ge |\deg^+(x)-\deg^-(x)|$,
which is a contradiction.  If for each vertex $v\in V(\g)\setminus\{z\}$, the only edge cut of size at
most $\deg(v)$ separating $v$ from $z$ consists of the edges incident with $v$, it is easy to conclude
using Theorem~\ref{thm-degbetter} that $(\g,z)$ is actually $\psi$-critical.

\begin{figure}
\begin{center}
\begin{tikzpicture} 
    \node[blackvertexv2] at (0,0) (v1) {};
    \node[blackvertexv2] at (1.5,0) (v2) {};
    \node[blackvertexv2] at (1.5,1.5) (v3) {};
    \node[blackvertexv2] at (0,1.5) (v4) {};
    \node[blackvertexv2] at (-2.5,.75) (z) [label= left:$z$] {};

    \draw[thick,black] (v1)--(v2)--(v3)--(v4)--(v1);
    \draw[thick,black] (v1)--(v3);
    \draw[thick,black] (v2)--(v4);
    \draw[black, bend left = 15, thick,postaction={decoration={markings,mark=at position 0.6 with {\arrow{>}}},decorate}] (z) to (v1);
    
    \draw[thick,black, bend right = 15, postaction={decoration={markings,mark=at position 0.6 with {\arrow{>}}},decorate}] (z) to (v1);
    
    \draw[thick,black, bend left = 15, postaction={decoration={markings,mark=at position 0.6 with {\arrow{<}}},decorate}] (z) to (v4);
    
    \draw[thick,black, bend right =15, postaction={decoration={markings,mark=at position 0.6 with {\arrow{<}}},decorate}] (z) to (v4);
    
    \draw[thick,black, bend left = 35,postaction={decoration={markings,mark=at position 0.6 with {\arrow{<}}},decorate}] (z) to (v3);
    
    \draw[thick,black, bend left = 50, postaction={decoration={markings,mark=at position 0.6 with {\arrow{<}}},decorate}] (z) to (v3);
    
    \draw[thick,black, bend right = 35,postaction={decoration={markings,mark=at position 0.6 with {\arrow{>}}},decorate}] (z) to (v2);
    
    \draw[thick,black, bend right = 50,postaction={decoration={markings,mark=at position 0.6 with {\arrow{>}}},decorate}] (z) to (v2);
    \begin{scope}[xshift = 5cm]
     \node[blackvertexv2] at (0,.75) (v1) {};
    \node[blackvertexv2] at (1.5,0) (v2) {};
    \node[blackvertexv2] at (1.5,1.5) (v3) {};
    \node[blackvertexv2] at (-2.5,.75) (z) [label= left:$z$] {};
    \draw[thick,black, bend left = 35,postaction={decoration={markings,mark=at position 0.6 with {\arrow{<}}},decorate}] (z) to (v3);
    
    \draw[thick,black, bend left = 50, postaction={decoration={markings,mark=at position 0.6 with {\arrow{<}}},decorate}] (z) to (v3);
    \draw[thick,black, bend right = 35,postaction={decoration={markings,mark=at position 0.6 with {\arrow{>}}},decorate}] (z) to (v2);
    
    \draw[thick,black, bend right = 50,postaction={decoration={markings,mark=at position 0.6 with {\arrow{>}}},decorate}] (z) to (v2);
    \draw[thick,black, bend right = 15,postaction={decoration={markings,mark=at position 0.6 with {\arrow{<}}},decorate}] (z) to (v1);
    
    \draw[thick,black, bend right = 30,postaction={decoration={markings,mark=at position 0.6 with {\arrow{<}}},decorate}] (z) to (v1);
    
    \draw[thick,black, bend left = 15,postaction={decoration={markings,mark=at position 0.6 with {\arrow{>}}},decorate}] (z) to (v1);

     \draw[thick,black, bend left = 30,postaction={decoration={markings,mark=at position 0.6 with {\arrow{>}}},decorate}] (z) to (v1);
     \draw[thick,black, postaction={decoration={markings,mark=at position 0.6 with {\arrow{<}}},decorate} ] (v1) to (v2);

     \draw[thick,black, postaction={decoration={markings,mark=at position 0.6 with {\arrow{>}}},decorate} ] (v2) to (v3);

      \draw[thick,black, bend right = 15, postaction={decoration={markings,mark=at position 0.5 with {\arrow{<}}},decorate} ] (v1) to (v3);

       \draw[thick,black, bend right = 15, postaction={decoration={markings,mark=at position 0.6 with {\arrow{>}}},decorate} ] (v1) to (v2);

     \draw[thick,black, postaction={decoration={markings,mark=at position 0.6 with {\arrow{<}}},decorate} ] (v3) to (v1);

    \end{scope}
    \begin{scope}[xshift = 5cm, yshift = 2.75cm]
   \node[blackvertexv2] at (0,1) (v1) {};
    \node[blackvertexv2] at (2,1) (v3) {};
    \node[blackvertexv2] at (-2.5,1) (z) [label= left:$z$] {};
    \draw[thick,black, bend left = 35,postaction={decoration={markings,mark=at position 0.6 with {\arrow{>}}},decorate}] (z) to (v3);
    
    \draw[thick,black, bend left = 50, postaction={decoration={markings,mark=at position 0.6 with {\arrow{>}}},decorate}] (z) to (v3);
    \draw[thick,black, bend right = 35,postaction={decoration={markings,mark=at position 0.6 with {\arrow{<}}},decorate}] (z) to (v3);
    
    \draw[thick,black, bend right = 50,postaction={decoration={markings,mark=at position 0.6 with {\arrow{<}}},decorate}] (z) to (v3);
    \draw[thick,black, bend right = 15,postaction={decoration={markings,mark=at position 0.6 with {\arrow{<}}},decorate}] (z) to (v1);
    
    \draw[thick,black, bend right = 30,postaction={decoration={markings,mark=at position 0.6 with {\arrow{<}}},decorate}] (z) to (v1);
    
    \draw[thick,black, bend left = 15,postaction={decoration={markings,mark=at position 0.6 with {\arrow{>}}},decorate}] (z) to (v1);

     \draw[thick,black, bend left = 30,postaction={decoration={markings,mark=at position 0.6 with {\arrow{>}}},decorate}] (z) to (v1);
     \draw[thick,black, bend left = 10, postaction={decoration={markings,mark=at position 0.6 with {\arrow{>}}},decorate} ] (v3) to (v1);

     \draw[thick,black, bend right = 10, postaction={decoration={markings,mark=at position 0.6 with {\arrow{<}}},decorate} ] (v3) to (v1);

     \draw[thick,black, bend left = 30, postaction={decoration={markings,mark=at position 0.6 with {\arrow{<}}},decorate} ] (v3) to (v1);

     \draw[thick,black, bend right = 30, postaction={decoration={markings,mark=at position 0.6 with {\arrow{>}}},decorate} ] (v3) to (v1);

    \end{scope}
     \begin{scope}[xshift = 6cm, yshift = -2.75cm]
   \node[blackvertexv2] at (0,1) (v1) {};
    \node[blackvertexv2] at (-2.5,1) (z) [label= left:$z$] {};
    \draw[thick,black, bend left = 75,postaction={decoration={markings,mark=at position 0.5 with {\arrow{>}}},decorate}] (z) to (v1);
    
    \draw[thick,black, bend left = 50, postaction={decoration={markings,mark=at position 0.6 with {\arrow{>}}},decorate}] (z) to (v1);

      \draw[thick,black, bend left = 15,postaction={decoration={markings,mark=at position 0.6 with {\arrow{<}}},decorate}] (z) to (v1);

     \draw[thick,black, bend left = 30,postaction={decoration={markings,mark=at position 0.5 with {\arrow{<}}},decorate}] (z) to (v1);

    \draw[thick,black, bend right = 75,postaction={decoration={markings,mark=at position 0.6 with {\arrow{>}}},decorate}] (z) to (v1);
    
    \draw[thick,black, bend right = 50,postaction={decoration={markings,mark=at position 0.5 with {\arrow{<}}},decorate}] (z) to (v1);
   
    \draw[thick,black, bend right = 15,postaction={decoration={markings,mark=at position 0.5 with {\arrow{<}}},decorate}] (z) to (v1);
    
    \draw[thick,black, bend right = 30,postaction={decoration={markings,mark=at position 0.6 with {\arrow{>}}},decorate}] (z) to (v1);

    \end{scope}
\end{tikzpicture}
\caption{On the left, we have a flow-critical graph $G$ where $z$ has boundary zero and the vertices where $z$ has in arcs to have boundary $2$, and the remaining have boundary $1$. On the right, we have the three tip-respecting contractions and a nowhere-zero $3$-flow that does not extend to $G$.}\label{fig:tightnessexample1}
\end{center}
\end{figure}
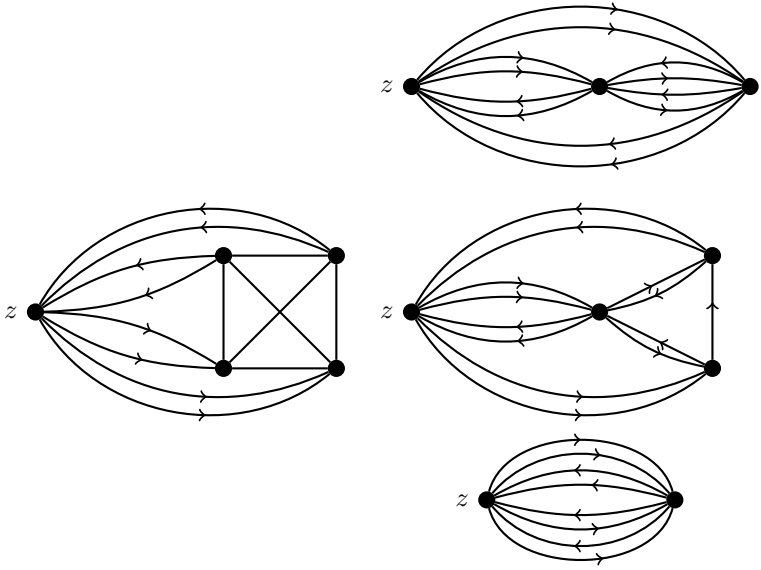
In light of Theorem~\ref{thm-degbetter}, we make the following two conjectures.

\begin{restatable}{conjecture}{conjfewlarge}\label{conj:fewlarge}
If $(\g,z)$ is a tame flow-critical canvas, then
$$\sum_{n \in C(\g,z)} (n-6) \leq \deg(z) -8.$$
\end{restatable}

\begin{restatable}{conjecture}{conjcensusplus}
\label{conj-censusplus}
There exists a function $f: \mathbb{N} \rightarrow \mathbb{N}$ such that if $(\g,z)$ is a tame flow-critical canvas, then
$$\sum_{n \in C(\g,z)} n \leq f(\deg(z)).$$
\end{restatable}

Let us remark that Conjecture~\ref{conj:fewlarge} implies a weaker version of Conjecture~\ref{conj:densitycrit1}, namely
that every flow-critical graph $G$ satisfies $|E(G)| \leq 3|V(G)|-4$; see Observation~\ref{obs-to-tame} for details.

To prove Theorem \ref{thm-degbetter}, we will need a theorem which allows us to
generate flow-critical tame canvases $(\g,z)$ such that $\deg(z) =k$. The
statement of this theorem is quite complicated, and thus we will defer it to
later sections. However, it has the following algorithmic consequence.

\begin{theorem}
\label{informalgenerationtheorem1}
For any positive integer $k$, there exists an algorithm that, given a positive integer $n$, 
generates all non-trivial flow-critical tame canvases $(\g,z)$ with $\deg(z) \leq k$ and $|V(\g)|\le n$,
using only a polynomial number (in $n$ and the number of such canvases) nowhere-zero flow existence tests
and with polynomial time complexity if the time complexity of these tests is excluded.
\end{theorem}

The more detailed version of the statement is Theorem \ref{thm-gen}. In our
proof of this theorem, we will be able to describe a set of four operations
which allow us to generate all possible flow-critical tame canvases.
Unfortunately, even though this looks like the type of theorem that would allow
us to prove Theorem \ref{thm-degbetter}, we actually will need another
generation theorem on easels. This theorem is even more technical than the
canvas generation theorem, and so again we will defer the statement until
later.  Here, we state just its natural algorithmic consequence.

\begin{theorem}
\label{informalgenerationtheorem2}
For any positive integer $k$ and non-negative integer $r$, there exists an algorithm that,
given a positive integer $n$, generates all critical tame easels $(\g,z,x,\psi)$ such that $\deg(z) \leq k$, $x \neq z$
is a vertex of degree at least $k -2 -r$, and $|V(\g)|\le n$,
using only a polynomial number (in $n$ and the number of such easels) nowhere-zero flow existence tests
and with polynomial time complexity if the time complexity of these tests is excluded.
\end{theorem}

The more detailed version of the statement is Theorem \ref{thm-genladeg}. In
our proof of this theorem, we have to include an additional operation over the
canvas generation theorem, leading to six operations, but we gain control over
the structure of the easel when certain operations occur. This theorem will be
strong enough to prove Theorem \ref{thm-degbetter}.

Now we sketch the proofs of our theorems. The proofs of Theorems \ref{thm-deg},
\ref{thm:tall}, \ref{informalgenerationtheorem1},
\ref{informalgenerationtheorem2} all follow the same general approach, which is
to modify the proof of Theorem \ref{thm:lovaszrealtheorem}. We briefly outline
the proof of Theorem \ref{thm:lovaszrealtheorem}. They first set up a very
particular partial order on the set of hypothetical counterexamples, and choose
a minimal counterexample with respect to this partial order. With this, they
first argue via a standard induction argument that there are no non-trivial tight
edge cuts, i.e, sets $A \subsetneq V(G) \setminus \{z\}$ with $\deg(A)=4 + |\tau(A)|$
of size at least two.  With the slack obtained in
the previous step, they now can argue, by splitting off edges and applying
induction, that there are no vertices of even degree with boundary zero. Similarly,
they can deduce that there are no ``mixed'' edges not incident to $z$\textemdash
that is, there are no edges $uv$ such that $\tau(u)$ contains a non-positive
value and $\tau(v)$ contains a non-negative value. With that in hand, they can
argue that a minimal counterexample must orient all edges away from $z$ or
towards $z$, and further that the degree of $z$ is not too large. To finish, they
take an arc incident to $z$, and replace it with two parallel arcs in the
opposite direction. They argue that the resulting graph and preflow is not a
counterexample by minimality, and use this to deduce the theorem.

To avoid repetition, let us start by sketching the ideas of Theorem
\ref{thm:tall}. We aim to follow the exact same outline as above, but now we
run into problems when we try to split off edges. Possibly we can no longer
apply induction because the resulting easel would not be tall. To remedy this,
we show that in a minimal counterexample $(\g,z,x,\psi)$, there are no small
cuts which separate $x$ from $z$. This is similar to the step in Theorem
\ref{thm:lovaszrealtheorem} where they argue that there is no non-trivial tight edge cut.
With this in hand, we now can
proceed as in Theorem \ref{thm:lovaszrealtheorem}, albeit with more technical
complications. We argue that there are no ``mixed'' edges, except possibly incident with
$z$ or $x$. Then we argue that $z$ has small degree, and all edges incident to
$z$ are oriented either away from or towards $z$, and $z$ is not adjacent to
$x$. The end is the same as in the proof of Theorem
\ref{thm:lovaszrealtheorem}: We take an arc incident to $z$, reverse it and add
two parallel arcs. Finally, we argue that the resulting easel is not a
counterexample by minimality, and deduce the theorem from this.

Theorems \ref{informalgenerationtheorem1} and \ref{informalgenerationtheorem2} follow a similar outline, except that at certain points, small technical differences arise causing additional complications. 

For the proof of Theorem \ref{thm-degbetter}, we consider a minimal
counterexample, and then we apply the easel generation theorem. We then need to
analyze the possible outcomes and show that for all possible operations, the
operations would not create a counterexample.

Finally to prove Theorem \ref{thm:introweakerversion} we first give an
elementary reduction showing that we can take any flow-critical graph on at
least three vertices and construct a tame flow-critical canvas. With this, one
simply needs to apply Theorem \ref{thm-degbetter} to the reduction to prove
Theorem \ref{thm:introweakerversion}. 

\subsection{Outline of paper}\label{subsec:paperoutline}
The paper is structured as follows: Section \ref{sec:preliminaries} contains
basic properties and definitions which will be needed throughout the rest of
the paper. In particular, Subsection \ref{subsec:flow-crit}  establishes basic
behaviour of flow-critical graphs and canvases, and further motivates our
precise definition of flow-criticality. Subsection \ref{subsec:tamecanvases}
covers tame critical canvases: Again, we establish basic properties, and in
particular prove several connectivity lemmas including a key lemma (Lemma
\ref{lemma-cut}) which we will use repeatedly in the proof of Theorem
\ref{thm:tall}.

Section \ref{sec:thmtall-proof} contains the proof of Theorem
\ref{thm:tall}, which immediately implies Theorem \ref{thm-deg}. The section
begins with a proof outline. Section \ref{sec:canvasgeneration} concerns
flow-critical canvas generation. Subsection \ref{subsec:canvasops} covers the
four operations for generating canvases.  Subsections \ref{subsec:canv-gen}
through \ref{subsec:z-homog} contain necessary lemmas for our canvas generation
theorem (Theorem \ref{thm-gen}), which is proved in Subsection
\ref{subsec:finish}: In particular, Subsection \ref{subsec:canv-gen} contains
results concerning cuts in critical canvases. In Subsection
\ref{subsec:z-homog}, we show that minimal counterexamples to Theorem
\ref{thm-gen} are \emph{$x$-homogeneous}, a notion which will be defined in
Section \ref{sec:preliminaries}; and finally, in Subsection \ref{subsec:finish}
we deduce that no minimum counterexamples exist. The structure of the proof is
overall very similar to that of the proof of Theorem \ref{thm:tall}. 

Finally,
Section \ref{sec:geneasels} contains results on easel generation, and the proof
of our easel generation theorem (Theorem \ref{thm-genladeg}). The general
structure is similar to that of Section \ref{sec:canvasgeneration} (and
therefore in turn very similar to the proof of Theorem \ref{thm:tall}). In
particular, in Subsection \ref{subsec:easelcuts} we characterize small cuts
around $x$ in a minimum counterexample $(\g,z,x,\psi)$. Subsection
\ref{subsec:xhomogeasel} shows that minimum counterexamples are
$x$-homogeneous; and, using the results established in the foregoing two
subsections, Subsection \ref{subsec:easelfinisher} contains the proof of
Theorem \ref{thm-genladeg}.

\section{Preliminaries}
\label{sec:preliminaries}

In this section, we introduce basic properties of flows and canvases that we use throughout the paper.

\subsection{Preliminaries on flow-criticality}\label{subsec:flow-crit}
This subsection begins with a discussion on the differences between and merits
of connected-flow-criticality and flow-criticality. In particular, we make two
easy but crucial observations: Contractions of flow-critical graphs are again
flow-critical (Observation \ref{obs-subcrit}), and in a canvas $(\g,z)$, all
multiple edges are incident to $z$ (Lemma \ref{lemma-simple}).  We first
address the difference between connected-flow-criticality and flow-criticality.
The following simple observation allows us to partially translate results
between the two settings.

\begin{observation}
\label{obs:connectedtoindependent}
    Every flow-critical $\Z$-bordered graph is connected-flow-critical. Furthermore, for every connected-flow-critical $\Z$-bordered graph $\g$, there exists a partition $\PP$ such that each part induces an edgeless subgraph and $\g / \PP$ is flow-critical.
\end{observation}

In particular, observe that any upper bound on the density of flow-critical
graphs applies to connected-flow-critical graphs. Next, we motivate our
slightly complicated definition of flow-criticality. Of course, the definition
is analogous to the dual definitions for colouring, and so this is motivation in and of
itself; but more importantly, the following two observations hold true with our
definition.

\begin{itemize}
\item If a canvas $(\g,z)$ is $\psi$-critical for a tip preflow $\psi$, then $(\g,z)$ is flow-critical.
\item Every canvas $(\g,z)$ has a tip-respecting flow-critical contraction $(\g/\PP,z)$ such that
every tip preflow extends to a nowhere-zero flow in $(\g,z)$ if and only if it does in $(\g/\PP,z)$.
\end{itemize}

Let us emphasize that a canvas $(\g,z)$ with at most two vertices is flow-critical simply because there are no non-trivial tip-preserving contractions. This is why we call these flow-critical canvases trivial. Thus to see point two, observe if the canvas $(\g,z)$ has a flow for every tip preflow $\psi$, then letting $\PP$ be the partition where all vertices are in the same part except $z$ gives the result. Otherwise, if $(\g,z)$ is not flow-critical itself, we find two vertices which are not $z$ whose identification to $(\g',z)$ does not have a tip preflow which extends in $(\g',z)$ but not $(\g,z)$. Repeating this procedure we obtain a partition $\PP$ such that $(\g / \PP)$ is flow-critical, and every tip preflow extends to a nowhere-zero flow in $(\g,z)$ if and only if it does in $(\g/\PP,z)$.

Let us make another important observation. Note that any tip preflow $\psi$
forms a nowhere-zero flow on such a trivial canvas by itself, and thus this canvas is not $\psi$-critical. Let us record this observation for later use. 
\begin{observation}\label{obs-ge3}
If a canvas $(\g,z)$ is $\psi$-critical for a tip preflow $\psi$, then $|V(\g)|\ge 3$.
\end{observation}

  One might consider excluding trivial canvases from the definition of flow-criticality; for example, this would simplify the following statement.
\begin{observation}\label{obs-critopsi}
If $(\g,z)$ is a non-trivial flow-critical canvas, then for some tip preflow $\psi$ there exists a $\psi$-critical tip-respecting contraction of $(\g,z)$.
\end{observation}
\begin{proof}
Since the canvas $(\g,z)$ has at least three vertices, it has a proper tip-respecting contraction, and thus
the definition of flow-criticality implies that there exists a tip preflow $\psi$
that does not extend to a nowhere-zero flow in $\g$.  Let $(\g',z)$ be a minimal tip-respecting contraction
of $(\g,z)$ such that $\psi$ does not extend to a nowhere-zero flow in $\g'$,
and observe that $(\g',z)$ is $\psi$-critical.
\end{proof}
On the other hand, declaring trivial canvases to be non-flow-critical would complicate the statement of the following important observation on restrictions of canvases.

\begin{definition}
Given a canvas $(\g,z)$ and a set $A\subseteq V(\g)\setminus \{z\}$, the \emph{restriction} $(\g,z)\restriction A$ of
$(\g,z)$ to $A$ is the canvas $(\g/B,b)$, where $B=V(\g)\setminus A$ and $b$ is the vertex resulting from the contraction of $B$.
\end{definition}
An intuitive way of thinking about the restriction is that $(\g,z)\restriction A$ is obtained from $(\g,z)$ by contracting everything except for $A$ to the tip vertex. 
\begin{observation}\label{obs-subcrit}
If a canvas $(\g,z)$ is flow-critical and $A$ is a subset of $V(\g)\setminus\{z\}$, then the canvas $(\g,z)\restriction A$ is flow-critical.
\end{observation}
\begin{proof}
Let $B=V(\g)\setminus A$ and let $(\g',b)=(\g,z)\restriction A$. Consider any non-trivial tip-respecting partition $\PP'$ of $V(\g')$.
To show that $(\g',b)$ is flow-critical, we need to find a tip preflow $\psi'$ extending to a nowhere-zero flow in $(\g'/\PP',b)$, but
not in $(\g,b)$.

Let $\PP$ be the partition of $V(\g)$ obtained from $\g$ by replacing $b$ by singleton parts consisting of the vertices of $B$;
since $\PP'$ is non-trivial, so is $\PP$.  Since $z\in B$, the partition $\PP$ is tip-respecting in $(\g,z)$.  
Since $(\g,z)$ is flow-critical, there exists a tip preflow $\psi$ such that $\psi$ extends to a nowhere-zero flow $\varphi$ in $\g/\PP$ but not in $\g$.
Let $\psi'$ be the restriction of $\varphi$ to the edges with exactly one end in $B$, which can be viewed as tip preflow in $(\g',b)$.
The restriction of $\varphi$ to the edges incident with at least one edge of $B$ shows that $\psi'$ extends to a nowhere-zero flow in $\g'/\PP'$.
On the other hand, $\psi'$ does not extend to a nowhere-zero flow $\varphi'$ in $\g'$, as otherwise we could combine $\varphi'$ with the
restriction of $\varphi$ to the edges not incident with $B$ to obtain a nowhere-zero flow in $\g$ extending $\psi$ and get a contradiction.
\end{proof}
Let us remark that another motivation for the technical definition of flow-criticality of canvases is so that Observation~\ref{obs-subcrit}
holds (it is not necessarily true that $(\g,z)\restriction A$ is $\psi'$-critical for any fixed tip preflow $\psi'$).

Now, let us make the following standard observation, which tells us that we do not need to check the flow conservation constraint at one of the vertices.

\begin{lemma}
    \label{obs-allbutone}
Let $\g=(G,\beta)$ be a $\Z$-bordered graph, let $y$ be a vertex of $\g$, and let $\vec{G}$ be an orientation of $G$ such that $\deg^+(v) -\deg^-(v) \equiv \beta(v)\pmod 3$ for all
$v\in V(\g)\setminus\{y\}$.  Then $\vec{G}$ is a nowhere-zero flow in $\g$.
\end{lemma}
\begin{proof}
Note that $\sum_{v\in V(\g)} (\deg^+(v) -\deg^-(v))=0$, since each edge is counted positively at its head and negatively at its tail in this sum.
Moreover, since $\beta$ is a $\Z$-boundary, $\sum_{v\in V(\g)} \beta(v)\equiv 0\pmod 3$.
Hence, we have

\begin{align*}
    \deg^+(y)-\deg^-(y)&=-\sum_{v\in V(\g)\setminus\{y\}} (\deg^+(v) -\deg^-(v)) \\&\equiv -\sum_{v\in V(\g)\setminus\{y\}}\beta(v)\equiv \beta(y)\pmod 3,
\end{align*}

and thus the flow conservation condition $\deg^+(v) -\deg^-(v) \equiv \beta(v)\pmod 3$ holds for $v=y$ as well.
\end{proof}

We end this subsection with the following simple observation.

\begin{lemma}\label{lemma-simple}
If $(\g,z)$ is a flow-critical canvas, then every edge of multiplicity more than one is incident with $z$.
\end{lemma}
\begin{proof}
Let $\g=(G,\beta)$.
Suppose for a contradiction that distinct vertices $u,v\in V(\g)\setminus\{z\}$ are joined by an edge of multiplicity greater than one.
Since the canvas $(\g,z)$ is flow-critical, there exists a tip preflow $\psi$ that extends to a nowhere-zero flow in $\g/\{u,v\}$ but not in $\g$.
The nowhere-zero flow in $\g/\{u,v\}$ can be viewed as a partial orientation of $\g$ extending $\psi$ and such that only the edges
between $u$ and $v$ are not oriented and $\deg^+(y)-\deg^-(y)\equiv \beta(y)\pmod 3$ holds for all $y\in V(\g)\setminus\{u,v\}$.
However, since there are at least two edges between $u$ and $v$, it is possible to orient them so that this condition
holds also for $y=u$, which by Lemma~\ref{obs-allbutone} shows that the resulting orientation is a nowhere-zero flow in $\g$ extending $\psi$. This is a contradiction.
\end{proof}

\subsection{Preliminaries on tame canvases}\label{subsec:tamecanvases}

In this subsection, we reformulate Theorem \ref{thm:lovaszrealtheorem} to show
that tame non-trivial flow-critical canvases have $\deg(z) \geq 6 + |\tau(z)|$.
We then use this to argue that tame canvases have the property that all tight
sets are single vertices\textemdash that is, if a set $A\subsetneq V(\g)\setminus\{z\}$ satisfies $\deg(A) =
4 + |\tau(A)|$ and $A$ does not contain $z$, then $A = \{v\}$ for some vertex
$v$. We will also prove a crucial lemma: Tall tame critical easels are
$2$-connected.

We start off by making some basic observations about boundary functions and $\tau$.
Note that if $\beta$ is a boundary function for a graph $G$ and $A \subseteq V(G)$, then we have $\beta(V(G)\setminus A)=-\beta(A)$, and thus $|\tau(V(G)\setminus A)|=|\tau(A)|$. The following observation is critical, and also immediate from the definition of $\tau$.

\begin{observation}\label{obs-tauprop}
Let $A$ and $B$ be subsets of vertices of a canvas.
\begin{itemize}
\item[(a)] $\deg(A)$ and $|\tau(A)|$ have the same parity.
\item[(b)] If $\deg(A) \not\equiv \deg(B) \pmod 2$ and $\beta(A)\neq \beta(B)$, then $\bigl||\tau(A)|-|\tau(B)|\bigr|=1$.
\item[(c)] If $\deg(A) \not \equiv \deg(B) \pmod 2$ and $\beta(A)=\beta(B)\neq 0$, then $\bigl||\tau(A)|-|\tau(B)|\bigr|=1$.
\end{itemize}
\end{observation}

With this, we can reformulate Theorem \ref{thm:lovaszrealtheorem} in the following manner.

\begin{corollary}\label{cor-mainlov}
If $(\g,z)$ is a tame non-trivial flow-critical canvas, then $\deg(z) \ge 6+|\tau(z)|$.
\end{corollary}
\begin{proof}
Suppose for a contradiction that $(\g,z)$ is a minimal counterexample, i.e., a tame non-trivial flow-critical canvas with the smallest number of vertices and such that $\deg(z) <6+|\tau(z)|$. Since $\deg(z)$ and $|\tau(z)|$ have the same parity, we have $\deg(z)\le 4+|\tau(z)|$.
Since the canvas $(\g,z)$ is flow-critical and non-trivial, there exists a tip-respecting preflow $\psi$ does not extend to a nowhere-zero flow in $\g$.
By Theorem~\ref{thm:lovaszrealtheorem}, there exists a non-empty set $A\subsetneq V(\g)\setminus\{z\}$ such that $\deg(A) < 4 + |\tau(A)|$.
Since $(\g,z)$ is tame, we have $|A|\ge 2$.  The canvas $(\g',b)=(\g,z)\restriction A$ is flow-critical by Observation~\ref{obs-subcrit}, and tame by inspection.
Since $\g'$ has fewer vertices than $\g$, $(\g',b)$ is not a counterexample to Corollary~\ref{cor-mainlov}, and thus $\deg(b) \ge 6+|\tau(b)|$.  However, $\deg(b)=\deg(A)< 4 + |\tau(A)|=4+|\tau(b)|$, which is a contradiction.
\end{proof}

Thus, the tameness condition extends from vertices to larger subsets, even in a stronger form.

\begin{lemma}\label{lemma-cut}
Let $(\g,z)$ be a flow-critical canvas and let $A\subseteq V(\g)\setminus\{z\}$ be a set of size at least two.
If $\deg(v)\ge 4+|\tau(v)|$ for every $v\in A$, then $\deg(A) \ge 6+|\tau(A)|$.
\end{lemma}
\begin{proof}
By Observation~\ref{obs-subcrit}, $(\g',b)=(\g,z)\restriction A$ is a flow-critical canvas.  Note that $(\g',b)$ is tame and has at least three vertices,
and thus by Corollary~\ref{cor-mainlov}, we have $\deg(A)= \deg(b) \ge 6+|\tau(b)|=6+|\tau(A)|$.  
\end{proof}

In particular, we have the following important consequence.

\begin{corollary}\label{cor-contrtame}
If $(\g,z)$ is a tame flow-critical canvas, then $\deg(A)\ge 6+|\tau(A)|$ holds for every set $A\subseteq V(\g)\setminus \{z\}$
of size at least two.  In particular, every tip-respecting contraction of $(\g,z)$ is tame.
\end{corollary}
\begin{proof}
The first claim follows by Lemma~\ref{lemma-cut}.
If $(\g',z)$ is a tip-respecting contraction of $(\g,z)$ and $v\in V(\g')$, then consider the set $A\subseteq V(\g)$ contracted
into $v$.  If $|A|\ge 2$, then $\deg(v)=\deg(A) \ge 6+|\tau(A)|=6+|\tau(v)|>4+|\tau(v)|$, as we have just proved.
If $A$ consists of a single vertex $u\in V(\g)$, then $\deg(v)=\deg(u)\ge 4+|\tau(u)|=4+|\tau(v)|$, since $(\g,z)$ is
tame.  Hence, $(\g',z)$ is also tame.
\end{proof}

We next prove several lemmas on the connectivity of critical canvases and easels.

\begin{lemma}\label{lemma-conn}
Let $(\g,z)$ be a canvas and $x\in V(\g)\setminus \{z\}$ be a vertex such that $\deg(v)\ge 4+|\tau(v)|$ for every $v\in V(\g)\setminus\{x,z\}$.
If $(\g,z)$ is non-trivial and flow-critical, then $\g$ is connected.
\end{lemma}
\begin{proof}
Suppose for a contradiction that $\g=(G,\beta)$ is not connected, and let $\comp$ be a component of $\g$ that does not contain $z$.
We claim that $\comp$ has a nowhere-zero flow.  This is trivially the case if $|V(\comp)|=1$, and thus suppose that $\comp$ has at least two vertices.
By Lemma~\ref{lemma-cut}, we have $\deg(A)\ge 4+|\tau(A)|$ for every non-empty $A\subseteq V(\comp)\setminus \{x\}$.
Since $\beta$ is a $\Z$-boundary, we have $\beta(V(\comp))\equiv 0\pmod 3$, and consequently $|\tau(A)|=|\tau(V(\comp)\setminus A)|$.
Since $\deg(A) =\deg(V(\comp)\setminus A)$, we conclude that $\deg(A')\ge 4+|\tau(A')|$ holds also for sets $A'\subseteq V(\comp)$ containing $x$
(in the case that $x\in V(\comp)$).  Let $\comp'$ be obtained from $\comp$ by adding an isolated vertex $z'$ with boundary value $0$.
By Theorem~\ref{thm:lovaszrealtheorem} applied to the canvas $(\comp',z')$, we conclude that $\comp$ indeed has a nowhere-zero flow.

Since $(\g,z)$ is non-trivial, flow-critical, and $\comp$ has a nowhere-zero flow, there must exist a vertex $v\in V(\g)\setminus V(\comp)$ distinct from $z$.
Since $(\g,z)$ is flow-critical, there exists a tip preflow $\psi$ that does not extend to a nowhere-zero flow in $\g$, but extends to a nowhere-zero flow $\varphi$ in the canvas $(\g',z)$ obtained from $(\g,z)$
by contracting $V(\comp)\cup\{v\}$.  However, the canvas $(\g',z)$ is isomorphic to $\g-V(\comp)$, and thus
$\varphi$ combines with a nowhere-zero flow in $\comp$ to a nowhere-zero flow in $\g$ extending $\psi$, which is a contradiction.
\end{proof}

\begin{lemma}\label{lemma-2con}
Suppose $(\g,z)$ is a $\psi$-critical canvas for a tip preflow $\psi$.  If $|V(\g)|\ge 4$, then $\g-z$ is 2-connected.
\end{lemma}
\begin{proof}
Let $\g=(G,\beta)$. If $\g-z$ is not 2-connected, then since $|V(\g)|\ge 4$, there exists a partition $\{A_1,A_2,\{z\},\{x\}\}$ of $V(\g)$ such that
$A_1$ and $A_2$ are non-empty and there are no edges between $A_1$ and $A_2$.  For $i\in\{1,2\}$, since $(\g,z)$ is $\psi$-critical,
$\psi$ extends to a nowhere-zero flow $\vec{G}_i$ in $\g/(\{x\}\cup A_{3-i})$.  Let $\vec{G}$ be the orientation of $G$ matching $\vec{G}_i$ on the
edges incident with $A_i$ for $i\in\{1,2\}$ and $\psi$ on the edges between $z$ and $x$.   Clearly, we have $\deg^+(v)-\deg^-(v)\equiv \beta(v)$ for every $v\in V(\g)\setminus\{x\}$,
 and by Lemma~\ref{obs-allbutone}, it follows that $\vec{G}$ is a nowhere-zero flow.  Since $\vec{G}$ extends $\psi$, this is a contradiction. 
\end{proof}

Let us remark that in Lemma~\ref{lemma-2con}, it is not sufficient to assume that $(\g,z)$ is flow-critical even to conclude that $\g-z$ is connected.
As an example, consider the canvas $(\g,z)$ consisting of a matching $u_1v_1$ and $u_2v_2$ and the tip $z$ joined to each of $u_1$, $v_1$, $u_2$, and $v_2$
by a double edge, with zero boundary (see Figure \ref{fig:connectivityfig}). It is easy to see that $(\g,z)$ is flow-critical.
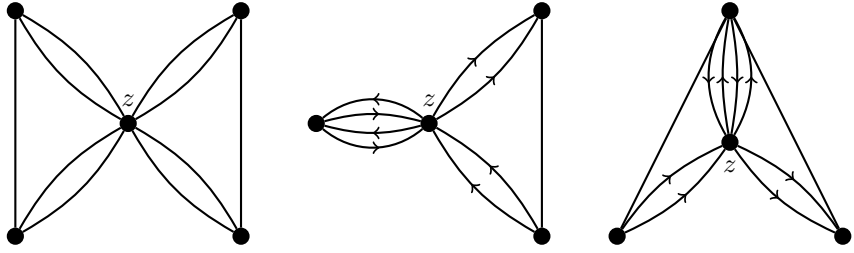
\begin{figure}
\begin{center}
\begin{tikzpicture}
\node[blackvertexv2] at (0,0) (v1) {};
\node[blackvertexv2] at (0,-3) (v2) {};
\node[blackvertexv2] at (1.5,-1.5) (z) [label =above:$z$] {};
\node[blackvertexv2] at (3,0) (v3) {};
\node[blackvertexv2] at (3,-3) (v4) {};
\draw[thick,black] (v1)--(v2);
\draw[thick,black] (v3)--(v4) {};
\draw[thick,black,bend left = 15] (z) to (v1);
\draw[thick,black,bend right =15] (z) to (v1);
\draw[thick,black,bend right = 15] (z) to (v2);
\draw[thick,black,bend left = 15] (z) to (v2);
\draw[thick,black,bend right =15] (z) to (v3);
\draw[thick,black,bend left =15] (z) to (v3);
\draw[thick,black,bend right =15] (z) to (v4);
\draw[thick,black,bend left =15] (z) to (v4);

\begin{scope}[xshift = 4cm]
\node[blackvertexv2] at (0,-1.5) (v1) {};
\node[blackvertexv2] at (1.5,-1.5) (z) [label =above:$z$] {};
\node[blackvertexv2] at (3,0) (v3) {};
\node[blackvertexv2] at (3,-3) (v4) {};

\draw[thick,black, bend left = 15, postaction={decoration={markings,mark=at position 0.5 with {\arrow{>}}},decorate}] (z) to (v1);
\draw[thick,black, bend left =40, postaction={decoration={markings,mark=at position 0.5 with {\arrow{<}}},decorate}] (z) to (v1);
\draw[thick,black, bend right =15, postaction={decoration={markings,mark=at position 0.5 with {\arrow{<}}},decorate}] (z) to (v1);
\draw[thick,black, bend right =40, postaction={decoration={markings,mark=at position 0.5 with {\arrow{>}}},decorate}] (z) to (v1);
\draw[thick,black, bend left = 15, postaction={decoration={markings,mark=at position 0.5 with {\arrow{>}}},decorate} ] (z) to (v3);
\draw[thick,black, bend right = 15,postaction={decoration={markings,mark=at position 0.5 with {\arrow{>}}},decorate}] (z) to (v3);
\draw[thick,black, bend left = 15,postaction={decoration={markings,mark=at position 0.5 with {\arrow{<}}},decorate}] (z) to (v4);
\draw[thick,black, bend right =15,postaction={decoration={markings,mark=at position 0.5 with {\arrow{<}}},decorate}] (z) to (v4);
\draw[thick,black] (v3)--(v4);
\end{scope}

\begin{scope}[xshift = 8cm]
\node[blackvertexv2] at (1.5,0) (v1) {};
\node[blackvertexv2] at (0,-3) (v2) {};
\node[blackvertexv2] at (1.5,-1.75) (z) [label =below:$z$] {};
\node[blackvertexv2] at (3,-3) (v4) {};

\draw[thick,black] (v1)--(v4);
\draw[thick,black] (v1)--(v2);

\draw[thick,black, bend left = 15, postaction={decoration={markings,mark=at position 0.5 with {\arrow{>}}},decorate}] (z) to (v4);
\draw[thick,black, bend right = 15, postaction={decoration={markings,mark=at position 0.5 with {\arrow{>}}},decorate}] (z) to (v4);

\draw[thick,black, bend right = 15, postaction={decoration={markings,mark=at position 0.5 with {\arrow{<}}},decorate}] (z) to (v2);
\draw[thick,black, bend left = 15, postaction={decoration={markings,mark=at position 0.5 with {\arrow{<}}},decorate}] (z) to (v2);

\draw[thick,black, bend left = 10, postaction={decoration={markings,mark=at position 0.5 with {\arrow{>}}},decorate}] (z) to (v1);

\draw[thick,black, bend left = 30, postaction={decoration={markings,mark=at position 0.5 with {\arrow{<}}},decorate}] (z) to (v1);

\draw[thick,black, bend right = 10, postaction={decoration={markings,mark=at position 0.5 with {\arrow{<}}},decorate}] (z) to (v1);

\draw[thick,black, bend right = 30, postaction={decoration={markings,mark=at position 0.5 with {\arrow{>}}},decorate}] (z) to (v1);
\end{scope}

\end{tikzpicture}
    \caption{A flow-critical graph $(G,\beta)$ where $G-z$ is not connected. Here $\beta(v) =0$ for all $v \in V(G)$.  Two tip-respecting contractions which do not extend to $G$ are shown.}
    \label{fig:connectivityfig}
\end{center}
\end{figure}
\begin{lemma}\label{lemma-at4}
Every tall tame critical easel has at least four vertices.
\end{lemma}
\begin{proof}
Let $(\g,z,x,\psi)$ be a tall tame critical easel with $\g=(G,\beta)$.  Since $\psi$ is not itself a nowhere-zero flow in $\g$,
we have $|V(\g)|\ge 3$.  Suppose for a contradiction that $V(\g)=\{v,z,x\}$.  Let $m$ be the number of edges between $v$ and $x$;
by Lemma~\ref{lemma-simple}, we have $m\le 1$.  Note that since the easel is tame, $\deg(v)\ge 4+|\tau(v)|$.  Since the easel is tall, we have
$\deg(x)\ge \deg(z)-2=(\deg(x)+\deg(v)-2m)-2$, and thus $\deg(v)\le 2m+2$.  We conclude that $m=1$, $\deg(v)=4$ and $\beta(v)=0$.
Since the easel is tall and $\deg(z)=\deg(x)+2$, the three edges between $v$ and $z$
are not all oriented by $\psi$ in the same direction (all towards $z$ or all away from $z$).
Hence, it is possible to direct the edge $vx$ so that $v$ has the same indegree and outdegree.
By Lemma~\ref{obs-allbutone}, this extends $\psi$ to a nowhere-zero flow in $\g$.
This is a contradiction, and thus $|V(\g)|\ge 4$.
\end{proof}

Lemmas~\ref{lemma-2con} and \ref{lemma-at4} have the following consequence.

\begin{corollary}\label{cor-2con}
If $(\g,z,x,\psi)$ is a tall tame critical easel, then $\g-z$ is 2-connected.
\end{corollary}

We will now use our newfound connectivity properties to make a simple observation on canvases which have no ``mixed'' edges. This requires some definitions.
\begin{definition}
\label{def:xhom}
We say that a vertex $v$ is \emph{in-friendly} if $\tau(v)$ contains a non-positive value, and \emph{out-friendly} if $\tau(v)$ contains a non-negative value.  Note that if $\beta(v)=0$, then $v$ is both in-friendly and out-friendly.
An edge $uv$ of a canvas $(\g, z)$ is \emph{mixed} if $u\neq z\neq v$, one of $u$ and $v$ is in-friendly, and the other one is out-friendly.
Let $x$ be a vertex of $\g$; we say that the canvas $(\g,z)$ is \emph{$x$-homogeneous} if all its mixed edges are incident with $x$.
Note that $x=z$ is possible, and in that case equivalently $(\g,z)$ has no mixed edges.
\end{definition}

\begin{observation}\label{obs-allplusminus}
Let $(\g,z)$ be an $x$-homogeneous $\psi$-critical canvas for a tip preflow $\psi$.
If $|V(\g)|\ge 4$, then either $\tau(v)>0$ for every $v\in V(\g)\setminus\{x,z\}$, or $\tau(v)<0$ for every $v\in V(\g)\setminus\{x,z\}$.
\end{observation}
\begin{proof}
Consider any edge $uv$ of $\g-\{x,z\}$.  Since this edge is not mixed, we either have $\tau(u),\tau(v)>0$, or $\tau(u),\tau(v)<0$.
Since $\g-\{x,z\}$ is connected by Lemma~\ref{lemma-2con}, the claim of the observation follows.
\end{proof}

Finally, we end this section with a simple proof that Theorem \ref{thm-deg} implies Theorem \ref{cor-tall}. 
\begin{proof}[Proof of Theorem \ref{cor-tall}]
Suppose for a contradiction that a tip preflow $\psi$ does not extend to a nowhere-zero flow in $\g$.
Then $(\g,z)$ has a tip-respecting $\psi$-critical contraction $(\g/\PP,z)$.  Note that $|V(\g/\PP)|\ge 3$ by Observation~\ref{obs-ge3}.
Since $\deg(A)\ge 4+|\tau(A)|$ for every non-empty $A\subsetneq V(\g)\setminus\{z\}$,
we conclude that the canvas $(\g/\PP,z)$ is tame.  Let $x'$ be the vertex of $\g/\PP$ corresponding to the part $X\in \PP$ containing $x$.
Since $\deg(X)>\deg(z)-2$ for every $X\subseteq V(\g)\setminus\{z\}$ containing $x$, we have $\deg(x')>\deg(z)-2$. This contradicts Theorem~\ref{thm-deg}.
\end{proof}

\section{Maximum degrees of tame critical canvases}\label{sec:thmtall-proof}
In this section, we prove Theorem \ref{thm:tall}, which implies Theorem
\ref{thm-deg} immediately. Let us outline again how the proof goes. We study the properties
of the (hypothetical) minimal counterexample, using a rather convoluted definition of
minimality. This is an artifact of the proof: One of the reductions would
fail if we used a more standard notion of minimality. Our first step is to define
this notion, and to argue that it indeed gives rise to a
strict partial order on canvases. This will also be relevant for future
sections.

With that out of the way, Step 2 of the proof is to argue that our minimal
counterexample has the right conditions to allow us to split off edges
effectively. That is, we argue that for sets of size at least two not
containing $z$, the tameness condition holds with slack; and similarly we argue
that there are no sets of size at least two that separate $x$ from $z$ in our
easel and have at most $\deg(z) +2$ edges. Fortunately, the tame part is
already done in Lemma \ref{lemma-cut}. Unfortunately, the
tall part is not nearly as easy: In fact this is the most delicate and
technical part of the proof, and requires a very careful choice of splitting
off edges for the induction to work.

Once we have this, Step 3 is to impose structure on the easel. In particular,
we show that we have no vertices of degree four in the canvas, and further,
that our easel is $x$-homogeneous (recall Definition \ref{def:xhom}). We can
actually push the structure even further: Not only is the easel
$x$-homogeneous, but in fact we can assert that $z$ has degree at most $\deg(x)
+1$, and all of the edges incident to $z$ are oriented either towards $z$ or
away from $z$.

With this, Step 4\textemdash the final step\textemdash is to perform
essentially the same reduction as in the proof of Theorem
\ref{thm:lovaszrealtheorem}. In particular, we take one of the arcs incident to
$z$,  and replace it with two arcs directed in the opposite direction. By how
we defined our partial order on the set of counterexamples, this new easel will
be strictly smaller in our partial ordering on the set of counterexamples,  and
hence not a counterexample; but of course a flow in the resulting easel is a
flow in the original, contradicting the fact that we had a counterexample at
all.

\subsection{Step 1: A partial order for canvases}
We start with an important definition. 
\begin{definition}
\label{def:ordering}
For a canvas $(\g, z)$, we define $o(\g,z)=|V(\g)|+|E(\g-z)|$. For triples $(\g_1,z_1,x_1)$ and $(\g_2,z_2,x_2)$ where $(\g_i,z_i)$ is a canvas and $x_i$ is a vertex of $\g_i$ for $i\in\{1,2\}$, we write $(\g_1,z_1,x_1)\prec (\g_2,z_2,x_2)$ if
\begin{enumerate}
\item $o(\g_1,z_1)<o(\g_2,z_2)$; or,
\item $o(\g_1,z_1)=o(\g_2,z_2)$, $(\g_1,z_1)$ is not $x_1$-homogeneous, and $(\g_2,z_2)$ is $x_2$-homogeneous; or,
\item $o(\g_1,z_1)=o(\g_2,z_2)$, $(\g_i,z_i)$ is $x_i$-homogeneous for $i\in\{1,2\}$, and $\deg(z_1)<\deg(z_2)$.
\end{enumerate}
\end{definition}
As this is a non-standard ordering, we provide the following straightforward check that $\prec$ indeed gives a strict partial ordering.

\begin{observation}\label{obs-precord}
The relation $\prec$ is a strict partial ordering with no infinite decreasing chains.
\end{observation}

\begin{proof}
We first argue that $\prec$ is transitive.  Suppose that 
$$(\g_1,z_1,x_1)\prec (\g_2,z_2,x_2)\prec(\g_3,z_3,x_3).$$
This implies that $o(\g_1,z_1)\le o(\g_2,z_2)\le o(\g_3,z_3)$.  If $o(\g_1,z_1)<o(\g_2,z_2)$ or $o(\g_2,z_2)<o(\g_3,z_3)$,
then $o(\g_1,z_1)<o(\g_3,z_3)$ and $(\g_1,z_1,x_1)\prec (\g_3,z_3,x_3)$.  Hence, suppose that $o(\g_1,z_1)=o(\g_2,z_2)=o(\g_3,z_3)$.
Since $(\g_1,z_1,x_1)\prec (\g_2,z_2,x_2)\prec(\g_3,z_3,x_3)$, the canvases $(\g_i,z_i)$ for $i\in\{2,3\}$ are $x_i$-homogeneous
and $\deg(z_2)<\deg(z_3)$.  If $(\g_1,z_1)$ is not $x_1$-homogeneous, then $(\g_1,z_1,x_1)\prec (\g_3,z_3,x_3)$ since $(\g_3,z_3)$ is $x_3$-homogeneous.
If $(\g_1,z_1)$ is $x_1$-homogeneous, then $\deg(z_1)<\deg(z_2)<\deg(z_3)$ and $(\g_1,z_1,x_1)\prec (\g_3,z_3,x_3)$.

Next, suppose for a contradiction that 
$$(\g_1,z_1,x_1)\succ (\g_2,z_2,x_2)\succ (\g_3,z_3,x_3) \succ \cdots$$
is an infinite decreasing chain.
Since $o(\g_i,z_i)$ is a non-negative integer and 
$$o(\g_1,z_1)\ge o(\g_2,z_2)\ge \cdots,$$ there exist $i_0$ and $m$ such that
$o(\g_i,z_i)=m$ for every $i\ge i_0$.  For every $i\ge i_0$, since $o(\g_{i+1},z_{i+1})=o(\g_i,z_i)$ and $(\g_{i+1},z_{i+1},x_{i+1})\prec (\g_i,z_i,x_i)$,
it follows that $(\g_{i+1},z_{i+1})$ is $x_i$-homogeneous.  Hence, 
$$(\g_{i_0},z_{i_0},x_{i_0})\succ (\g_{i_0+1},z_{i_0+1},x_{i_0+1})\succ \cdots$$ implies
$$\deg(z_{i_0}) > \deg(z_{i_0+1}) > \cdots,$$ which is a contradiction since degrees are non-negative integers.
\end{proof}

By a \textit{minimal} tall tame critical easel $(\g,z,x,\psi)$, we mean one for which the triple $(\g,z,x)$ is minimal in the $\prec$ ordering.  Note that since there are no decreasing infinite
chains in $\prec$, if Theorem~\ref{thm:tall} were false, then a minimal tall tame critical easel $(\g,z,x,\psi)$ would exist. 

\subsection{Step 2: There are no small cuts around $x$}

As outlined at the start, we prove in two steps that for a minimal tall tame critical easel $(\g, z, x, \psi)$, there are no small cuts around $x$. The first of these steps is particularly easy:
\begin{lemma}\label{lemma-sepxsmall}
Suppose $(\g,z,x,\psi)$ is a minimal tall tame critical easel.  If $A\subsetneq V(\g)\setminus\{z\}$ contains $x$
and $|A|\ge 2$, then $\deg(A) \ge \deg(x)+2$.
\end{lemma}
\begin{proof}
Suppose for a contradiction that $|A|\ge 2$ but $\deg(A)\le \deg(x)+1$, and let us choose a minimal set $A$ with this property.
Let $(\g_0,b)=(\g,z)\restriction A$. By Observation~\ref{obs-subcrit}, $(\g_0, b)$ is flow-critical, and since $|V(\g_0)|=|A|+1>2$,
there exists a tip preflow $\psi'$ and a tip-respecting partition $\PP$ of $V(\g_0)$ such that $\g'=\g_0/\PP$ is $\psi'$-critical. Let $x'$ be the vertex of $\g'$ corresponding to the part $X$ of $\PP$ containing $x$. 

Corollary \ref{cor-contrtame} implies that $(\g',b,x',\psi')$ is tame. 
Moreover, either $|X| =1$, in which case we have
$\deg(x') = \deg(x) \geq \deg(z) -1$, or $|X| \geq 2$, in which case the
minimality of $A$ implies $\deg(x') = \deg(X) \geq \deg(x)+2$.
In both cases we have $\deg(x') \geq \deg(x) \geq \deg(A)-1 = \deg(b)-1$, and
thus it follows that the easel $(\g',b,x',\psi')$ is a tall.

Moreover, $o(\g',b) < o(\g,z)$, and thus the tall tame critical easel $(\g',b,x',\psi')$ contradicts the minimality of $(\g,z,x,\psi)$.
\end{proof}

In the proof of Lemma \ref{lemma-sepxsmall}, we do not need to worry about the case that $(\g',b,x',\psi')$ is non-tall
because all edges between $b$ and $x'$ have the same direction and $\deg(x')=\deg(b)-2$,
since our assumption is that $\deg(A)<\deg(x)+2$. Now we show that with some clever selection of edges we can improve the bound in the above lemma by one.

\begin{lemma}\label{lemma-sepx}
Suppose $(\g,z,x,\psi)$ is a minimal tall tame critical easel.  If $A\subsetneq V(\g)\setminus\{z\}$ contains $x$
and $|A|\ge 2$, then $\deg(A) > \deg(x)+2$.
\end{lemma}
\begin{proof}
Suppose for a contradiction that $A\neq V(\g)\setminus\{z\}$ but
$\deg(A)=\deg(x)+2$, and let us consider a maximal such set $A$.  Let
$C=V(\g)\setminus (A\cup \{z\})$. Our goal is to find two edges $e_{1}$ and
$e_{2}$ between $C$ and $A \setminus \{x\}$ such that the canvas obtained by
contracting $A$ to a vertex and splitting off $e_{1}$ and $e_{2}$ has a flow extending $\psi$.

Since $\g-\{x,z\}$ is connected by Corollary~\ref{cor-2con} and $C\neq\emptyset$, there exists an edge $e_1$ between $C$ and $A\setminus\{x\}$.
Let $u_1$ be the end of $e_1$ in $C$.  Let $Q$ be the set of edges of $\g$ between $C\cup\{z\}$ and $A\setminus\{x\}$; since $\deg(A)=\deg(x)+2$, we have $|Q|\ge 2$.
Let us now describe how to choose an edge $e_2\in Q\setminus\{e_1\}$:
\begin{enumerate}
\item If $\psi$ directs all edges between $z$ and $C$ away from $z$ and there exists an edge $e$ between $A\setminus\{x\}$ and $z$ directed by $\psi$
towards $z$, then let $e_2=e$.
\item Otherwise, if $\psi$ directs all edges between $z$ and $C$ towards $z$ and there exists an edge $e$ between $A\setminus\{x\}$ and $z$ directed by $\psi$
away from $z$, then let $e_2=e$.
\item Otherwise, if there exists an edge in $Q$ not incident with $u_1$, choose $e_2$ as such an edge.
\item Otherwise choose $e_2\in Q\setminus\{e_1\}$ arbitrarily. 
\end{enumerate}

See Figure \ref{fig:pickingedgescases} for an illustration of the cases.
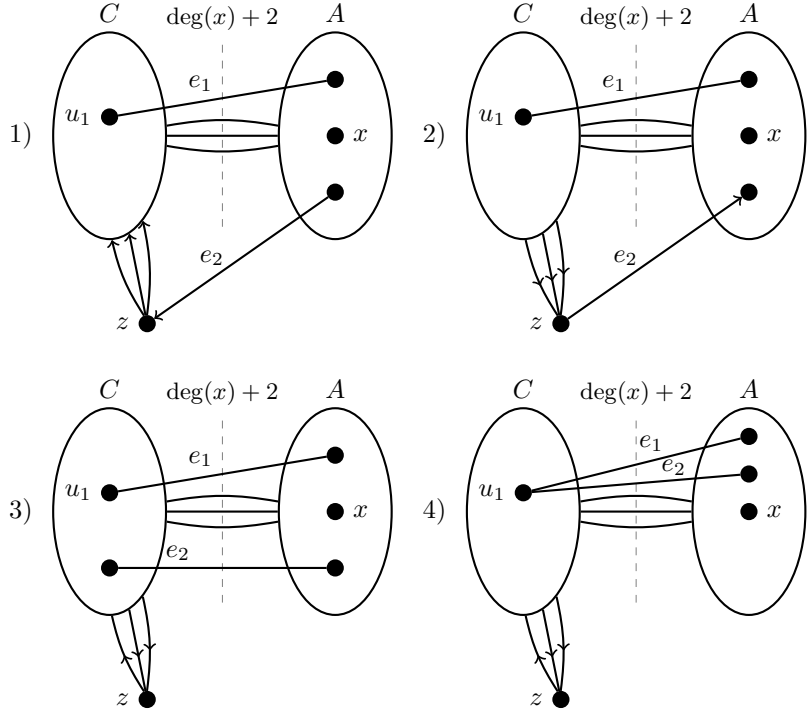
\begin{figure}
\begin{center}
\begin{tikzpicture}
    \node[blackvertexv2] at (0,.25) (u1) [label = left:$u_{1}$] {};
    \node[dummywhite] at (-1.5,0) (case1) [label =right: $1)$] {};
    \node[ellipsenodev1] at (0,0) (ellipse1) [label = above:$C$] {};
    \node[ellipsenodev1] at (3,0) (ellipse2) [label = above:$A$] {};
    \node[blackvertexv2] at (3,0) (x) [label = right:$x$] {};
    \node[blackvertexv2] at (3,.75) (u2) {};
    \node[blackvertexv2] at (.5,-2.5) (z) [label = left:$z$] {};
    \node[blackvertexv2] at (3,-.75) (u3) {};
    \node[dummywhite] at (1.5,1.25) (dummywhite1) [label = above: \small $\deg(x) +2$]  {};
    \node[dummywhite] at (1.5,-1.25) (dummywhite2){};
     \draw[ultra thin,dashed,gray] (dummywhite1) to (dummywhite2);
     \draw[thick,black] (u2)--node[left,,yshift=.2cm]{$e_{1}$}(u1);
     \draw[thick,black, ->] (z) to (ellipse1);
     \draw[thick,black, bend left = 10,->] (z) to (ellipse1);
     \draw[ thick,black, bend right = 10,->] (z) to (ellipse1);
     \draw[thick,black] (ellipse1) to (ellipse2);
     \draw[thick,black,bend right =10] (ellipse1) to (ellipse2);
     \draw[thick,black,bend left =10] (ellipse1) to (ellipse2);
     \draw[thick,black,->] (u3) to node[left,xshift=-.1cm]{$e_{2}$}(z);

     \begin{scope}[xshift =5.5cm]
           \node[blackvertexv2] at (0,.25) (u1) [label = left:$u_{1}$] {};
    \node[ellipsenodev1] at (0,0) (ellipse1) [label = above:$C$] {};
    \node[ellipsenodev1] at (3,0) (ellipse2) [label = above:$A$] {};
      \node[dummywhite] at (-1.5,0) (case1) [label =right:$2)$] {};
    \node[blackvertexv2] at (3,0) (x) [label = right:$x$] {};
    \node[blackvertexv2] at (3,.75) (u2) {};
    \node[blackvertexv2] at (.5,-2.5) (z) [label = left:$z$] {};
    \node[blackvertexv2] at (3,-.75) (u3) {};
    \node[dummywhite] at (1.5,1.25) (dummywhite1) [label = above: \small $\deg(x) +2$]  {};
    \node[dummywhite] at (1.5,-1.25) (dummywhite2){};
     \draw[ultra thin,dashed,gray] (dummywhite1) to (dummywhite2);
     \draw[thick,black] (u2)--node[left,,yshift=.2cm]{$e_{1}$}(u1);
     \draw[thick,black,  postaction={decoration={markings,mark=at position 0.5 with {\arrow{<}}},decorate}] (z) to (ellipse1);
     \draw[thick,black, bend left = 10,postaction={decoration={markings,mark=at position 0.5 with {\arrow{<}}},decorate}] (z) to (ellipse1);
     \draw[ thick,black, bend right = 10,postaction={decoration={markings,mark=at position 0.5 with {\arrow{<}}},decorate}] (z) to (ellipse1);
     \draw[thick,black] (ellipse1) to (ellipse2);
     \draw[thick,black,bend right =10] (ellipse1) to (ellipse2);
     \draw[thick,black,bend left =10] (ellipse1) to (ellipse2);
     \draw[thick,black,<-] (u3) to node[left,xshift=-.1cm]{$e_{2}$}(z);
     \end{scope}

      \begin{scope}[xshift =5.5cm,yshift =-5cm]
       \node[blackvertexv2] at (0,.25) (u1) [label = left:$u_{1}$] {};
    \node[ellipsenodev1] at (0,0) (ellipse1) [label = above:$C$] {};
    \node[ellipsenodev1] at (3,0) (ellipse2) [label = above:$A$] {};
     \node[dummywhite] at (-1.5,0) (case1) [label =right:$4)$] {};
    \node[blackvertexv2] at (3,0) (x) [label = right:$x$] {};
    \node[blackvertexv2] at (3,.5) (u2) {};
    \node[blackvertexv2] at (.5,-2.5) (z) [label = left:$z$] {};
    \node[blackvertexv2] at (3,1) (u3) {}; 
    \node[dummywhite] at (1.5,1.25) (dummywhite1) [label = above: \small $\deg(x) +2$]  {};
    \node[dummywhite] at (1.5,-1.25) (dummywhite2){};
     \draw[ultra thin,dashed,gray] (dummywhite1) to (dummywhite2);
     \draw[thick, black] (u2)--node[above,yshift=.3cm,xshift=.2cm]{$e_{1}$}(u1);
     \draw[thick,black,  postaction={decoration={markings,mark=at position 0.5 with {\arrow{<}}},decorate}] (z) to (ellipse1);
     \draw[thick,black, bend left = 10,postaction={decoration={markings,mark=at position 0.5 with {\arrow{>}}},decorate}] (z) to (ellipse1);
     \draw[ thick,black, bend right = 10,postaction={decoration={markings,mark=at position 0.5 with {\arrow{<}}},decorate}] (z) to (ellipse1);
     \draw[thick,black] (ellipse1) to (ellipse2);
     \draw[thick,black,bend right =10] (ellipse1) to (ellipse2);
     \draw[thick,black,bend left =10] (ellipse1) to (ellipse2);
    \draw[thick,black] (u1) to node[below,yshift =.2cm, xshift =.5cm]{$e_{2}$}(u3);

     \end{scope}

      \begin{scope}[yshift =-5cm]
 \node[blackvertexv2] at (0,.25) (u1) [label = left:$u_{1}$] {};
    \node[ellipsenodev1] at (0,0) (ellipse1) [label = above:$C$] {};
    \node[ellipsenodev1] at (3,0) (ellipse2) [label = above:$A$] {};
        \node[dummywhite] at (-1.5,0) (case1) [label =right:$3)$] {};
    \node[blackvertexv2] at (3,0) (x) [label = right:$x$] {};
    \node[blackvertexv2] at (3,.75) (u2) {};
    \node[blackvertexv2] at (.5,-2.5) (z) [label = left:$z$] {};
    \node[blackvertexv2] at (3,-.75) (u3) {}; 
    \node[blackvertexv2] at (0,-.75) (u4) {};
    \node[dummywhite] at (1.5,1.25) (dummywhite1) [label = above: \small $\deg(x) +2$]  {};
    \node[dummywhite] at (1.5,-1.25) (dummywhite2){};
     \draw[ultra thin,dashed,gray] (dummywhite1) to (dummywhite2);
     \draw[thick, black] (u2)--node[left,,yshift=.2cm]{$e_{1}$}(u1);
     \draw[thick,black,  postaction={decoration={markings,mark=at position 0.5 with {\arrow{<}}},decorate}] (z) to (ellipse1);
     \draw[thick,black, bend left = 10,postaction={decoration={markings,mark=at position 0.5 with {\arrow{>}}},decorate}] (z) to (ellipse1);
     \draw[ thick,black, bend right = 10,postaction={decoration={markings,mark=at position 0.5 with {\arrow{<}}},decorate}] (z) to (ellipse1);
     \draw[thick,black] (ellipse1) to (ellipse2);
     \draw[thick,black,bend right =10] (ellipse1) to (ellipse2);
     \draw[thick,black,bend left =10] (ellipse1) to (ellipse2);
     \draw[thick,black] (u3)--node[left,yshift=.2cm, xshift=-.3cm]{$e_{2}$}(u4);
     \end{scope}
\end{tikzpicture}
\caption{The four possible cases in Lemma \ref{lemma-sepx} for how we choose the edges $e_{1}$ and $e_{2}$. Note that in cases $1$ and $2$, we must have all of the edges from $z$ to $C$ oriented in the same direction, whereas in cases $3$ and $4$, either there is no edge from $z$ to $A$, or not all the edges go the same direction. Hence, these figures are merely examples of the four possible cases.}
\label{fig:pickingedgescases}
\end{center}
\end{figure}

Let $\g_1$ be the $\Z$-bordered graph obtained from $\g$ by contracting $A$ to a single vertex $a$ and then splitting off the edges $e_1$ and $e_2$;
let $e$ denote the resulting edge added in the case that $e_1$ and $e_2$ are not incident with the same vertex in $V(\g)\setminus A$.
We can view $\psi$ as a tip preflow in $(\g_1,z)$ (in the case that $e_2$ is incident with $z$, $e$ inherits its orientation).

\begin{claim*}
\label{claim:preflowextension}
The preflow $\psi$ extends to a nowhere-zero flow $\vec{G}_{1}$ in $\g_1$.
\end{claim*}

\begin{subproof}
If $\psi$ does not extend to a nowhere-zero flow in $\g_1$, then there exists a $\psi$-critical tip-respecting contraction $(\g'_1,z)$ of $(\g_1,z)$.
Let $a'$ be the vertex of $\g'_1$ into which we contracted $a$, and consider
the critical easel $(\g'_1,z,a',\psi)$. We aim to show that
$(\g'_{1},z,a',\psi)$ is a tall tame critical easel, contradicting the minimality of
$(\g,z,x,\psi)$.

\begin{subclaim*}
   The easel $(\g',z,a',\psi)$ is tame.
\end{subclaim*}

\begin{subsubproof}
 Let $u'$ be any vertex in $V(\g'_1)\setminus \{z,a'\}$, and let $U$ be the set of vertices of $\g$ contracted into $u'$. 
    Note that if both $e_1$ and $e_2$ have exactly one end in $U$, then $\deg_{\g'_1}(u')=\deg_{\g}(U)-2$, otherwise $\deg_{\g'_1}(u')=\deg_{\g}(U)$.
    In particular, $\deg_{\g'_1}(u')$ and $\deg_{\g}(U)$ have the same parity, and thus $\tau(u')=\tau(U)$.
    We now discuss three cases depending on the size of $U$ and the incidence of $e_1$ and $e_2$ with the vertices of $U$.
\begin{itemize} 
\item If $U$ contains at least two vertices, then by Lemma~\ref{lemma-cut} we have $\deg_{\g'_1}(u')\ge \deg_{\g}(U)-2\ge 4+|\tau(U)|=4+|\tau(u')|$.
\item If $U$ consists of a single vertex, say $u$, and  $\deg_{\g'_1}(u')=\deg_{\g}(U)=\deg_{\g}(u)$, then $\deg_{\g'_1}(u')=\deg_{\g}(u)\ge 4+|\tau(u)|=4+|\tau(u')|$
by the tameness of $(\g,z,x,\psi)$. 
\item If $|U|=1$ and $\deg_{\g'_1}(u')<\deg_{\g}(U)$, then $U=\{u_1\}$ and both $e_1$ and $e_2$ are incident with $u_{1}$.  By our choice of $e_2$, this means that all edges of $Q$ are incident with $u_{1}$. In $\g'_1$, all edges of $Q\setminus\{e_1,e_2\}$ join $u'$ with $a'$.  By Lemma~\ref{lemma-simple}
applied to $(\g'_1,z)$, there is at most one edge between $u'$ and $a'$, and thus $|Q|\le 3$.  Since $\deg(A)=\deg(x)+2$,
there are $\deg(A)-|Q|=\deg(x)+2-|Q|$ edges between $x$ and $C\cup\{z\}$, and thus there are $|Q|-2$ edges between $x$ and $A\setminus\{x\}$.
Consequently, $\deg(A\setminus\{x\})=2|Q|-2\le 4$.  By Lemma~\ref{lemma-simple} for $(\g,z)$, the edges $e_1$ and $e_2$ are incident with
distinct vertices of $A\setminus\{x\}$, and thus $|A\setminus \{x\}|\ge 2$.  This contradicts Lemma~\ref{lemma-cut} for the set $A\setminus\{x\}$ in $(\g,z)$, and thus this case does not happen.
\end{itemize}
Therefore, we have $\deg_{\g'_1}(u')\ge 4+|\tau(u')|$ for every $u'\in V(\g'_1)\setminus\{z,a'\}$, and thus the easel $(\g'_1,z,a',\psi)$
is tame. 
\end{subsubproof}

Now we check that $(\g_{1}',z,a',\psi)$ is tall.

\begin{subclaim*}
The easel $(\g_{1}',z,a',\psi)$ is tall.
\end{subclaim*}

\begin{subsubproof}
Let $A'$ be the subset of vertices of $\g$ contracted into $a'$. Note that $A'$ is a superset of $A$.
\begin{itemize}
\item If $A'\neq A$, then the maximality of $A$ implies $\deg_{\g'}(a')\ge \deg_{\g}(A')-2>\deg(x)\ge \deg(z)-2$.
\item If $A'=A$, then $\deg(a')\ge \deg_{\g}(A)-2=\deg(x)\ge \deg(z)-2$.  Moreover, we claim that if $\deg(z)=\deg(x)+2$, then $\psi$ does not direct all edges between $z$ and $V(\g'_1)\setminus \{z,a'\}$ the same way.  Indeed, since $\deg(z)=\deg(x)+2=\deg(A)$ and at least the edge $e_1$ with exactly one end in $A$ is not incident with $z$, there exists at least one edge between $z$
and $C$.  If $\psi$ directs all edges from $z$ to $C$ in the same way, say away from $z$, and $\deg(z)=\deg(x)+2$, then the tallness of $(\g,z,x,\psi)$ implies that there exists an edge between $A\setminus\{x\}$ and $z$ directed by $\psi$ towards $z$, and $e_2$ is chosen as such an edge.
Therefore, the edge $e$ of $\g'_1$ arising from the splitting off $e_1$ and $e_2$ is not incident with $a'$ and it is directed in the same way as $e_2$ by $\psi$, i.e., towards $z$.
\end{itemize}
We conclude that the easel $(\g'_1,z,a',\psi)$ is tall. 
\end{subsubproof}

  The above two subclaims contradict our assumption that $(\g,z,x,\psi)$ is a minimal tall tame critical easel, and thus $\psi$ extends to a nowhere-zero flow $\vec{G}_1$ in $\g_1$.
\end{subproof}

The claim implies that $\psi$ extends to a nowhere-zero flow $\vec{G}_a$ in $\g/A$ with one of the edges $e_1$ and $e_2$ directed towards $a$ and the other one away from $a$,
obtained from $\vec{G}_1$ by directing $e_1$ and $e_2$ according to the orientation of $e$ (or arbitrarily in opposite directions in the
case that $e_1$ and $e_2$ are incident with the same vertex of $C$, and thus the edge $e$ is not added when splitting off $e_1$ and $e_2$).

Let $(\g_2,z_2)=(\g,z)\restriction A$ and let $\psi_2$ be the tip preflow matching the orientations of the edges of $\vec{G}_a$ incident with $a$.
If $\psi_2$ extended to a nowhere-zero flow in $\g_2$, then this nowhere-zero flow would combine with $\vec{G}_a$
to a nowhere-zero flow in $\g$ extending $\psi$, which is a contradiction.  Hence, there exists a $\psi_2$-critical tip-respecting contraction $\g'_2$
of $\g_2$.  Let $x_2$ be the vertex of $\g'_2$ into which we contracted $x$, and let $X_2$ be the set of vertices of $\g$ contracted to $x_2$.
Then $(\g'_2,z_2,x_2,\psi_2)$ is a critical easel, and it is easy to see that this easel is tame by Lemma~\ref{lemma-cut}. We now show that it is tall:
\begin{itemize}
\item If $|X_2|>1$, then Lemma~\ref{lemma-sepxsmall} implies $\deg(x_2)=\deg(X_2)\ge \deg(x)+2=\deg(A)=\deg(z_2)$.
\item If $X_2=\{x\}$, then $\deg(x_2)=\deg(x)=\deg(z_2)-2$. The choice of $\psi_2$ implies that $\psi_2$ directs the edges $e_1$ and $e_2$
in the opposite direction.  Moreover, since neither $e_1$ nor $e_2$ are incident with $x$ in $\g$, they are also not incident with $x_2$
in $\g'_2$.
\end{itemize}
Now, since $o(\g'_2,z_2)<o(\g,z)$, this contradicts the assumption that $(\g,z,x,\psi)$ is a minimal tall tame critical easel.
\end{proof}

\subsection{Step 3: Minimal tall tame easels are $x$-homogeneous}

We are now sufficiently prepared to get rid of degree four vertices (which necessarily have boundary $0$). In fact, we could similarly get rid of every even degree vertex with boundary zero, but this is not needed for our proof.

\begin{lemma}\label{lemma-no4}
If $(\g,z,x,\psi)$ is a minimal tall tame critical easel, then every vertex $v\in V(\g)\setminus\{x,z\}$ has degree at least five.
\end{lemma}
\begin{proof}
Let $\g=(G,\beta)$.  Suppose for a contradiction that there exists a vertex $v \in V(\g) \setminus \{x,z\}$ such that $\deg(v)\le 4$.  Since the easel is tame, we have $\deg(v)\ge 4+|\tau(v)|$, and thus $\deg(v)=4$, $|\tau(v)|=0$, and $\beta(v)=0$.  By Lemmas~\ref{lemma-2con} and \ref{lemma-at4},
there exists an edge $e_1$ between $v$ and a vertex $v_1\in V(\g)\setminus\{v,x,z\}$.
Let us now distinguish several cases:
\begin{itemize}
\item[(i)] If there are two edges between $v$ and $z$ that are directed oppositely by $\psi$, then let $e_3$ and $e_4$ be such edges
and let $e_2$ be the edge incident with $v$ and distinct from $e_1$, $e_3$, and $e_4$.
\item[(ii)] If $v$ is adjacent to $z$ but all edges between $v$ and $z$ are directed in the same way by $\psi$ (all towards $z$ or all away from $z$),
then note that there are at most two such edges, as otherwise $\psi$ would not extend to a nowhere-zero flow
in $\g/\{v_1,x\}$, contradicting the $\psi$-criticality of $(\g,z)$.  We let $e_2$ be an edge between $z$ and $v$,
and let $e_3$ and $e_4$ be the edges incident with $v$ and distinct from $e_1$ and $e_2$.
\item[(iii)] Finally, if $v$ is not adjacent to $z$, then we assign the labels $e_2$, $e_3$, and $e_4$ to
the edges incident with $v$ and distinct from $e_1$ arbitrarily.
\end{itemize}
See Figure \ref{fig:deg4splitting} for an illustration of these cases.

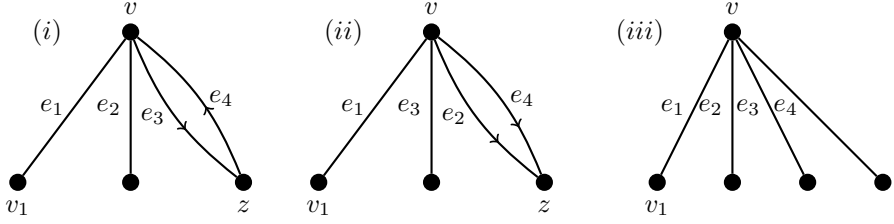
\begin{figure}
\begin{center}
\begin{tikzpicture}
\node[blackvertexv2] at (0,0) (v) [label= above:$v$] {};
\node[blackvertexv2] at (1.5,-2) (z) [label = below:$z$] {};
\node[blackvertexv2] at (-1.5,-2) (v1) [label = below:$v_{1}$] {};
\node[blackvertexv2] at (0,-2) (v2) {};
\node[dummywhite] at (-.75,0) (dummywhite1) [label = left:$(i)$]  {};
\draw[thick,black] (v) -- node[left]{$e_{1}$}(v1);
\draw[thick,black, bend left = 15,postaction={decoration={markings,mark=at position 0.6 with {\arrow{<}}},decorate}] (v) to node[right]{$e_{4}$}(z);
\draw[thick,black, bend right = 15,postaction={decoration={markings,mark=at position 0.6 with {\arrow{>}}},decorate}] (v) to node[left]{$e_{3}$}(z);
\draw[thick,black] (v) to node[left]{$e_{2}$} (v2);

\begin{scope}[xshift = 4cm]
\node[blackvertexv2] at (0,0) (v) [label= above:$v$] {};
\node[blackvertexv2] at (1.5,-2) (z) [label = below:$z$] {};
\node[blackvertexv2] at (-1.5,-2) (v1) [label = below:$v_{1}$] {};
\node[dummywhite] at (-.75,0) (dummywhite1) [label = left:$(ii)$]  {};
\node[blackvertexv2] at (0,-2) (v2) {};
\draw[thick,black] (v) -- node[left]{$e_{1}$}(v1);
\draw[thick,black, bend left = 15,postaction={decoration={markings,mark=at position 0.7 with {\arrow{>}}},decorate}] (v) to node[right]{$e_{4}$}(z);
\draw[thick,black, bend right = 15,postaction={decoration={markings,mark=at position 0.7 with {\arrow{>}}},decorate}] (v) to node[left]{$e_{2}$}(z);
\draw[thick,black] (v) to node[left]{$e_{3}$} (v2);
\end{scope}

\begin{scope}[xshift = 8cm]
\node[blackvertexv2] at (0,0) (v) [label= above:$v$] {};
\node[blackvertexv2] at (1,-2) (z) {};
\node[blackvertexv2] at (-1,-2) (v1) [label = below:$v_{1}$] {};
\node[blackvertexv2] at (0,-2) (v2) {};
\node[blackvertexv2] at (2,-2) (v3) {};
\node[dummywhite] at (-.75,0) (dummywhite1) [label = left:$(iii)$]  {};
\draw[thick,black] (v) -- node[left,yshift =-.1]{$e_{1}$}(v1);
\draw[thick,black] (v) to node[left,yshift =-.1]{$e_{3}$}(z);
\draw[thick,black] (v) to node[left,yshift=-.1]{$e_{4}$}(v3);
\draw[thick,black] (v) to node[left,yshift =-.1]{$e_{2}$} (v2);
\end{scope}
\end{tikzpicture}
\caption{The three cases in Lemma \ref{lemma-no4}. Note in case (i), the edge $e_{2}$ may be incident to $z$, and in case (iii) there may be parallel edges.} 
\label{fig:deg4splitting}
\end{center}
\end{figure}
Let $\g'$ be obtained from $\g$ by splitting off $e_1$ with $e_2$ and $e_3$ with $e_4$ and deleting the now isolated vertex $v$.
Note that if $e_3$ and $e_4$ are both directed by $\psi$, then they are directed in opposite directions.
Hence, $\psi$ naturally corresponds to a tip preflow $\psi'$ in the canvas $(\g',z)$.
Moreover, since $\psi$ does not extend to a nowhere-zero flow in $\g$, it is easy to see that $\psi'$
does not extend to a nowhere-zero flow in $\g'$, either.  Hence, there exists a $\psi'$-critical tip-respecting
contraction $(\g'',z)$ of $(\g',z)$.  Let $x'$ be the vertex of $\g''$ into which we contracted $x$. Consider now the easel $(\g'',z,x',\psi')$. We will show it is tame and tall.  

\begin{claim*}
The easel $(\g'',z,x',\psi')$ is tame.
\end{claim*}

\begin{subproof}
 Consider any vertex $u\in V(\g'')\setminus\{z,x'\}$, and let $A_0$ be the set of vertices of $\g$ contracted into $u$.  If at least three of the edges $e_1$, \ldots, $e_4$ have an
end in $A_0$, then let $A=A_0\cup\{v\}$; otherwise let $A=A_0$. Observe that not all four edges $e_{1},\ldots,e_{4}$ have an end in $A_{0}$, as otherwise we contradict criticality of $(\g,z,x,\psi)$.  Note that this ensures that either $\deg(u)=\deg_{\g}(A)$
or $\deg(u)=\deg_{\g}(A)-2$.  Moreover, in the latter case $v$ has either two or three neighbours in $A$, and thus by Lemma~\ref{lemma-simple},
we have $|A|\ge 2$.  Lemma~\ref{lemma-cut} and the tameness of the easel $(\g,z,x,\psi)$ then imply that $|\tau(u)|=|\tau(A)|\ge 4+\deg_{\g}(A)=4+\deg(u)$ in the former case,
and $|\tau(u)|=|\tau(A)|\ge 6+\deg_{\g}(A)=4+\deg(u)$ in the latter case.
\end{subproof}

\begin{claim*}
The easel $(\g'',z,x',\psi')$ is tall.
\end{claim*}

\begin{subproof}
Consider the set $X_0$ of the vertices of $\g$ contracted into $x'$, and let $X=X_0\cup\{v\}$ if at
least three of the edges $e_1$, \ldots, $e_4$ have an end in $X_0$ and $X=X_0$ otherwise. As before, not all four edges can have an end in $X_{0}$, as otherwise we contradict the criticality of $(\g,z,x,\psi)$.
Thus either $\deg(x')=\deg_{\g}(X)$, or $\deg(x')=\deg_{\g}(X)-2$ and $|X|\ge 2$.
\begin{itemize}
\item If $|X|\ge 2$, then Lemma~\ref{lemma-sepx} implies
$\deg(x') \ge \deg_{\g}(X)-2>\deg(x)\ge \deg_{\g}(z)-2\ge \deg_{\g''}(z)-2$, and thus the easel $(\g'',z,x',\psi')$ is tall.
\item If $X=\{x\}$, then $\deg(x')=\deg_{\g}(X)=\deg(x)$.  If $\deg(x)>\deg_{\g}(z)-2$ or $\deg_{\g}(z)>\deg_{\g''}(z)$,
this again implies that the easel $(\g'',z,x',\psi')$ is tall.
\item Finally, suppose that $X=\{x\}$, $\deg(x)=\deg_{\g}(z)-2$, and $\deg_{\g} (z)=\deg_{\g''}(z)$.
In particular, the labels of $e_2$, $e_3$, and $e_4$ were not chosen according to (i),
as in that case splitting off $e_3$ with $e_4$ decreases the degree of $z$.
Since the easel $(\g,z,x,\psi)$ is tall and $\deg(x)=\deg_{\g}(z)-2$, there exist two edges $e_5$ and $e_6$ incident with $z$, not
incident with $x$, and directed oppositely by $\psi$.  Since the labels were not chosen according to (i), we may assume without loss of generality that $e_6$ is not incident with $v$, and thus that $e_6$ is an edge of $\g''$ not incident with $x'$.
If $e_5$ is incident with $v$, then case (iii) does not occur, and further the choice in (ii) ensures that $e_2$ is incident with $z$ and directed in the same way as $e_5$. Moreover,
since $e_1$ is not incident with $x$, we have that $e_1$ and $e_2$ are split off to an edge not incident with $x$ and directed in the
same way as $e_5$. Hence, in this case we again conclude that the easel $(\g'',z,x',\psi')$ is tall.
\end{itemize}
\end{subproof}

The above claims imply that $(\g'',z,x',\psi')$ is a tall tame critical easel, and since $o(\g'',z)<o(\g,z)$, we have that $(\g'',z,x,\psi')$ is smaller than $(\g,z,x,\psi)$, contradicting the assumption that $(\g,z,x,\psi)$ is a minimal tall tame critical easel. This concludes the proof.
\end{proof}

With this, we can remove mixed edges.

\begin{lemma}\label{lemma-norede}
If $(\g,z,x,\psi)$ is a minimal tall tame critical easel, then no edge in $E(\g-\{x,z\})$ is mixed, and thus the canvas $(\g,z)$ is $x$-homogeneous.
\end{lemma}
\begin{proof}
Suppose for a contradiction that $uv\in E(\g-\{x,z\})$ is a mixed edge where, say, $\tau(u)$ is out-friendly and $\tau(v)$ is in-friendly.  Let $\g=(G,\beta)$ and let $\g'=(G-uv,\beta')$, where $\beta'(y)=\beta(y)$ for $y\in V(G)\setminus\{u,v\}$, $\beta'(u)=\beta(u)-1$,
and $\beta'(v)=\beta(v)+1$.  Note that $G-uv$ is connected by Corollary~\ref{cor-2con}, and since $\beta'(V(G))=\beta(V(G))=0$,
$\beta'$ is a $\Z$-boundary for $G-uv$.  If $\g'$ had a nowhere-zero flow extending $\psi$, we could extend it to a nowhere-zero flow in $\g$
by directing the edge $uv$ towards $v$.  Hence, $\psi$ does not extend to a nowhere-zero flow in $\g'$, and thus $(\g',z)$ has a tip-respecting
$\psi$-critical contraction $(\g'',z)$.  Let $x'$ be the vertex of $\g''$ into which we contracted $x$, and consider the easel $(\g'',z,x',\psi)$. We claim this easel is tame and tall.

\begin{claim*}
The easel $(\g'',z,x',\psi)$ is tame.
\end{claim*}

\begin{subproof}
 Consider any vertex $y\in V(\g'')\setminus \{z,x'\}$, and let $Y$ be the set of vertices of $\g$ contracted into $y$.  If $|\{u,v\}\cap Y|\in\{0,2\}$, then $\deg(y)=\deg_{\g}(Y)$ and $\beta'(y)=\beta(Y)$, and thus $\deg(y)\ge 4+|\tau(y)|$
by the tameness of $(\g,z,x,\psi)$ and Lemma~\ref{lemma-cut}.  Hence, by symmetry we can assume that $u\in Y$ and $v\not\in Y$.
Then $\deg(y)=\deg_{\g}(Y)-1$ and $\beta'(y)=\beta(Y)-1$, and thus $|\tau(y)|\le |\tau(Y)|+1$ by Observation~\ref{obs-tauprop}(b).
\begin{itemize}
\item If $|Y|\ge 2$, then Lemma~\ref{lemma-cut} gives $\deg(y)=\deg_{\g}(Y)-1\ge 5+|\tau(Y)|\ge 4+|\tau(y)|$.
\item Otherwise, $Y=\{u\}$.  If $\tau(u)$ contains a positive element $b$, then $\tau(y)=\{b-1\}$, $|\tau(y)|=|\tau(u)|-1$,
and by the tameness of  $(\g,z,x,\psi)$ we have $\deg(y)=\deg(u)-1\ge (4+|\tau(u)|)-1\ge 4+|\tau(y)|$.
\item Finally, if $Y=\{u\}$ and $\tau(u)=\{0\}$, then $\deg(u)$ is even, and Lemma~\ref{lemma-no4} gives $\deg(u)\ge 6$.
Since $|\tau(y)|=1$, it follows that $\deg(y)=\deg(u)-1\ge 5=4+|\tau(y)|$.
\end{itemize}
Thus we conclude that $(\g'',z,x',\psi)$ is tame.
\end{subproof}

\begin{claim*}
The easel $(\g'',z,x',\psi)$ is tall.
\end{claim*}

\begin{subproof}
 Let $X$ be the subset of $V(\g)$ contracted into $x'$.  If $|X|\ge 2$, then Lemma~\ref{lemma-sepx}
implies $\deg(x') \ge \deg_{\g}(X) -1\ge \deg(x) + 2\ge \deg(z)$.  If $X=\{x\}$, then $\deg(x')=\deg(x)\ge \deg(z)-2$; and if $\deg(x)=\deg(z)-2$,
the two oppositely directed edges witnessing that $(\g,z,x,\psi)$ is tall are not incident with $x'$ in $\g''$ and witness that $(\g'',z,x',\psi)$ is tall.
\end{subproof}

The above two claims imply $(\g'',z,x',\psi)$ is a critical tall tame easel, and further that $o(\g'',z)<o(\g,z)$. This contradicts the assumption that $(\g,z,x,\psi)$ is a minimal tall tame critical easel, completing the proof.
\end{proof}

A similar reduction applies to suitably directed edges at $z$.

\begin{lemma}\label{lemma-sameor}
If $(\g,z,x,\psi)$ is a minimal tall tame critical easel, then $z$ is not adjacent to $x$,
and either $\tau(v)>0$ for every $v\in V(\g)\setminus\{x,z\}$ and $\psi$ directs all edges away from $z$, or
$\tau(v)<0$ for every $v\in V(\g)\setminus\{x,z\}$ and $\psi$ directs all edges towards $z$.
Moreover, $\deg(z)\le \deg(x)+1$.
\end{lemma}
\begin{proof}

Lemma~\ref{lemma-norede} together with Observation~\ref{obs-allplusminus} imply that either $\tau(v)>0$ for every $v\in V(\g)\setminus\{x,z\}$,
or $\tau(v)<0$ for every $v\in V(\g)\setminus\{x,z\}$.  By symmetry, we can assume the former.

First we claim that there are no edges between $x$ and $z$.   Suppose this is not true and let $e$ be be such an edge. Then deleting $e$ and adjusting the boundary at $x$ and $z$ accordingly (depending on which way $e$ is directed by $\psi$) would give a tall tame critical easel $(\g_0,z,x,\psi_0)$
such that $o(\g_0,z)=o(\g,z)$ and $\deg_{\g_0}(z)<\deg_{\g}(z)$. Since the canvas $(\g,z)$ is $x$-homogeneous, we would have $(\g_0,z,x)\prec (\g,z,x)$.  This is a contradiction, and thus $z$ and $x$ are non-adjacent.

Now we claim that $\psi$ orients all edges away from $z$. Suppose for a contradiction that $\psi$ orients an edge $e=uz$ towards $z$.
Let $\g=(G,\beta)$ and let $\g'=(G-e,\beta')$, where $\beta'(y)=\beta(y)$ for $y\in V(G)\setminus\{u,z\}$, $\beta'(u)=\beta(u)-1$,
and $\beta'(z)=\beta(z)+1$.  Note that $\beta'$ is a $\Z$-boundary for $G-e$: By Corollary~\ref{cor-2con}, $G-z$ is connected,
and thus $G-e$ is disconnected only if $\deg(z)=1$.  But then the definition of a tip preflow ensures that $\beta(z)=-1$ and $\beta'(z)=0$,
and thus also $\beta'(V(G-z))=0$.

Let $\psi'$ be obtained from $\psi$ by deleting the edge $e$.
Note that $\psi'$ does not extend to a nowhere-zero flow in $\g'$, as otherwise adding $e$ oriented as in $\psi$ would give a nowhere-zero flow in $\g$ extending $\psi$.
Conversely, $\psi'$ extends to a nowhere-zero flow in any proper tip-respecting contraction $(\g'/ \PP,z)$, since $\psi$ extends to a nowhere-zero flow in $(\g / \PP,z)$
which can be turned into a nowhere-zero flow in $(\g'/ \PP,z)$ by deleting $e$.  Therefore, $(\g',z,x,\psi')$ is a critical easel.
Since $\deg_{\g'}(z)\le \deg_{\g}(z)-1\le \deg(x)+1$, this easel is tall.  As in the proof of Lemma~\ref{lemma-norede}, we can also show that this easel is tame.
Note that $o(\g',z)=o(\g,z)$, the canvas $(\g,z)$ is $x$-homogeneous, and $\deg_{\g'}(z)<\deg_{\g}(z)$, and thus
$(\g',z,x)\prec (\g,z,x)$.  This contradicts the assumption that $(\g,z,x,\psi)$ is a minimal tall tame critical easel.

Lastly, we need to show that $\deg(z) \leq \deg(x) +1$. This is the case by the definition of tallness,
since $\psi$ directs all edges between $z$ and $V(\g)\setminus \{x,z\}$ in the same
direction.
\end{proof}

\subsection{Step 4: No minimal tall tame critical easels exist.}

Finally, we complete the proof of Theorem \ref{thm:tall}, stating that tall tame easels are not critical.
See Figure~\ref{fig:step4} for a sketch of the situation.

\begin{figure}
\begin{center}
\begin{tikzpicture}
\node[blackvertexv2] at (0,0) (z) [label = below: \footnotesize{$\deg(z) \leq \deg(x) + 1$}] [label = left:$z$] {};
\node[blackvertexv2] at (3,0) (x) [label = below:$x$] {};
\draw[black,dotted] (z) to (x);
\node[ellipsenodev3] at (1.5,2) (ellipse) [label = above:{$\g - \{z,x\}$}] [label = center:{$\tau(v) >0$}] {};
\draw[thick,black, postaction={decoration={markings,mark=at position 0.5 with {\arrow{>}}},decorate}] (z) to (ellipse.south west);
\draw[thick,black,bend left = 30,postaction={decoration={markings,mark=at position 0.5 with {\arrow{>}}},decorate}] (z) to (ellipse.195);
\draw[thick,black,bend left = 60,postaction={decoration={markings,mark=at position 0.5 with {\arrow{>}}},decorate}] (z) to (ellipse.190);
\draw[thick,black,bend right = 30,postaction={decoration={markings,mark=at position 0.5 with {\arrow{>}}},decorate}] (z) to (ellipse.205);
\draw[thick,black,bend right = 60,postaction={decoration={markings,mark=at position 0.5 with {\arrow{>}}},decorate}] (z) to (ellipse.210);
\draw[thick,black,bend right = 30] (x) to (ellipse.350);
\draw[thick,black,bend right = 15] (x) to (ellipse.345);
\draw[thick,black,bend left = 30] (x) to (ellipse.335);
\draw[thick,black,bend left = 15] (x) to (ellipse.south east);
\begin{scope}[xshift = 5cm]
\node[blackvertexv2] at (0,0) (z) [label = below: \footnotesize{$\deg(z) \leq \deg(x) + 2$}] [label = left:$z$] {};
\node[blackvertexv2] at (3,0) (x) [label = below:$x$] {};
\draw[black,dotted] (z) to (x);
\node[ellipsenodev3] at (1.5,2) (ellipse) [label = above:{$\g - \{z,x\}$}] [label = center:{$\tau(v) >0$}] {};

\draw[thick,black, postaction={decoration={markings,mark=at position 0.5 with {\arrow{<}}},decorate}, bend left = 80] (z) to (ellipse.190);
\draw[thick,black, postaction={decoration={markings,mark=at position 0.5 with {\arrow{>}}},decorate}] (z) to (ellipse.south west);
\draw[thick,black,bend left = 25,postaction={decoration={markings,mark=at position 0.5 with {\arrow{>}}},decorate}] (z) to (ellipse.195);
\draw[thick,black,bend left = 40,postaction={decoration={markings,mark=at position 0.5 with {\arrow{<}}},decorate}] (z) to (ellipse.190);
\draw[thick,black,bend right = 30,postaction={decoration={markings,mark=at position 0.5 with {\arrow{>}}},decorate}] (z) to (ellipse.205);
\draw[thick,black,bend right = 60,postaction={decoration={markings,mark=at position 0.5 with {\arrow{>}}},decorate}] (z) to (ellipse.210);
\draw[thick,black,bend right = 30] (x) to (ellipse.350);
\draw[thick,black,bend right = 15] (x) to (ellipse.345);
\draw[thick,black,bend left = 30] (x) to (ellipse.335);
\draw[thick,black,bend left = 15] (x) to (ellipse.south east);
\end{scope}
\end{tikzpicture}
\caption{The situation at the start of Step 4. All vertices $v \in \g - \{z,x\}$ have $\tau(v) >0$, the vertex $z$ has degree at most $\deg(x) +1$, and moreover $z$ does not have an arc to $x$. The final reduction is to take any arc incident to $z$, reverse it, and observe that the resulting easel contradicts our choice of counterexample.}\label{fig:step4}
\end{center}
\end{figure}
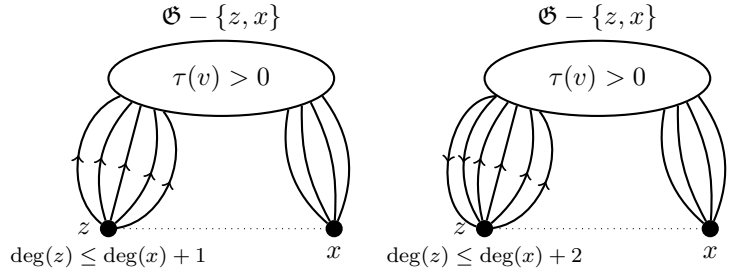
\begin{proof}[Proof of Theorem~\ref{thm:tall}]
Suppose for a contradiction that Theorem~\ref{thm:tall} is false, and thus there exists a minimal tall tame critical easel $(\g,z,x,\psi)$.
By Lemma~\ref{lemma-sameor} and symmetry, we may assume that $xz\not\in E(\g)$, that $\tau(v)>0$ for every $v\in V(\g)\setminus\{x,z\}$, that $\psi$ directs all edges away from $z$, and that $\deg(z)\le \deg(x)+1$.  Since $\g$ is connected by Lemma~\ref{lemma-conn}, there exists at least one edge $e=vz$ incident with $z$.
Let $\g=(G,\beta)$. Let $\g'=(G',\beta)$ where $G'$ is obtained from $G$ by adding an edge $e'$ parallel to $e$, and let $\psi'$ be the preflow around $z$ in $\g'$ obtained from $\psi$ by reversing $e$ and directing $e'$ towards $z$.  Note that $\psi'$ does not extend to a nowhere-zero flow in $\g'$, as otherwise
the same orientation of the edges of $\g-z$ would together with $\psi$ give a nowhere-zero flow in $\g$.
Moreover, note that $\psi'$ extends to a nowhere-zero flow in every proper tip-respecting contraction of $(\g',z)$, since
$\psi$ extends to a nowhere-zero flow in the corresponding tip-respecting contraction of $(\g,z)$.
Therefore, $(\g',z)$ is $\psi'$-critical, and $(\g',z,x,\psi')$ is a critical easel. We show below that $(\g',z,x,\psi')$ is tall and tame.

\begin{claim*}
The easel $(\g',z,x,\psi')$ is tame.
\end{claim*}

\begin{subproof}
The degree and boundary of all vertices that are not $v$ and $z$ are the same as in $\g$, and hence it suffices to show that $\deg_{\g'}(v) \geq 4 + |\tau_{\g'}(v)|$. 
Since $\tau(v)>0$, we have $\beta(v)\neq 0$.
Moreover, $\deg_{\g'} (v)=\deg_{\g}(v)+1$, and thus $|\tau_{\g'}(v)|\le |\tau_{\g}(v)|+1$ by Observation~\ref{obs-tauprop}(c).  Therefore, $\deg_{\g'}(v)=\deg_{\g}(v)+1\ge 5+|\tau_{\g}(v)|\ge 4+|\tau_{\g'}(v)|$. 
\end{subproof}

\begin{claim*}
The easel $(\g',z,x,\psi')$ is tall.
\end{claim*}

\begin{subproof}
Since $\deg_{\g}(z)\le \deg(x)+1$, we have $\deg_{\g'}(z)\le \deg(x)+2$.
Moreover, if $\deg_{\g'}(z)=\deg(x)+2$, then $\deg_{\g}(z) = \deg(x)+1\ge 2$ (we have $\deg(x)\ge 1$, since $\g$ is connected by Lemma~\ref{lemma-conn}),
and thus $\g'$ contains both edges $e$ and $e'$ directed by $\psi'$ towards $z$
and an edge $e''\neq e$ of $\g$ directed by $\psi'$ away from $z$.
\end{subproof}

Note that $o(\g',z)=o(\g,z)$.  Since $\tau_{\g}(v)>0$ and the degree of $v$ in $\g'$ differs in parity, we have $\tau_{\g'}(v)<0$.
By Lemmas~\ref{lemma-at4} and \ref{lemma-2con}, $v$ has a neighbour $u$ in $V(\g)\setminus \{x,z\}$, and $\tau_{\g'}(u)=\tau_{\g}(u)>0$.
Therefore, the canvas $(\g',z)$ is not $x$-homogeneous.  It follows that $(\g',z,x)\prec (\g,z,x)$, contradicting the choice of $(\g,z,x,\psi)$
as a minimal tall tame critical easel. This completes the proof of Theorem \ref{thm:tall}.
\end{proof}
\section{Generating flow-critical canvases}\label{sec:canvasgeneration}

The goal of this section is to prove a theorem which allows us to generate flow-critical tame canvases $(\g,z)$ with $\deg(z) \leq k$ efficiently. Here, ``efficiently" means relative to generating all tame canvases with $n$ vertices and testing them for flow-criticality. The idea is to take the argument given in Theorem \ref{thm:tall} and tweak it to give a generation theorem for  generating flow-critical tame canvases. Before we can state our theorem, we need quite a bit of preparation.

\subsection{Operations for generating flow-critical canvases}\label{subsec:canvasops}

We start off with an important definition.

\begin{definition}
Given a canvas $(\g,z)$ with $\deg(z)\le k+1$, we say that a tip preflow $\psi$
is a \emph{$k$-tallness-witnessing preflow} if $\psi$ does not extend to a
nowhere-zero flow in $\g$ and if $\deg(z)=k+1$, then additionally $\psi$ does
not direct all edges incident with $z$ in the same direction. If there exists a
$k$-tallness-witnessing preflow in $(\g,z)$, then we say that $(\g,z)$ is
\emph{$k$-tall}. For an integer $k$, let $\GG_k$ denote the class of all
$k$-tall flow-critical tame canvases.
\end{definition}

This seems to be a strange definition, as one might anticipate the class $\GG_k$ only containing canvases where $\deg(z)\le k$. 
However, as in the final reduction in Theorem~\ref{thm:tall}, we will need to
generate some canvases where $\deg(z) = k+1$, hence the exception in the above
definition. Observe that our notion of $k$-tall is very close to the notion for
easels; indeed, if $(\g,z,x,\psi)$ is a tall easel and $\psi$ does not extend
to a nowhere-zero flow in $\g$, then $(\g,z)$ is $(\deg(x)+1)$-tall, with
$\psi$ being a $(\deg(x)+1)$-tallness-witnessing preflow. Let us remark that
$\GG_k$ only contains non-trivial canvases (since each of them has a
$k$-tallness-witnessing preflow and this tip preflow does not extend to a
nowhere-zero flow).  Moreover, note that Corollary~\ref{cor-mainlov} implies
that $\GG_k=\emptyset$ for $k\le 4$, and that $\GG_5$ contains only canvases with tips of degree $6$.

We now introduce the operations needed in the algorithm for generating canvases in $\GG_{k}$.

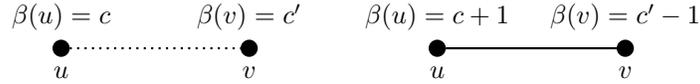
\begin{figure}
\begin{center}
\begin{tikzpicture}
  \node[blackvertexv2] at (0,0) (u) [label = above:$\beta(u) \equal c$] [label = below:$u$] {};
  \node[blackvertexv2] at (2.5,0) (v) [label = above:$\beta(v) \equal c'$] [label = below:$v$]{};
  \draw[thick,black,dotted] (u) to (v);

  \begin{scope}[xshift = 5cm]
    \node[blackvertexv2] at (0,0) (u) [label = above:$\beta(u) \equal c +1 $] [label = below:$u$]{};
  \node[blackvertexv2] at (2.5,0) (v) [label = above:$\beta(v) \equal c'-1$] [label = below:$v$]{};
  \draw[thick,black] (u) to (v);
  \end{scope}
\end{tikzpicture} 
\caption{A \textit{$1$-alteration} at $u$ and $v$. If either $u$ or $v$ is $z$, then this is a \textit{tip-alteration}. }
\label{fig:onealt}
\end{center}
\end{figure}

\begin{itemize}
\item Suppose that a canvas $(\g_0,z)$ is obtained from a canvas $(\g',z)$ by adding an edge between distinct vertices $u$ and $v$,
increasing the boundary at $u$ by $1$ and decreasing the boundary at $v$ by $1$.
If $u\neq z\neq v$, then we say that $(\g_0,z)$ is a \emph{1-alteration} of
$(\g',z)$ at $u$ and $v$; otherwise, $(\g_0,z)$ is a \emph{tip-alteration} of
$(\g',z)$ at the vertex in $\{u,v\}\setminus\{z\}$.  
See Figure~\ref{fig:onealt}.
\item Let $(\g',z)$ be a canvas containing more than one edge between $z$ and a vertex $v\in V(\g')\setminus\{z\}$.
The canvas obtained from $(\g',z)$ by deleting one edge between $z$ and $v$, keeping the boundary unchanged, is a \emph{tip-reduction} of $(\g',z)$ at $v$. See Figure \ref{fig:tipreduct}.
\item Let $(\g', z)$ be a canvas. Let $X$ be a non-empty set of size at most
two whose elements are edges and vertices of $\g'$. A \emph{2-alteration} of $(\g', z)$
on the elements of $X$ is the canvas obtained as follows: We delete from $\g'$
the edges contained in $X$, and add a new vertex $y$ with boundary zero. For
every edge $uv \in X$, we add the edges $uy$ and $yv$. For every vertex $w \in
X$, we add two copies of the edge $yw$. 
See Figure \ref{fig:twoalt}. 

\item Let $\GG$ be a class of canvases.  A \emph{$\GG$-partition} for a canvas $(\g,z)$ is a tip-respecting partition $\PP$ of $V(\g)$
such that for every part $P\in \PP$ of size at least two, the canvas $(\g,z)\restriction P$ is isomorphic to a canvas in $\GG$. We say a canvas $(\g,z)$ is a \emph{$\GG$-expansion} of a canvas $(\g',z)$ if
there exists a $\GG$-partition $\PP$ for $(\g,z)$ such that $\g/\PP=\g'$.  That is, we can obtain $(\g,z)$ from the canvas $(\g',z)$
by choosing for some of the vertices $v\in V(\g')\setminus\{z\}$ a canvas $(\g_v,z_v)\in \GG$ such that $\deg(v)=\deg(z_v)$
and replacing $v$ by $\g_{v}-z_{v}$, with the edges incident with $v$ in $\g'$ redirected to the ends of edges incident with $z_v$. See Figure \ref{fig:gexpansion}.
\end{itemize}

\begin{figure}
\begin{center}
    \begin{tikzpicture}
    \node[blackvertexv2] at (0,0) (z) [label = above:$\beta(z) \equal c$][label = below:$z$] {};
    \node[blackvertexv2] at (2,0) (v) [label = above:$\beta(z) \equal c'$] [label = below:$v$] {};
    \draw[thick,black, bend right = 45] (z) to (v);
    \draw[thick,black] (v) to (z);
    \begin{scope}[xshift = 5cm]
     \node[blackvertexv2] at (0,0) (z) [label = above:$\beta(z) \equal c$][label = below:$z$] {};
    \node[blackvertexv2] at (2,0) (v) [label = above:$\beta(z) \equal c'$] [label = below:$v$] {};
    \draw[thick,black,dotted] (z) to (v);
    \draw[thick,black, bend right = 45] (z) to (v);
    \end{scope}
    \end{tikzpicture}
    \caption{A \textit{tip-reduction} at $z$. Note that the boundaries remain unchanged,
    and there can be more than two edges from $z$ to $v$. }
    \label{fig:tipreduct}
\end{center}
\end{figure}
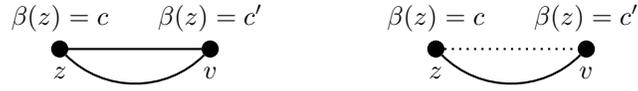

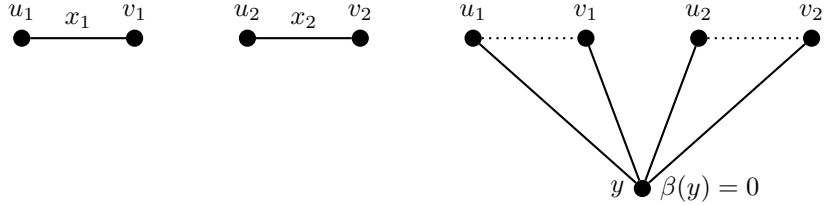
\begin{figure}
\begin{center}
    \begin{tikzpicture}
        \node[blackvertexv2] at (0,0) (u1) [label = above:$u_{1}$] {};
        \node[blackvertexv2] at (1.5,0) (v1) [label = above:$v_{1}$] {};
        \node[blackvertexv2] at (3,0) (u2) [label = above:$u_{2}$] {};
        \node[blackvertexv2] at (4.5,0) (v2) [label = above:$v_{2}$] {};
        \draw[thick,black]  (u1) to node[above]{$x_{1}$}(v1);
        \draw[thick,black] (u2) to node[above]{$x_{2}$}(v2);
        \begin{scope}[xshift = 6cm]
         \node[blackvertexv2] at (0,0) (u1) [label = above:$u_{1}$] {};
        \node[blackvertexv2] at (1.5,0) (v1) [label = above:$v_{1}$] {};
        \node[blackvertexv2] at (3,0) (u2) [label = above:$u_{2}$] {};
        \node[blackvertexv2] at (4.5,0) (v2) [label = above:$v_{2}$] {};
        \node[blackvertexv2] at (2.25,-2) (y) [label= left:$y$] [label = right:$\beta(y) \equal 0$] {};
        \draw[thick,black] (y) --(u1);
        \draw[thick,black] (y) --(v1);
        \draw[thick,black] (y)--(u2);
        \draw[thick,black] (y)--(v2);
        \draw[thick,black,dotted] (u1) to (v1);
        \draw[thick,black,dotted] (u2) to (v2);
        \end{scope}
    \end{tikzpicture}
\end{center}
\caption{A 2-alteration where $X = \{x_1,x_2\}$.}
\label{fig:twoalt}
\end{figure}

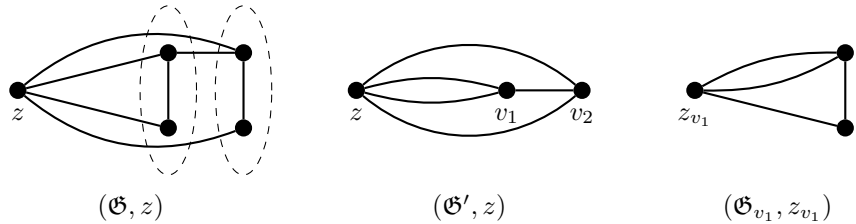
\begin{figure}
\begin{center}
\begin{tikzpicture}
\node[blackvertexv2] at (0,0) (z)  [label =below:$z$] {};
\node[blackvertexv2] at (2,-.5) (v1) {};
\node[blackvertexv2] at (2,.5) (u1) {};
\node[blackvertexv2] at (3,-.5) (v2) {};
\node[blackvertexv2] at (3,.5) (u2) {};
\node[dummywhite] at (1.5,-1.9) (dummy1) [label = above:{$(\g,z)$}] {}; 
 \node[ellipsenodev2] at (2,0) (ellipse1) {};
 \node[ellipsenodev2] at (3,0) (ellipse2) {};
\draw[thick,black] (v1) to (u1);
\draw[thick,black] (v2) to (u2);
\draw[thick,black] (z) to (v1);
\draw[thick,black] (z) to (u1);
\draw[thick,black,bend left =30] (z) to (u2);
\draw[thick,black,bend right =30] (z) to (v2);
\draw[thick,black] (u1)--(u2);

\begin{scope}[xshift = 4.5cm]
\node[blackvertexv2] at (0,0) (z)  [label =below:$z$] {};
\node[blackvertexv2] at (2,0) (u1) [label = below:$v_{1}$]{};
\node[blackvertexv2] at (3,0) (v2) [label = below:$v_{2}$]{};
\node[dummywhite] at (1.5,-1.9) (dummy1) [label = above:{$(\g',z)$}] {}; 
\draw[thick,black, bend left = 15] (z) to (u1);
\draw[thick,black, bend right = 15] (z) to (u1);
\draw[thick,black, bend left =40] (z) to (v2);
\draw[thick,black, bend right = 40] (z) to (v2);
\draw[thick,black] (u1)--(v2);
\end{scope}

\begin{scope}[xshift = 9cm]
\node[blackvertexv2] at (0,0) (z)  [label =below:$z_{v_{1}}$] {};
\node[blackvertexv2] at (2,-.5) (v1) {};
\node[blackvertexv2] at (2,.5) (u1) {};
\draw[thick,black] (v1) to (u1);
\draw[thick,black, bend left = 15] (z) to (u1);
\draw[thick,black, bend right =15] (z) to (u1);
\draw[thick,black] (z) to (v1);
\node[dummywhite] at (1.1,-1.9) (dummy1) [label = above:{$(\g_{v_{1}},z_{v_{1}})$}] {}; 
\end{scope}
\end{tikzpicture}
\caption{A canvas $(\g,z)$ which is a $\GG$-expansion of a canvas $(\g',z)$. Here $(\g_{v_{1}},z_{v_{1}})$ is shown, and is actually isomorphic to $(\g_{v_{2}},z_{v_{2}})$. The $\GG$-partition is shown via the dashed ellipses, with the part containing $z$ omitted.}
\label{fig:gexpansion}
\end{center}
\end{figure}

With that, we can state our canvas generation theorem. Recall our partial order $\prec$ from Definition \ref{def:ordering};
for canvases $(\g_1,z_1)$ and $(\g_2,z_2)$, we write $(\g_1,z_1)\prec (\g_2,z_2)$ if $(\g_1,z_1,z_1)\prec (\g_2,z_2,z_2)$.

\begin{theorem}\label{thm-gen}
Let $k$ be a positive integer.  For every canvas $(\g,z)\in\GG_k$, at least one of the following claims holds:
\begin{itemize}[align=left]
\item[(SMALL)] $|V(\g)|=3$; or,
\item[(EXPA)] $\GG_k$ contains a canvas $(\g',z)\prec (\g,z)$ such that $(\g,z)$ is a $\GG_k$-expansion of $(\g',z)$ or
a $\GG_k$-expansion of a 2-alteration of $(\g',z)$; or,
\item[(EXPB)] $\GG_k$ contains a canvas $(\g',z)\prec (\g,z)$ such that $(\g,z)$ is a $\GG_k$-expansion of a 1-alteration of $(\g',z)$; or,
\item[(ADD)] $(\g,z)$ is a tip-alteration of a canvas $(\g',z)\prec (\g,z)$ belonging to $\GG_k$; or,
\item[(REM)] $(\g,z)$ is a tip-reduction of a canvas $(\g',z)\prec (\g,z)$ belonging to $\GG_k$.
\end{itemize}
Moreover, if $\g$ has minimum degree at most four, then (SMALL) or (EXPA) holds.
\end{theorem}

Before proving this theorem, let us remark that it easily implies the algorithmic result stated in Theorem~\ref{informalgenerationtheorem1}:
All non-trivial flow-critical tame canvases $(\g,z)$ with $\deg(z) \leq k$ belong to $\GG_k$ and can be obtained by a sequence
of operations from the statement of Theorem~\ref{thm-gen}.  The reader might perhaps be worried about the operation of
$\GG_k$-expansion and its effect on the time complexity; indeed, there usually are superpolynomially many (in the number of vertices)
$\GG_k$-expansions that we need to test for flow-criticality.  However, it is easy to see that for $k\ge 6$, the number of non-trivial flow-critical tame canvases $(\g,z)$ with $\deg(z) \leq k$
and at most $n$ vertices is (at least) exponential in $n$, and thus the number of $\GG_k$-expansions to test is bounded by a polynomial in the number
of returned canvases.

\subsection{The proof of Theorem \ref{thm-gen}}\label{subsec:canv-gen}
The outline of the proof of Theorem~\ref{thm-gen} is quite similar to the proof of Theorem~\ref{thm:tall}.  We start off with Step 1, which is  an analogue of Step 2 in the proof of Theorem \ref{thm:tall}, showing that if we are given a set $A$ where the degree of $A$ is at most $k+1$, then the restriction to $A$ is also in $\GG_{k}$. This contains the majority of the technical work in the proof. 
 Once we have this we can proceed to Step 2, which says that minimal counterexamples have minimum degree at least five, there are no mixed edges not incident to $z$, and our preflow orients all edges towards $z$ or away from $z$. This step is nearly identical to Step $3$ in the proof of Theorem \ref{thm:tall}. 
Once we have that, Step 3 of the proof of Theorem~\ref{thm-gen} is to deduce no minimal counterexamples exist, and here the finish is nearly identical to Step 4 in the proof of Theorem \ref{thm:tall}.

Let us start the proof by defining what we mean by a minimal counterexample. 

\begin{definition}
A \emph{minimal $k$-counterexample} is a triple $(\g,z,\psi)$, where $(\g,z)\in\GG_k$ is a minimal canvas in the $\prec$ ordering that does not
satisfy the conclusion of Theorem~\ref{thm-gen} and $\psi$ is a $k$-tallness-witnessing preflow.
\end{definition}

Note that a minimal $k$-counterexample clearly satisfies $|V(\g)|\ge 4$, as otherwise $(\g,z)$ satisfies (SMALL).

\subsection{Step 1: Restrictions from sets with small cuts are in $\GG_{k}$.}

The goal of this subsection is to generalize the following observation to sets with degree at least $k+1$. Note that this observation holds without any assumption on minimality, whereas all other statements in this section require being a minimal $k$-counterexample. 

\begin{observation}\label{obs-inGk}
Let $(\g,z)$ be a tame flow-critical canvas.  If $A\subseteq V(\g)\setminus \{z\}$ has size at least two
and $\deg(A)\le k$, then $(\g,z)\restriction A\in\GG_k$.  Consequently, if $(\g',z)$ is a tip-respecting
contraction of $(\g,z)$ and every vertex other than $z$ has degree at most $k$, then $(\g,z)$ is a $\GG_k$-expansion of $(\g',z)$.
\end{observation}
\begin{proof}
The canvas $(\g_0,z_0)=(\g,z)\restriction A$ is flow-critical by Observation~\ref{obs-subcrit}, and clearly tame.
Since $|V(\g_0)|\ge 3$ and $(\g_0,z_0)$ is flow-critical, there exists a tip preflow $\psi_0$ that does not extend to a nowhere-zero flow in $\g_0$.
Since $\deg(z_0)=\deg(A)\le k$, $\psi_0$ is a $k$-tallness-witnessing preflow, and thus $(\g_0,z_0)\in \GG_k$. The second part of the observation follows immediately from the definition.
\end{proof}

Before we can extend Observation \ref{obs-inGk}, we will need some lemmas on minimal $k$-counterexam\-ples. We start off by observing that if $(\g,z,\psi)$ is a minimal $k$-counterexample, then it is $\psi$-critical.

\begin{lemma}\label{lemma-gen-psi}
If $(\g,z,\psi)$ is a minimal $k$-counterexample for a positive integer $k$, then the canvas $(\g,z)$ is $\psi$-critical.
\end{lemma}
\begin{proof}
Since $\psi$ does not extend to a nowhere-zero flow in $\g$, there exists a $\psi$-critical tip-respecting contraction $(\g',z)$ of $(\g,z)$.
Note that $(\g',z)$ is $k$-tall as witnessed by $\psi$ and tame by Corollary~\ref{cor-contrtame}, and thus $(\g',z)\in\GG_k$.
By Theorem~\ref{thm-deg}, every vertex other than $z$ has degree at most $\deg(z)-2\le k-1$ in $\g'$.  
By Observation~\ref{obs-inGk}, the canvas $(\g,z)$ is a $\GG_k$-expansion of $(\g',z)$.  Since $(\g,z)$ is not obtained
according to Theorem~\ref{thm-gen} (EXPA), it follows that $(\g',z)\not\prec (\g,z)$, and thus $(\g',z)=(\g,z)$.  Therefore, $(\g,z)$ is $\psi$-critical.
\end{proof}

Let us also note that in a tame flow-critical canvas $(\g,z)$, contracting any set of vertices not separated from $z$ by
an edge cut of size at most $\deg(z)-2$ results in a canvas where every tip preflow extends.

\begin{lemma}\label{lemma-Gk-contr}
Let $(\g,z)$ be a tame flow-critical canvas, and let $A\subseteq V(\g)\setminus \{z\}$ be a non-empty set of its vertices.
If $\deg(X) \ge\deg(z)-1$ for every set $X$ such that $A\subseteq X\subsetneq V(\g)\setminus \{z\}$,
then every tip preflow extends to a nowhere-zero flow in $(\g',z)=(\g / A,z)$. 
\end{lemma}
\begin{proof}
Note that $\deg_{\g'}(B)\ge 4+|\tau(B)|$ for every non-empty $B\subsetneq V(\g')\setminus\{z\}$
by the tameness of $(\g,z)$ and Lemma~\ref{lemma-cut}.
Let $x$ be the vertex of $(\g',z)$ to which $A$ is contracted;
then $\deg(X')\ge \deg(z)-1$ for every $X'\subseteq V(\g')\setminus\{z\}$
containing $x$ by the assumptions.  Theorem~\ref{cor-tall} implies that
every tip preflow extends to a nowhere-zero flow in $(\g',z)$.
\end{proof}

As mentioned above, we aim to extend Observation~\ref{obs-inGk} to subsets with $\deg(A)=k+1$, and to this end, we need a technical lemma. We start off with a definition.
\begin{definition}
\label{def:validflow} 
Let $(\g,z)$ be a canvas, let $\psi$ be a tip preflow, let $A$ be a subset of $V(\g)\setminus \{z\}$, and let $a$ be the vertex of $\g/A$
to which $A$ is contracted.  Let $x$ be a vertex of $A$.  We say that a nowhere-zero flow $\vec{G}$ in $\g/A$ is \emph{$(k,x,\psi,A)$-valid}
if
\begin{itemize}
\item $\vec{G}$ extends $\psi$,
\item there exist edges $e_1$ and $e_2$ with exactly one end in $A$ (i.e., incident in $\g/A$ with $a$) which $\vec{G}$ orients
in opposite directions (one of them to $a$ and the other one away from $a$), and moreover,
\item $e_1$ and $e_2$ are not incident with $x$, unless $\deg(z)=k+1$ and $\psi$ orients all
edges between $z$ and $V(\g)\setminus\{z,x\}$ in the same direction.
\end{itemize}
\end{definition}
\begin{lemma}\label{lemma-extori}
Let $k$ be a positive integer, $(\g,z)\in \GG_k$ a canvas, and $\psi$ a $k$-tallness-witnessing preflow in $(\g,z)$.
Suppose $A\subseteq V(\g)\setminus \{z\}$ satisfies $\deg(A)=k+1$ and let $x$ be a vertex of $A$.  If $\psi$ extends to a nowhere-zero flow in $\g/A$, then it also extends to a $(k, x,\psi,A)$-valid one.
\end{lemma}
\begin{proof}
We prove Lemma~\ref{lemma-extori} by induction on the number of vertices of $\g$.  
Suppose for a contradiction that there exists a set $A\subseteq V(\g)\setminus \{z\}$ such that $\deg(A)=k+1$, $x\in A$, and $\psi$ extends to a
nowhere-zero flow in $\g/A$, but no such flow is $(k, x,\psi,A)$-valid.
Let us choose such a set $A$ of maximal size.  Note that $A\neq V(\g)\setminus \{z\}$: Otherwise,
$\deg(z)=k+1$ and $\psi$ does not orient all edges between $z$ and $A$ in the same direction since it is $k$-tallness-witnessing,
and consequently $\psi$ would give a $(k, x,\psi,A)$-valid nowhere-zero flow in $\g/A$.

First, let us show that all edge cuts separating $A$ from $z$ are large.
\begin{claim*}
\label{subclaim:insideGKcontract}
We have that $\deg(A')\ge k+2$ for every $A'\subsetneq V(\g)\setminus \{z\}$ such that $A\subsetneq A'$.
\end{claim*}
\begin{subproof}
Suppose that there exists $A'\subsetneq V(\g)\setminus \{z\}$ such that $A\subsetneq A'$
and $\deg(A') \le k+1$, and let us choose such a set $A'$ of minimum size.
Since $\psi$ extends to a nowhere-zero flow in $\g/A$, it also extends to a nowhere-zero flow $\vec{G}_0$ in $\g/A'$.
Let $a'$ be the vertex resulting from the contraction of $A'$.
If $\deg(A')=k+1$, then by the maximality of $A$, we can assume that $\vec{G}_0$ is $(k,x,\psi,A')$-valid;
and in particular, $\vec{G}_0$ does not orient all edges between $V(\g)\setminus A'$ and $A'\setminus \{x\}$ in the
same direction, unless $\deg(z)=k+1$ and $\psi$ orients all edges between $z$ and $V(\g)\setminus\{z,x\}$ in the same direction.

Let $(\g',b)=(\g,z)\restriction A'$ and let $\psi'$ be the tip preflow in which the edges incident with $b$
are directed the way $\vec{G}_0$ orients the corresponding edges incident with $a'$.  Since $\psi$ does not
extend to a nowhere-zero flow in $\g$, we see that $\psi'$ does not extend to a nowhere-zero flow in $\g'$.
Moreover, $\deg(b)\le k+1$, and if $\deg(b)=k+1$, then $\psi'$ does not orient all edges incident with $b$ in the same way.
Therefore, $\psi'$ is a $k$-tallness-witnessing preflow in $(\g',b)$, and $(\g',b)\in\GG_k$.

Since $A'$ is a minimal superset of $A$ with $\deg(A')\le k+1$, and since $\deg(A)=k+1$,
every set $X$ such that $A\subseteq X\subsetneq A'=V(\g')\setminus\{b\}$ satisfies $\deg(X)\ge k+1\ge \deg(A')=\deg(b)$.
By Lemma~\ref{lemma-Gk-contr}, $\psi'$ extends to a nowhere-zero flow $\vec{G}_1$ in $(\g'/A,b)$.
Note that $|V(\g')|<|V(\g)|$, and thus by the induction hypothesis, we can assume that $\vec{G}_1$
is $(k,x,\psi',A)$-valid.

The combination of $\vec{G}_0$ and $\vec{G}_1$ gives a $(k,x,\psi,A)$-valid nowhere-zero flow in $(\g /A,z)$:
Note that $\vec{G}_1$ can only orient all edges between $V(\g')\setminus A$ and $A\setminus \{x\}$
in the same direction if $\deg(b)=\deg(A')=k+1$ and $\psi'$ orients all edges between $b$ and $V(\g')\setminus \{b,x\}=A'\setminus\{x\}$ in the same direction,
which by the choice of $\vec{G}_0$ can only happen when $\deg(z)=k+1$ and $\psi$ orients all edges between $z$ and $V(\g)\setminus\{z,x\}$ in the same direction.
This contradicts the existence~of~$A$. 
\end{subproof}
Since $\deg(A)=k+1$ and $\deg(z)\le k+1$, Theorem~\ref{thm-deg} implies $|A|\ge 2$.  Moreover, since $A\neq V(\g-z)$,
we have $C=V(\g)\setminus(A\cup\{z\})\neq\emptyset$ and $|V(\g)|\ge 4$.  Therefore, Lemma~\ref{lemma-2con} implies that $\g-\{z,x\}$ is connected,
and thus there exists an edge $e_1$ with one end $u_1\in C$ and the other end in $A\setminus \{x\}$.
\begin{itemize}
\item If $\deg(z)=k+1$ and $\psi$ directs all edges between $z$ and $C$ in one direction
(there exists at least one such edge, since $\deg(A)=k+1=\deg(z)$ and $e_1$ is not incident with $z$), then since $\psi$ is a $k$-tallness-witnessing preflow,
there exists an edge between $z$ and $A$ directed in the opposite direction; let $e_2$ be such an edge, not incident with $x$ if possible, and let $u_2=z$.
\item Otherwise, since $\deg(x)\le \deg(z)-2\le k-1$ by Theorem~\ref{thm-deg} and $\deg(A)=k+1$, there exists an edge $e_2\neq e_1$ between $u_2\in C \cup \{z\}$
and $A\setminus\{x\}$.  If possible, choose $e_2$ so that $u_2\neq u_1$.
\end{itemize}

Let $\g_1$ be obtained from $\g/A$ by splitting off $e_1$ with $e_2$, and if $u_1\neq u_2$, then let $e$ be the resulting edge.  Note that $\psi$ can be naturally interpreted
as a tip preflow $\psi_1$ in $(\g_1,z)$.  If $\psi_1$ extended to a nowhere-zero flow in $\g_1$, then, directing $e_1$ and $e_2$ according
to $e$ (or arbitrarily in the opposite directions if $u_1=u_2$) would give a $(k, x,\psi,A)$-valid nowhere-zero flow in $\g/A$
(note that $e_2$ can be chosen to be incident with $x$ only if $\deg(z)=k+1$ and $\psi$ directs all edges not incident with $x$ in the same direction);
this would contradict the choice of $A$.  Therefore, $\psi_1$ does not extend to a nowhere-zero flow in $\g_1$, and thus
there exists a $\psi_1$-critical tip-respecting contraction $(\g'_1,z)$ of $(\g_1,z)$.
Let $a'$ be the vertex of $\g'_1$ into which we contracted $A$, and let $A'\supseteq A$ be the set of vertices of $\g$ contracted into $a'$.

In the case that $u_1=u_2$, consider the set $U$ of vertices of $\g_1$ containing $u_1$ and contracted to a vertex $u$ of $\g'_1$,
and suppose that $u\neq a'$.  If $|U|\ge 2$, then $\deg(u)\ge \deg(U)-2\ge 4+|\tau(U)|=4+|\tau(u)|$ by Corollary~\ref{cor-contrtame}.
If $U=\{u_1\}$, then by Lemma~\ref{lemma-simple}, there is at most one edge between $u_1$ and $a'$ in $\g'_1$, and thus
there are at most three edges between $u_1$ and $A$ in $\g$.  We only chose $e_2$ incident with $u_1$ because every edge with exactly one end in $A$
is incident with $u_1$ or $x$, and since $\deg(A)=k+1$, at least $k-2$ of these edges are incident with $x$.  By Theorem~\ref{thm:tall},
we have $\deg(x)\le \deg(z)-2\le k-1$, and thus there is at most one edge between $x$ and $A\setminus\{x\}$.  Consequently,
$\deg(A\setminus \{x\})\le 4$, and Corollary~\ref{cor-mainlov} implies that $|A\setminus\{x\}|=1$.  However, since $e_1$ and $e_2$
both end in $A\setminus \{x\}$, this contradicts Lemma~\ref{lemma-simple} for $(\g,z)$.
Therefore, if $u\neq a'$, then $\deg(u)\ge 4+|\tau(U)|=4+|\tau(u)|$.

Using Corollary~\ref{cor-contrtame}, it is now easy to see that $\deg_{\g'_1}(v)\ge 4+|\tau(v)|$ holds for every vertex $v\in V(\g'_1)\setminus\{z,a'\}$. Hence $(\g',z,a',\psi_{1})$ is a critical tame easel. We now argue that $(\g',z,a',\psi_{1})$ is tall. 
We consider two cases:
\begin{itemize}
\item If $A'\neq A$, then $\deg_{\g}(A')\ge k+2$ by the above claim, and thus $\deg_{\g'_1}(a') \ge k\ge \deg(z)-1$, and the easel $(\g'_1,z,a',\psi_1)$ is tall.
\item If $A' = A$, then $\deg_{\g'_{1}}(a') = \deg_{\g}(A) -2 = k-1$. 
If $\deg(z) \leq k$, this again implies that the easel $(\g'_{1},z,a',\psi_{1})$ is tall.
Therefore, we can assume have $\deg(z)=k+1$.
Since $\deg(A) = \deg(z)$ and $e_{1}$
is an edge between $A$ and $C$, there exists at least one edge $e_{0}$ of $\g$
between $z$ and $C$.
\begin{itemize}
\item If $\psi$ does not direct all edges between $z$ and $C$ in
the same way, then $\psi_{1}$ also does not direct the corresponding edges
between $z$ and $V(\g'_{1}) \setminus \{a',z\}$ in the same way. 
Therefore, the easel $(\g'_1,z,a',\psi_1)$ is tall.

\item If $\psi$ directs all edges between $z$ and $C$ in the same way, then since
$\deg(z)=k+1$ and $\psi$ is $k$-tallness-witnessing, $\psi$ directs an edge
between $z$ and $A$ in the opposite way, and such an edge was chosen as $e_2$.
But then $e$ is an edge of $\g'_1$ between $z$ and $V(\g'_1)\setminus\{a',z\}$
directed opposite to $e_0$.
Hence, we again conclude that the easel $(\g'_1,z,a',\psi_1)$ is tall.
\end{itemize}
\end{itemize}
However, this contradicts Theorem~\ref{thm:tall}.
\end{proof}

We are now ready to strengthen Observation~\ref{obs-inGk}.

\begin{lemma}\label{lemma-gen-sepx}
Let $(\g,z)\in \GG_k$ be a canvas for a positive integer $k$, let $\psi$ be a $k$-tallness-witnessing preflow in $(\g,z)$,
and suppose that $(\g,z)$ is $\psi$-critical.  If a set $A\subseteq V(\g)\setminus \{z\}$ of size at least two satisfies $\deg(A)\le k+1$,
then $(\g,z)\restriction A\in \GG_k$.
\end{lemma}
\begin{proof}
By Observation~\ref{obs-inGk}, we can assume that $\deg(A)=k+1$.
The canvas $(\g_1,b)=(\g,z)\restriction A$ is flow-critical by Observation~\ref{obs-subcrit} and it is clearly tame.
Hence, it suffices to show that it is $k$-tall.

Since $(\g,z)$ is $\psi$-critical, $\psi$ extends to a nowhere-zero $\vec{G}_1$ flow in $(\g/A,z)$.  Let $a$ be the vertex arising from the contraction of $A$.
By Lemma~\ref{lemma-extori} (with $x$ being any vertex of $A$), we can assume that $\vec{G}_1$ does not direct all edges incident with $a$ in the same way.
The restriction of $\vec{G}_1$ to the edges incident with $a$ can be interpreted as preflow $\psi_1$
around $b$ in $\g_1$. Since $\psi$ does not extend to a nowhere-zero flow in $\g$, $\psi_1$ does not extend to a nowhere-zero flow in $\g_1$,
and thus $\psi_1$ is $k$-tallness-witnessing in $(\g_1,b)$.  Therefore, $(\g_1,b)$ is indeed $k$-tall.
\end{proof}

Note that by Lemma~\ref{lemma-gen-psi}, the previous lemma applies when $(\g,z,\psi)$ is a minimal $k$-counter\-exam\-ple.

\subsection{Step 2: Minimal counterexamples are $z$-homogeneous}\label{subsec:z-homog}
We start this subsection by eliminating vertices of degree four.
\begin{lemma}\label{lemma-gen-no4}
If $(\g,z,\psi)$ is a minimal $k$-counterexample for a positive integer $k$, then $\g$ has minimum degree at least five.
\end{lemma}
\begin{proof}
Suppose for a contradiction that $\g$ has a vertex $v$ of degree at most four.  Note that $v\neq z$ by Corollary~\ref{cor-mainlov}, and by the tameness of $(\g,z)$ we have further that $\deg(v)=4$ and $v$ has zero boundary.
By Lemma~\ref{lemma-gen-psi}, the canvas $(\g, z)$ is $\psi$-critical.  Let $C=V(\g)\setminus\{v,z\}$; since $|V(\g)|\ge 4$, we have $|C|\ge 2$.
Lemma~\ref{lemma-2con} implies that there exist two edges $e_1$ and $e_3$ between $v$ and vertices in $C$. 
Label the remaining two edges incident with $v$ by $e_2$ and $e_4$ arbitrarily.  For $i\in\{1,\ldots,4\}$, let $v_i$ be the end of $e_i$ different from $v$,
and let $N=\{v_1,\ldots, v_4\}$. By Lemma \ref{lemma-simple}, $v_1$ and $v_3$ are distinct.


Let $\g_1$ be the $\Z$-bordered graph obtained from $\g$ by splitting off $e_1$ with $e_2$ (giving an edge $e'_1$) and $e_3$ with $e_4$
(giving an edge $e'_3$) and deleting the now isolated vertex $v$.
The preflow $\psi$ naturally corresponds to a tip preflow $\psi_1$ in $(\g_1,z)$, and $\psi_1$ does not extend to a nowhere-zero flow in $\g_1$.
Hence, $(\g_1,z)$ has a tip-respecting $\psi_1$-critical contraction $(\g',z)$. Let $\PP'$ be the tip-respecting partition of $V(\g_1)$
such that $\g'=\g_1/\PP'$.
\begin{claim*}
The canvas $(\g',z)$ is tame, $\psi_{1}$ is a $k$-tall-witnessing preflow and $(\g',z) \prec (\g,z)$. 
\end{claim*}

\begin{subproof}
We first argue that $(\g',z)$ is tame. Consider any $A'\in \PP'$ other than $\{z\}$, and let $a$ be the corresponding vertex of $\g'$.
Let $A=A'\cup\{v\}$ if $|A'\cap N|\ge 3$ and $A=A'$ otherwise.
If $|A|\ge 2$, then Corollary~\ref{cor-contrtame} implies $\deg(a)\ge \deg_{\g}(A)-2\ge 4+|\tau(A)|=4+|\tau(a)|$.
If $|A|=1$, then by Lemma~\ref{lemma-simple} at most one of the edges $e_1$, \ldots, $e_4$ has an end in $A=A'$,
and $\deg(a)=\deg_{\g}(A)\ge 4+|\tau(A)|=4+|\tau(a)|$ since $(\g,z)$ is tame.  Therefore, the canvas $(\g',z)$ is tame.

Now observe that if $\deg_{\g'}(z)=k+1$, then $\deg_{\g}(z)=k+1$ and $\psi$ directs two edges incident with $z$ in opposite ways,
and thus so does $\psi_1$.  Therefore, $\psi_1$ is a $k$-tallness-witnessing preflow for $(\g', z)$, and $(\g',z)\in \GG_k$. Lastly, we also have $o(\g',z)<o(\g,z)$, and thus $(\g',z)\prec (\g,z)$.
\end{subproof}
 For $i\in \{1,2\}$, define $x_i:=e'_{2i-1}$ if this edge is present in $\g'$, and otherwise, let $x_i$ be the vertex into which the part $X_i\in\PP'$ containing $v_{2i-1}$ and $v_{2i}$ was contracted.
\begin{itemize}
\item If $x_1\neq x_2$, then let $(\g_0,z)$ be the $2$-alteration of $(\g',z)$ on $x_1$ and $x_2$, with the newly added vertex labelled $v$,
and let $\PP=\PP'\cup\{\{v\}\}$.
\item If $x_1=x_2$, then note that $x_1$ is a vertex other than $z$; let $(\g_0,z)=(\g',z)$ and let $\PP=(\PP'\setminus\{X_1\})\cup\{X_1\cup \{v\}\}$.
\end{itemize}
Observe that in either case, we have $(\g_0,z)=(\g / \PP,z)$. To finish, it suffices to show that $(\g,z)$ is a $\GG_{k}$-expansion of $(\g_{0},z)$.  Consider any part $P\in\PP$ of size at least two, and let $p$ be the corresponding vertex of $\g_0$.
Since $(\g',z)\in \GG_k$, Theorem~\ref{thm-deg} implies $\deg_{\g'}(p)\le k-1$, and thus $\deg(P)\le \deg_{\g'}(p)+2\le k+1$.
By Lemma~\ref{lemma-gen-sepx}, we have $(\g,z)\restriction P\in \GG_k$.  
Therefore, $(\g,z)$ is $\GG_k$-expansion of $(\g_0,z)$, and thus $(\g,z)$ is obtained as in Theorem~\ref{thm-gen} (EXPA), contradicting our choice.
\end{proof}

Next, let us get rid of mixed edges.

\begin{lemma}\label{lemma-gen-norede2}
If $(\g,z,\psi)$ is a minimal $k$-counterexample for a positive integer $k$, then
there is no mixed edge $uv\in E(\g-z)$.
\end{lemma}
\begin{proof}
Let $\g=(G,\beta)$.  Recall that the canvas $(\g,z)$ is $\psi$-critical by Lemma~\ref{lemma-gen-psi}.
Suppose for a contradiction that $e=uv\in E(\g-z)$ is a mixed edge, say with $v$ in-friendly and
$u$ out-friendly.  Let $\g_1=(G-e,\beta')$, where $\beta'(y)=\beta(y)$ for $y\in V(G)\setminus\{u,v\}$, $\beta'(u)=\beta(u)-1$, and $\beta'(v)=\beta(v)+1$.
Lemmas~\ref{lemma-conn} and \ref{lemma-2con} imply that $G-e$ is connected, and thus $\beta'$ is a $\Z$-boundary for $G-e$.
If $\g_1$ had a nowhere-zero flow extending $\psi$, it would give a nowhere-zero flow in $\g$ extending $\psi$
by directing the edge $e$ towards $v$.  Hence, $\psi$ does not extend to a nowhere-zero flow in $\g_1$, and thus $(\g_1,z)$ has a tip-respecting
$\psi$-critical contraction $(\g',z)$. Let $\g'=\g_1/\PP$.  Note that $o(\g',z)<o(\g,z)$, and thus $(\g',z)\prec (\g,z)$.

As in the proof of Lemma~\ref{lemma-norede}, observe that the assumptions on in-friendliness and out-friendliness of $u$ and $v$
together with Lemma~\ref{lemma-gen-no4} imply that $(\g_1,z)$ is tame, and together with Corollary~\ref{cor-contrtame},
it follows that $(\g',z)$ is tame.  Moreover, $(\g',z)$ is $k$-tall, since $\psi$ is a $k$-tallness-witnessing preflow around $z$ in $\g'$.
Hence $(\g',z)\in\GG_k$.

If $u$ and $v$ are contained in the same part of $\PP$, then let $(\g_0,z)=(\g',z)$.  Otherwise, let $u'$ and $v'$ be the vertices of $\g'$
into which $u$ and $v$ were contracted, and let $(\g_0,z)$ be the $1$-alteration of $(\g',z)$ obtained by adding the edge $u'v'$,
increasing the boundary at $u'$ by one, and decreasing it at $v'$ by one.  Note that in either case, $(\g_0,z)=(\g / \PP,z)$.

For every vertex $y\in V(\g_0)\setminus\{z\}$, if $A$ is the set of vertices of $\g$ contracted into $y$ and $|A|\ge 2$,
then $\deg(A)\le \deg_{\g'}(y)+1\le \deg(z)-1\le k$ by Theorem~\ref{thm-deg}, and thus $(\g,z)\restriction A$
belongs to $\GG_k$ by Observation~\ref{obs-inGk}.
Therefore, $(\g,z)$ is a $\GG_k$-expansion of $(\g_0,z)$ and it is obtained as in Theorem~\ref{thm-gen} (EXPA) or (EXPB), contradicting our choice. 
\end{proof}

Consequently, $(\g,z)$ is $z$-homogeneous, and by Observation~\ref{obs-allplusminus},
either $\tau(v)>0$ for every $v\in V(\g)\setminus\{z\}$, or $\tau(v)<0$ for every $v\in V(\g)\setminus\{z\}$. We  now argue that $\psi$ orients all edges towards $z$ or away from $z$.

\begin{lemma}\label{lemma-gen-norede1}
If $(\g,z,\psi)$ is a minimal $k$-counterexample for a positive integer $k$, then
either $\tau(v)>0$ for every $v\in V(\g)\setminus\{z\}$ and $\psi$ directs all edges away from $z$,
or $\tau(v)<0$ for every $v\in V(\g)\setminus\{z\}$ and $\psi$ directs all edges towards $z$.
\end{lemma}
\begin{proof}
Let $\g=(G,\beta)$.  Recall that the canvas $(\g,z)$ is $\psi$-critical by Lemma~\ref{lemma-gen-psi}.
By Lemma~\ref{lemma-gen-norede2}, we can by symmetry assume that $\tau(v)>0$ for every $v\in V(\g)\setminus\{z\}$.
Suppose for a contradiction that $\psi$ directs an edge $e=uz$ of $\g$ towards $z$.
Let $\g'=(G-e,\beta')$, where $\beta'(y)=\beta(y)$ for $y\in V(G)\setminus\{z,v\}$, $\beta'(z)=\beta(z)+1$, and $\beta'(u)=\beta(u)-1$.
Let $\psi'$ be obtained from $\psi$ by removing the edge $e$. As in the proof of Lemma~\ref{lemma-sameor}, we argue that $(\g',z)$ is a $\psi'$-critical tame canvas, and since $\deg_{\g'}(z)\le \deg_{\g}(z)-1\le k$, we have $(\g',z)\in \GG_k$.
Note that $(\g,z)$ is a tip-alteration of $(\g',z)$.
Since $o(\g',z)=o(\g,z)$, the canvas $(\g,z)$ is $z$-homogeneous, and $\deg_{\g'}(z)<\deg_{\g}(z)$,
we have $(\g',z)\prec (\g,z)$.  Therefore, $(\g,z)$ is obtained as in Theorem~\ref{thm-gen} (ADD), which is a contradiction.
\end{proof}
\subsection{Step 3: There are no $k$-minimal counterexamples}\label{subsec:finish}
It is now easy to finish the proof of the canvas generation theorem.

\begin{proof}[Proof of Theorem~\ref{thm-gen}]
Suppose for a contradiction that there exists a minimal $k$-counterexample $(\g,z,\psi)$ for a positive integer $k$.
By Lemma~\ref{lemma-gen-norede1} and symmetry, we can assume that $\tau(v)>0$ for every $v\in V(\g)\setminus\{z\}$ and $\psi$ directs all edges away from $z$.
Since $\psi$ is $k$-tallness-witnessing, it follows that $\deg(z)\le k$.  By Corollary~\ref{cor-mainlov}, we have $\deg(z)\ge 6$.

Consider any edge $e=zv$ incident with $z$ and let $e_0\neq e$ be another edge between $z$ and $V(\g)\setminus \{z\}$.
Let $\g'$ be the $\Z$-boundaried graph obtained from $\g$ by adding
an edge $e'$ parallel to $e$ and let $\psi'$ be the preflow around $z$ in $\g_1$ obtained from $\psi$ by
directing $e$ and $e'$ towards $z$.  As in the proof of Theorem~\ref{thm:tall}, we argue that $(\g',z)$ is $\psi'$-critical and tame.
Since $\deg_{\g'}(z)=\deg_{\g}(z)+1\le k+1$ and $\psi'$ directs $e$ and $e_0$ in the opposite ways, we have $(\g',z)\in\GG_k$.
Note that $(\g,z)$ is a tip-reduction of $(\g',z)$.

By Lemma~\ref{lemma-2con}, $v$ has a neighbour $u \neq z$.  Since we changed the parity of the degree of $v$, we have $\tau_{\g'}(v)<0$, while $\tau_{\g'}(u)=\tau_{\g}(u)>0$.
Hence, $(\g',z)$ is not $z$-homogeneous.  Since $o(\g',z)=o(\g,z)$ and $(\g,z)$ is $z$-homogeneous, it follows that $(\g',z)\prec (\g,z)$.
Therefore, $(\g,z)$ is obtained as in Theorem~\ref{thm-gen} (REM), which is a contradiction.
\end{proof}

\section{Generating $(k,r)$-tall easels}\label{sec:geneasels}

It might seem that Theorem~\ref{thm-gen} is ideally suited for inductive proofs of statements such as Theorem~\ref{thm-degbetter}.
This unfortunately is not the case: The step (EXPB) turns out to be problematic, as it allows for the possibility that $(\g,z)$ is a 1-alteration of a canvas $(\g',z)\prec (\g,z)$. In this case, the underlying graph of $\g'$ is obtained from the underlying graph of $\g$ by deleting an edge, and thus $\g'$ does not necessarily contain a vertex of degree $\deg(z)-2$ or $\deg(z)-3$.  Thus, in order to prove Theorem~\ref{thm-degbetter}, we need a more technical variation on Theorem~\ref{thm-gen}.

We require the following definition:

\begin{definition}
\label{def:krtall}
Let $k$ be a positive integer, $r$ an integer, and $(\g,z)$ a canvas with $\deg(z) \le k+1$.  We say that a pair $(x,\psi)$ is a \emph{witness of $(k,r)$-tallness} of $(\g,z)$ if
\begin{itemize}
\item $\psi$ is a $k$-tallness-witnessing preflow,
\item $x\neq z$ is a vertex of $\g$ of degree at least $k-2-r$, and
\item if $\deg(x)=k-2-r$ and $\deg(z)=k+1$, then $\psi$ does not orient all edges between $z$ and $V(\g)\setminus\{z,x\}$ in the same direction
(all towards $z$ or all away from $z$).
\end{itemize}
\end{definition}
Let $\GG_{k,r}$ be the class of all easels $(\g,z,x,\psi)$ such that $(\g,z)$ is a tame flow-critical canvas with $\deg(z)\le k+1$ and $(x,\psi)$ is a witness of $(k,r)$-tallness.
Note that $(\g,z,x,\psi)\in \GG_{k,r}$ implies $(\g,z)\in\GG_k$. 
Moreover, if $(x, \psi)$ is a witness of $(k,-1)$-tallness, then $(\g,z,x,\psi)$ is a tall easel, and thus Theorem~\ref{thm:tall} implies that $\GG_{k,r}=\emptyset$ for $r\le 0$.

Compared to Theorem~\ref{thm-gen}, in the result on generation of $(k,r)$-tall canvases, we need an additional operation, see Figure~\ref{fig-xalter}.

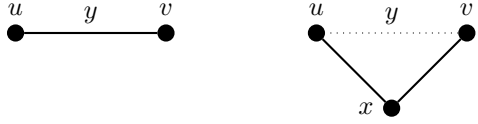
\begin{figure}
\begin{center}
\begin{tikzpicture}
    \node[blackvertexv2] at (0,0) (u) [label = above:$u$] {};
    \node[blackvertexv2] at (2,0) (v) [label = above:$v$] {};
    \draw[thick, black] (v) to node[above]{$y$}(u); 
    \begin{scope}[xshift = 4cm]
    \node[blackvertexv2] at (0,0) (u) [label = above:$u$] {};
    \node[blackvertexv2] at (2,0) (v) [label = above:$v$] {};
    \node[blackvertexv2] at (1,-1) (x) [label = left:$x$] {};
    \draw[dotted, black] (v) to node[above]{$y$}(u); 
    \draw[thick, black] (x)--(u);
    \draw[thick,black] (x)--(v);
    \end{scope}
\end{tikzpicture}
\caption{An $x$-alteration where $y$ is an edge.}\label{fig-xalter}
\end{center}
\end{figure}
\begin{definition}
Suppose $(\g',z)$ is a canvas, $x\neq z$ is a vertex of $\g'$, and $y$ is either a vertex $v\in V(\g-\{x,z\})$, or an edge $uv\in E(\g-\{x,z\})$; in the former case, let $u=v$.
A canvas $(\g_0,z)$ is the \emph{$x$-alteration} of $\g'$ on $y$ if $\g_0$ is obtained from $\g'$ by adding (possibly parallel) edges $ux$ and $vx$  and if $y$ is an edge, deleting it.
\end{definition}

Moreover, we need the following more precise variant of the expansion operation.
\begin{definition}
Let $\GG$ be a class of canvases and $\mathcal{H}$ be a class of easels. 
If $(\g,z)$ is a canvas containing a vertex $x\neq z$, $(\g',z)$ is a canvas with a vertex $x'\neq z$, and $Y$ is a set of vertices and edges of $\g'$, we say that
$(\g,z)$ is a \emph{$(\GG,x\to x',\mathcal{H},Y)$-expansion} of $(\g',z)$ if $(\g,z)$ is a $\GG$-expansion of $(\g',z)$ and moreover the following points hold. If $\PP$ is the tip-respecting partition such that $\g'=\g/\PP$, then
\begin{itemize}
\item the part $P\in \PP$ containing $x$ is contracted into the vertex $x'$; moreover,
\item letting $(\g_1,b)=(\g,z)\restriction P$, if $|P|>1$ then there exists a tip preflow $\psi$ such that $(\g_1,b,x,\psi)\in \GG'$, and
\item for every \emph{vertex} $y\in Y$ different from $z$, the part of $\PP$ contracted into $y$ has size at least two.
\end{itemize}
\end{definition}

Let us remark that the edges which may appear in $Y$ are ignored in this definition; this is just to simplify the notation,
see e.g. the operation (EXPA) below, where we would otherwise have to handle the case that $y_1$ or $y_2$ is an edge.
We are now ready to state the generation theorem.

\begin{theorem}\label{thm-genladeg}
Let $k$ and $r$ be integers.  For every easel $(\g,z,x,\psi)\in\GG_{k,r}$, at least one of the following claims holds: 
\begin{itemize}
\item (SMALL) $|V(\g)|=3$; or,

\item $\GG_{k,r}$ contains an easel $(\g',z,x',\psi')$ such that $(\g',z,x')\prec (\g,z,x)$ and 
\begin{itemize}[align=left, leftmargin=3mm, labelwidth=1.5cm,itemindent=1.5cm,]
\item[(EXP)] $(\g,z)$ is a $(\GG_k,x\to x',\GG_{k,r},\emptyset)$-expansion of $(\g',z)$; or,
\item[(EXPA)] $(\g,z)$ is a $(\GG_k,x\to x',\GG_{k,r},\{y_1,y_2\})$-expansion of a 2-alteration of $(\g',z)$ on some vertices or edges $y_1$ and $y_2$; or,
\item[(EXPX)] $\deg(z)=k+1$, $\deg(x)=k-2-r$, $\deg(x') \ge\deg(x)$, and $(\g,z)$ is a $(\GG_k,x\to x',\GG_{k,r},\{y,x'\})$-expansion of an $x'$-alteration of $(\g',z)$ on some vertex or edge $y$; or,
\item[(EXPB)] $(\g,z)$ is a $(\GG_k,x\to x',\GG_{k,r},\emptyset)$-expansion of a 1-alteration of $(\g',z)$; or,
\item[(ADD)] $(\g,z)$ is a tip-alteration of $(\g',z)$ at a vertex different from $x=x'$; or,
\end{itemize}
\item (REM) $\GG_{k,r}\cup \GG_{k+1,r}$ contains an easel $(\g',z,x,\psi')$ such that $(\g',z,x')\prec (\g,z,x)$ and $(\g,z)$ is a tip-reduction of $(\g',z)$ at a vertex different from $x=x'$.
\end{itemize}
Moreover:
\begin{itemize}
\item If $\g$ contains a vertex other than $x$ of degree four, then (SMALL), (EXP), (EXPA), or (EXPX) holds.
\item If $(\g,z)$ is not $\psi$-critical, then (EXP) holds.
\end{itemize}
\end{theorem}
The fact that $(\g',z,x,\psi)$ may belong to $\GG_{k+1,r}$ rather than to $\GG_{k,r}$ in (REM) might seem somewhat problematic, but it does not pose any particular difficulty
for induction purposes, since the partial ordering $\prec$ does not have any infinite decreasing chains.

Let us remark that while Theorem~\ref{thm-genladeg} suffices to prove Theorem~\ref{thm-degbetter}, it does not seem to be strong enough
to attempt a proof of Conjecture~\ref{conj-censusplus}, since in (EXPB), we lose control over the degrees.
The proof of Theorem~\ref{thm-genladeg} is quite similar to the proof of Theorem~\ref{thm-gen}, with several technical complications. As before, we proceed in three steps, the first of which is to show there are no small cuts around $x$, the second of which is to argue that a minimal counterexample is $x$-homogeneous, and the final of which is to argue that there are no counterexamples. Let us define what our minimal counterexample is:

\begin{definition}
A \emph{minimal $(k,r)$-counterexample} is an easel $(\g,z,x,\psi)\in\GG_{k,r}$ not satisfying the conclusion of Theorem~\ref{thm-genladeg} with $(\g,z,x)$ minimal in the $\prec$ ordering. 
\end{definition}

Note that if $(\g,z,x,\psi)$ is a minimal $(k,r)$-counterexample, then clearly $|V(\g)|\ge 4$, as otherwise $(\g,z)$ satisfies (SMALL).

\subsection{Step 1: There are no small cuts around $x$}\label{subsec:easelcuts}
This subsection builds towards Lemma \ref{lemma-aroundx}, which shows there are no small cuts around $x$. We start off by arguing that minimal $(k,r)$-counterexamples are $\psi$-critical.
\begin{lemma}\label{lemma-genladeg-psi}
If $(\g,z,x,\psi)$ is a minimal $(k,r)$-counterexample for non-negative integers $k$ and $r$, then the canvas $(\g,z)$ is $\psi$-critical.
\end{lemma}
\begin{proof}
Since $\psi$ does not extend to a nowhere-zero flow in $\g$, there exists a $\psi$-critical tip-respecting contraction $(\g',z)$ of $(\g,z)$.
Note that $(\g',z)$ is tame by Corollary~\ref{cor-contrtame}, and $\psi$ is a $k$-tallness-witnessing preflow in $(\g',z)$.
By Theorem~\ref{thm-deg}, every vertex other than $z$ has degree at most $\deg(z)-2\le k-1$ in $\g'$.  
Let $x'$ be the vertex of $\g'$ to which we contracted the set $X\subseteq V(\g)\setminus\{z\}$ containing $x$.
 If $|X|>1$, let $(\g_1,b)=(\g,z)\restriction X$ and $\psi'$ be any tip preflow that does not extend to a nowhere-zero flow in $\g_1$.
Since $\deg(b)\le k-1$, the easel $(\g_1,b,x,\psi')$ belongs to $\GG_{k,r}$.
By Observation~\ref{obs-inGk}, the canvas $(\g,z)$ is a $(\GG_k,x\to x',\GG_{k,r},\emptyset)$-expansion of $(\g',z)$.  

Let us now argue that $(x',\psi)$ is a witness of $(k,r)$-tallness of $(\g',z)$.  To this end, if $|X|\ge 2$, then Theorem~\ref{thm-deg} applied to $(\g,z)\restriction X$ implies that $\deg(x')=\deg(X)\ge \deg(x)+2>k-2-r$.
If $X=\{x\}$, then $\deg(x')=\deg(x)\ge k-2-r$ and the edges between $z$ and $x'$ in $\g'$ are exactly the same as those between $z$ and $x$ in $\g$.
In either case, we conclude that $(x',\psi)$ is a witness of $(k,r)$-tallness of $(\g',z)$, and $(\g,z,x',\psi)\in \GG_{k,r}$.

Since $(\g,z,x,\psi)$ does not satisfy Theorem~\ref{thm-genladeg}(EXP), it follows that $(\g',z,x')\not\prec (\g,z,x)$, and thus $(\g',z)=(\g,z)$.
Therefore, the canvas $(\g,z)$ is $\psi$-critical.
\end{proof}

We will need the following analogue of Lemma~\ref{lemma-gen-sepx}.
\begin{lemma}\label{lemma-genladeg-sepx}
Let $k$ and $r$ be non-negative integers and let $(\g,z,x,\psi)\in \GG_{k,r}$ be an easel such that the canvas $(\g,z)$ is $\psi$-critical.
If a set $A\subseteq V(\g)\setminus \{z\}$ of size at least two satisfies $\deg(A)\le k+1$ and $x\in A$,
then letting $(\g_1,b)=(\g,z)\restriction A$, there exists a tip preflow $\psi_1$ such that $(\g_1,b,x,\psi_1)\in \GG_{k,r}$.
\end{lemma}
\begin{proof}
The claim follows from Observation~\ref{obs-inGk} when $\deg(A)\le k$, and thus suppose that $\deg(A)=k+1$.
We have $(\g_1,b)\in \GG_k$ by Lemma~\ref{lemma-gen-sepx}.

Since $(\g,z)$ is $\psi$-critical, the tip preflow $\psi$ extends to a nowhere-zero flow $\vec{G}_1$ in $(\g/A,z)$.  Let $a$ be the vertex arising from the contraction of $A$.
By Lemma~\ref{lemma-extori}, we can assume that $\vec{G}_1$ is $(k,x,\psi,A)$-valid (recall Definition \ref{def:validflow}).  Note that since $(\g,z,x,\psi)\in \GG_{k,r}$, if $\deg(z)=k+1$
and $\deg(x)=k-2-r$, then $\psi$ does not direct all edges between $z$ and $V(\g-\{z,x\})$ the same way, and thus $\vec{G}_1$ directs two edges not incident with $x$ in opposite directions.
The restriction of $\vec{G}_1$ to the edges incident with $a$ can be interpreted as preflow $\psi_1$
around $b$ in $\g_1$. Since $\psi$ does not extend to a nowhere-zero flow in $\g$, $\psi_1$ does not extend to a nowhere-zero flow in $\g_1$,
and thus $(x,\psi_1)$ is a witness of $(k,r)$-tallness of $(\g_1,b)$.  Therefore, $(\g_1,b,x,\psi_1)\in \GG_{k,r}$.
\end{proof}

Let us remark that by Lemma~\ref{lemma-genladeg-psi}, Lemma~\ref{lemma-genladeg-sepx} applies to minimal $(k,r)$-counterexamples.
Next, it will be convenient to restrict small cuts around $x$.
\begin{lemma}\label{lemma-aroundx}
Let $k$ and $r$ be non-negative integers.  If $(\g,z,x,\psi)$ is a minimal $(k,r)$-counter\-exam\-ple, $\deg(z)=k+1$, $\deg(x)=k-2-r$,
and $X$ is a set of vertices of $\g$ such that $\{x\}\subsetneq X\subsetneq V(\g)\setminus \{z\}$,
then $\deg(X) > \deg(x)+2$.
\end{lemma}
\begin{proof}
Suppose for a contradiction that there exists such a set $X$ with $\deg(X)\le \deg(x)+2$, and let us choose one of maximal size.
Theorem~\ref{thm-deg} applied to $(\g,z)\restriction X$ implies that $\deg(X)=\deg(x)+2$.  Let $C=V(\g)\setminus (X\cup\{z\})$.
By Lemma~\ref{lemma-2con}, there exists an edge $e_1$ between $C$ and $X\setminus\{x\}$; let $v_1$ be the end of $e_1$ in $C$.
We have $\deg(z)=k+1>(k-2-r)+2=\deg(x)+2=\deg(X)$, and thus there exists at least one
edge between $z$ and $C$. We now describe how to pick an edge $e_{2} \neq e_{1}$.
\begin{itemize}
\item If $\psi$ directs all edges between $z$ and $C$ in the same way, then since $(x,\psi)$ is a witness of $(k,r)$-tallness, $\deg(x)=k-2-r$, and $\deg(z)=k+1$,
we can choose $e_2$ as an edge between $z$ and $X\setminus\{x\}$ directed by $\psi$ in the opposite way.
\item  Otherwise, if there exists an edge between $\{z\}\cup C\setminus \{v_1\}$ and $X\setminus\{x\}$, choose $e_2$ as such an edge arbitrarily.
\item If no such edge exists, choose $e_2$ as an arbitrary edge between $v_1$ and $X\setminus\{x\}$ different from $e_1$ 
(which exists, since $\deg(X)=\deg(x)+2$). 
\end{itemize}
Let $v_2$ be the end of $e_2$ in $\{z\}\cup C$.
Let $(\g'_1,z)$ be obtained from $(\g/X,z)$ by splitting off $e_1$ with $e_2$, and if $v_1\neq v_2$, then denote by $e$ the resulting edge.
Let $x'_1$ be the vertex of $\g'_1$ into which we contracted $X$.
Let $\psi'$ be the tip preflow in $(\g'_1,z)$ obtained from $\psi$ by, in case that $e_2$ is incident with $z$, directing $e$ in the same way as $\psi$ directs $e_2$.

\begin{claim*}
The preflow $\psi'$ does not extend to a nowhere-zero flow in $(\g'_1,z)$.
\end{claim*}

\begin{subproof}
Suppose towards a contradiction that $\psi'$ extends to a nowhere-zero flow
$\vec{G}_1$ in $(\g'_1,z)$.  Let $(\g'_2,b)=(\g,z)\restriction X$ and let
$\psi_2$ be the tip preflow obtained from the restriction of $\vec{G}_1$ to
the edges with exactly one end in $X$ by directing $e_1$ and $e_2$ according to the orientation of $e$ in $\vec{G}_1$
(or arbitrarily in opposite directions in case that $v_1=v_2$).  Since $\psi$ does not extend to a nowhere-zero flow in $(\g,z)$,
we see that $\psi_2$ does not extend to a nowhere-zero flow in $(\g'_2,b)$.  Let $(\g_2,b)$ be a $\psi_2$-critical tip-respecting contraction of $(\g'_2,b)$,
and let $a$ be the vertex into which we contracted the set $A$ of vertices of $\g_2$ containing $x$.  If $|A|\ge 2$, then Theorem~\ref{thm-deg} for
$(\g,z)\restriction A$ implies that $\deg(a)=\deg(A)\ge \deg(x)+2=\deg(X)=\deg(b)$, but that contradicts Theorem~\ref{thm-deg} for $(\g_2,b)$.
Therefore, we have $A=\{x\}$ and $\deg(a)=\deg(x)=\deg(X)-2=\deg(b)-2$.  Since neither $e_1$ nor $e_2$ is incident with $x$, the preflow $\psi_2$ directs two edges not incident with $a$
in opposite ways.  We conclude that $(\g_2,b,a,\psi_2)$ is a tall tame critical easel, contradicting Theorem~\ref{thm:tall}.
\end{subproof}
Since $\psi'$ does not extend to a nowhere-zero flow in $(\g'_1,z)$, there exists a $\psi'$-critical tip-respecting contraction $(\g',z)$ of $(\g'_1,z)$, where $\g'=\g'_1/\PP'$.
Let $X'_1$ be the part of $\PP'$ containing $x'_1$, let $X'=(X'_1\setminus \{x'_1\})\cup X$, and let $\PP$ be obtained from $\PP'$ by replacing $X'_1$ by $X'$.
Let $x'$ be the vertex of $\g'$ into which we contracted $X'_1$.  Note that the choice of $e_2$ implies that $\psi'$ does not direct all edges around $z$ in the same way,
and thus $\psi'$ is a $k$-tallness-witnessing preflow in $(\g',z)$.
\begin{itemize}
\item If $X'\neq X$, then by the maximality of $X$ we have $\deg(x')\ge \deg_{\g}(X')-2>\deg_{\g}(x)=k-2-r$, and thus $(x',\psi')$ is a witness of $(k,r)$-tallness of $(\g',z)$.
\item If $X'=X$, then the choice of $e_2$ implies that $\psi'$ does not direct all edges between $z$ and $C$ in $\g'_1$ in the same way, and thus it also does not
direct all edges between $z$ and $V(\g')\setminus \{x',z\}$ in $\g'$ in the same way.  Moreover, we have $\deg(x')=\deg_{\g}(X)-2=\deg(x)=k-2-r$.
Hence, we again conclude that $(x',\psi')$ is a witness of $(k,r)$-tallness of $(\g',z)$.
\end{itemize}
Moreover, we claim that the canvas $(\g',z)$ is tame, and thus $(\g',z,x',\psi')\in \GG_{k,r}$.  To see that, we need the following observation.

\begin{claim*}
If $v_1$ and $v_2$ are contained in the same part $P\in \PP$, then $|P|\ge 2$.  
\end{claim*}
\begin{subproof}
Suppose for a contradiction that $|P|=1$, and thus $v_1=v_2$ and $P\neq X'$.
By the choice of $e_2$, this implies that every edge between $\{z\}\cup C$ and $X$ is incident with $v_1$ or $x$.
Let $m$ be the number of edges between $v_1$ and $X\setminus \{x\}$.
Since only $e_1$ and $e_2$ are split off, $\g'$ contains at least $m-2$ edges between the vertex corresponding to $P$ and $x'$, and by Lemma~\ref{lemma-simple} applied to $\g'$,
we have $m\le 3$.  Since $\deg(X)=\deg(x)+2$ and there are $\deg(X)-m$ edges between $\{z\}\cup C$ and $X$ that are incident with $x$,
there are $\deg(x)-(\deg(X)-m)=m-2$ edges between $x$ and $X\setminus\{x\}$.   We conclude that
$\deg_{\g}(X)\setminus\{x\} = 2m-2\le 4$.  By Observation~\ref{obs-subcrit} applied to $(\g,z)\restriction (X\setminus\{x\})$,
Corollary~\ref{cor-mainlov}, and the tameness of $(\g,z)$, we conclude that $m=3$ and $|X\setminus \{x\}|=1$.
However, that implies that $\g$ contains a triple edge between $v_1$ and the vertex of $X\setminus\{x\}$,
contradicting Lemma~\ref{lemma-simple}.
\end{subproof}

From this claim and Corollary~\ref{cor-contrtame}, it is easy to see that the canvas $(\g',z)$ is tame,
and thus $(\g',z,x',\psi')\in \GG_{k,r}$.  Clearly $o(\g',z)<o(\g,z)$, implying that $(\g',z,x')\prec (\g,z,x)$.

If $\{v_1,v_2\}\cap X'\neq\emptyset$, then $(\g',z)$ is isomorphic to $(\g / \PP,z)$.
By Theorem~\ref{thm-deg}, every vertex of $(\g',z)$ other than $z$ has degree at most $\deg(z)-2\le k-1$,
and thus by Observation~\ref{obs-inGk}, $(\g,z)$ is a $(\GG_k,x\to x',\GG_{k,r},\emptyset)$-expansion of $(\g',z)$.
Hence, $(\g,z,x,\psi)$ satisfies Theorem~\ref{thm-genladeg}(EXP), which is a contradiction.

It follows that $\{v_1,v_2\}\cap X'=\emptyset$.  If $v_1$ and $v_2$ are in the same part $P\in \PP$,
then let $y$ be the vertex of $\g'$ corresponding to this part; otherwise, $e$ is an edge of $\g'$ and we let $y=e$.
Let $(\g_0,z)$ be the $x'$-alteration of $(\g',z)$ on $y$, and observe that $(\g_0,z)=(\g / \PP,z)$.
By Theorem~\ref{thm-deg} applied to $(\g',z)$, every part $A\neq\{z\}$ of $\PP$ other than $X'$ and $P$ (in case that $y$ is a vertex) satisfies $\deg_{\g}(A)\le k-1$,
and $\deg_{\g}(X'),\deg_{\g}(P)\le k+1$.  Moreover, $|X'|\ge |X|\ge 2$, and by the claim above, if $y$ is a vertex, then $|P|\ge 2$.
Since $(\g,z)$ is $\psi$-critical by Lemma~\ref{lemma-genladeg-psi}, Lemma~\ref{lemma-gen-sepx}
implies that $(\g,z)\restriction P'\in \GG_k$ for every $P'\in \PP\setminus \{z\}$.  Together with Lemma~\ref{lemma-genladeg-sepx} applied to $X'$, this implies that
$(\g,z)$ is a $(\GG_k,x\to x',\GG_{k,r},\{x',y\})$-expansion of $(\g_0,z)$.
We conclude that $(\g,z,x,\psi)$ satisfies Theorem~\ref{thm-genladeg}(EXPX), which is a contradiction.
\end{proof}

\subsection{Step 2: Minimal $(k,r)$-counterexamples are $x$-\\homomogenous}\label{subsec:xhomogeasel}
We are now ready to deal with vertices of degree four.

\begin{lemma}\label{lemma-genladeg-no4}
Let $k$ and $r$ be non-negative integers.  If $(\g,z,x,\psi)$ is a minimal $(k,r)$-counter\-exam\-ple,
then all vertices of $\g$ other than $x$ have degree at least five.
\end{lemma}
\begin{proof}
Suppose for a contradiction that $\g$ has a vertex $v\neq x$ of degree at most four.  Note that $v\neq z$ by Corollary~\ref{cor-mainlov},
and $\deg(v)=4$ and $v$ has zero boundary by tameness of $(\g,z)$.
Lemma~\ref{lemma-genladeg-psi} implies that the canvas $(\g,z)$ is $\psi$-critical.  Let $C=V(\g)\setminus\{v,z,x\}$; since $|V(\g)|\ge 4$, we have $C\neq\emptyset$.
Lemma~\ref{lemma-2con} implies that there exists an edge $e_1$ between $v$ and a vertex in $C$.
\begin{itemize}
\item[(i)] If $\psi$ directs two edges between $z$ and $v$ in opposite ways, then let $e_3$ and $e_4$ be these edges and let $e_2$
be the edge incident with $v$ different from $e_1$, $e_3$, and $e_4$.
\item[(ii)] If $v$ is adjacent to $z$ and $\psi$ directs all edges between $z$ and $v$ in the same way, then 
since $\psi$ extends to a nowhere-zero flow in $\g/(C\cup \{x\})$, the edge between $z$ and $v$ has multiplicity at most two.
Let $e_2$ be an edge between $z$ and $v$, and assign labels $e_3$ and $e_4$ arbitrarily to the edges incident with $v$ different from $e_1$ and $e_2$.
\item[(iii)] If $v$ is not adjacent to $z$, then assign labels $e_2$, $e_3$, and $e_4$ to the edges incident with $v$ different from $e_1$ arbitrarily.
\end{itemize}
For $i\in\{1,\ldots,4\}$, let $v_i$ be the end of $e_i$ different from $v$, and let $N=\{v_1,\ldots, v_4\}$.

Let $\g_1$ be the $\Z$-bordered graph obtained from $\g$ by splitting off $e_1$ with $e_2$ (giving an edge $e'_1$) and $e_3$ with $e_4$
(giving an edge $e'_3$, unless $e_3$ and $e_4$ are both incident with $z$) and deleting the now isolated vertex $v$.
The preflow $\psi$ naturally corresponds to a tip preflow $\psi_1$ in $(\g_1,z)$, and $\psi_1$ does not extend to a nowhere-zero flow in $\g_1$. Hence, $(\g_1,z)$ has a tip-respecting $\psi_1$-critical contraction $(\g',z)$. Let $\PP'$ be the tip-respecting partition of $V(\g_1)$ such that $\g'=\g_1/\PP'$. 
Let $X'$ be the part of $\PP'$ containing $x$ and let $x'$ be the corresponding vertex of $\g'$.
\begin{claim*}
The easel $(\g',z,x',\psi_1)$ belongs to $\GG_{k,r}$ and $(\g',z,x')\prec (\g,z,x)$.
\end{claim*}
\begin{subproof}
Consider any part $A'\in \PP'$ other than $\{z\}$, and let $a$ be the corresponding vertex of $\g'$.
Let $A=A'\cup\{v\}$ if $|A'\cap N|\ge 3$ and $A=A'$ otherwise.
If $|A|\ge 2$, then Corollary~\ref{cor-contrtame} implies $\deg(a)\ge \deg_{\g}(A)-2\ge 4+|\tau(A)|=4+|\tau(a)|$.
If $|A|=1$, then by Lemma~\ref{lemma-simple} at most one of the edges $e_1$, \ldots, $e_4$ has an end in $A=A'$,
and $\deg(a)=\deg_{\g}(A)\ge 4+|\tau(A)|=4+|\tau(a)|$ since $(\g,z)$ is tame.  Therefore, the canvas $(\g',z)$ is tame.

Moreover, if $\deg_{\g'}(z)=k+1$, then $\deg_{\g}(z)=k+1$ and $\psi$ directs two edges incident with $z$ in opposite ways,
and thus so does $\psi_1$.  Therefore, $\psi_1$ is a $k$-tallness-witnessing preflow for $(\g', z)$, and $(\g',z)\in \GG_k$.
Clearly, we also have $o(\g',z)<o(\g,z)$, and thus $(\g',z,x')\prec (\g,z,x)$.

If $|X'\cap N|\ge 3$, then let $X=X'\cup\{v\}$, otherwise let $X=X'$.  If $|X|\ge 2$, then Theorem~\ref{thm-deg} applied to $(\g,z)\restriction X$
implies $\deg(x')\ge \deg_{\g}(X)-2\ge \deg(x)\ge k-2-r$.
If $X=\{x\}$, then $\deg(x')=\deg(x)\ge k-2-r$.

Suppose now that $\deg_{\g'}(z)=k+1$ and $\deg(x')=k-2-r$. It follows that $\deg_{\g}(z)=k+1$ and the labels of edges incident with $v$
were not chosen according to (i).  Furthermore, since $\deg(x')\ge \deg(x)\ge k-2-r$, we have $\deg(x)=k-2-r$.
If $|X|\ge 2$, then $\deg_{\g}(X)\le \deg(x')+2=\deg(x)+2$, contradicting Lemma~\ref{lemma-aroundx}; therefore, we have $X=\{x\}$.
Moreover, since $(x,\psi)$ is a witness of $(k,r)$-tallness of $(\g,z)$,
the tip preflow $\psi$ directs distinct edges between $z$ and $V(\g)\setminus \{z,x\}$ in opposite ways.  The choice of the labels in
(ii) and (iii) ensures that $\psi_1$ also directs distinct edges between $z$ and $V(\g_1)\setminus \{z,x\}$ in opposite ways,
and since $X=\{x\}$, the tip preflow $\psi_1$ directs distinct edges between $z$ and $V(\g')\setminus \{z,x'\}$ in $\g'$ in opposite ways as well.

Therefore, $(x',\psi_1)$ is a witness of $(k,r)$-tallness of $(\g',z)$, and $(\g',z,x',\psi_1)\in \GG_{k,r}$.
\end{subproof}
For $i\in \{1,2\}$, let $y_i$ be the edge $e'_{2i-1}$ if it is present in $\g'$, and let $y_i$ be the vertex
into which the part $Y_i\in\PP'$ containing $v_{2i-1}$ and $v_{2i}$ was contracted otherwise.  Note that
$v_3=v_4=y_2=z$ and $Y_2=\{z\}$ in the case (i); and that for $i\in\{1,2\}$, if $y_i\neq z$ is a vertex,
then $v_{2i-1}\neq v_{2i}$ by Lemma~\ref{lemma-simple}, and thus $|Y_i|\ge 2$.

If $y_1\neq y_2$, then let $(\g_0,z)$ be the $2$-alteration of $(\g',z)$ on $y_1$ and $y_2$, with the newly added vertex labelled $v$,
and let $\PP=\PP'\cup\{\{v\}\}$.
If $y_1=y_2$, then note that $y_1$ is a vertex (different from $z$), and let $(\g_0,z)=(\g',z)$ and $\PP=(\PP'\setminus\{Y_1\})\cup\{Y_1\cup \{v\}\}$.
Observe that in either case, we have $(\g_0/\PP,z)=(\g/\PP,z)$.

Consider any part $P\in\PP$ of size at least two, and let $p$ be the corresponding vertex of $\g_0$.
Since $(\g',z)\in \GG_k$, Theorem~\ref{thm-deg} implies $\deg_{\g'}(p)\le k-1$, and thus $\deg_{\g}(P)\le \deg_{\g'}(p)+2\le k+1$.
By Lemmas~\ref{lemma-gen-sepx} and \ref{lemma-genladeg-sepx}, we conclude that
$(\g,z)$ is a $(\GG_k,x\to x',\GG_{k,r},\{y_1,y_2\})$-expansion of $(\g_0,z)$, and thus $(\g,z,x,\psi)$ satisfies Theorem~\ref{thm-genladeg}(EXP) or (EXPA).
This is a contradiction.
\end{proof}

Next, let us get rid of mixed edges.

\begin{lemma}\label{lemma-genladeg-norede2}
Let $k$ and $r$ be non-negative integers.  If $(\g,z,x,\psi)$ is a minimal $(k,r)$-counter\-exam\-ple, then there is no mixed edge $uv\in E(\g-\{z,x\})$.
\end{lemma}
\begin{proof}
Let $\g=(G,\beta)$.  Recall that the canvas $(\g,z)$ is $\psi$-critical by Lemma~\ref{lemma-genladeg-psi}.
Suppose for a contradiction that $e=uv\in E(\g-\{z,x\})$ is a mixed edge, say with $v$ in-friendly and
$u$ out-friendly.  Let $\g_1=(G-e,\beta')$, where $\beta'(y)=\beta(y)$ for $y\in V(G)\setminus\{u,v\}$, $\beta'(u)=\beta(u)-1$, and $\beta'(v)=\beta(v)+1$.
Lemmas~\ref{lemma-conn} and \ref{lemma-2con} imply that $G-e$ is connected, and thus $\beta'$ is a $\Z$-boundary for $G-e$.
If $\g_1$ had a nowhere-zero flow extending $\psi$, it would give a nowhere-zero flow in $\g$ extending $\psi$
by directing the edge $e$ towards $v$.  Hence, $\psi$ does not extend to a nowhere-zero flow in $\g_1$, and thus $(\g_1,z)$ has a tip-respecting
$\psi$-critical contraction $(\g',z)$. Let $\PP$ be the tip-respecting partition such that $\g'=\g_1/\PP$, and let $x'$ be the vertex obtained by contracting the part $X\in \PP$ containing $x$.
Note that $o(\g',z)<o(\g,z)$, and thus $(\g',z,x')\prec (\g,z,x)$.

As in the proof of Lemma~\ref{lemma-norede}, observe that the assumptions on in-friendliness and out-friendliness of $u$ and $v$ together with
Lemma~\ref{lemma-genladeg-no4} imply that $(\g_1,z)$ is tame, and together with Corollary~\ref{cor-contrtame},
it follows that $(\g',z)$ is tame.  Moreover, $(\g',z)$ is $k$-tall, since $\psi$ is a $k$-tallness-witnessing preflow around $z$ in $\g'$.
Hence $(\g',z)\in\GG_k$.

If $|X|\ge 2$, then Theorem~\ref{thm-deg} implies $\deg(x')\ge \deg_{\g}(X)-1 >\deg(x) \ge k-2-r$.  If $X=\{x\}$,
then $\deg(x')=\deg(x)\ge k-2-r$; and moreover, if $\psi$ directs two edges not incident with $x$ in the opposite direction in $\g$,
then it does so in $\g'$ as well.  Therefore, $(x',\psi)$ is a witness of $(k,r)$-tallness of $(\g',z)$, and we have $(\g',z,x',\psi)\in\GG_{k,r}$.

If $u$ and $v$ are contained in the same part of $\PP$, then let $(\g_0,z)=(\g',z)$.  Otherwise, let $u'$ and $v'$ be the vertices of $\g'$
into which the parts containing $u$ and $v$ were contracted, and let $(\g_0,z)$ be the $1$-alteration of $(\g',z)$ obtained by adding the edge $u'v'$,
increasing the boundary at $u'$ by one, and decreasing it at $v'$ by one.  Note that in either case, $(\g_0,z)=(\g/\PP,z)$.

For every vertex $y\in V(\g_0)\setminus\{z\}$, if $A$ is the set of vertices of $\g$ contracted into $y$ and $|A|\ge 2$,
then note that $\deg_{\g}(A)\le \deg_{\g'}(y)+1\le \deg(z)-1\le k$ by Theorem~\ref{thm-deg}.
By Lemmas~\ref{lemma-gen-sepx} and \ref{lemma-genladeg-sepx}, we conclude that $(\g,z)$ is a $(\GG_k,x\to x',\GG_{k,r},\emptyset)$-expansion of $(\g_0,z)$.
It follows that the easel $(\g,z,x,\psi)$ satisfies Theorem~\ref{thm-genladeg}(EXP) or (EXPB),
which is a contradiction.
\end{proof}

Consequently, for any minimal $(k,r)$-counter\-exam\-ple $(\g,z,x,\psi)$, the canvas $(\g,z)$ is $x$-homogeneous, and by Observation~\ref{obs-allplusminus},
either $\tau(v)>0$ for every $v\in V(\g)\setminus\{x,z\}$, or $\tau(v)<0$ for every $v\in V(\g)\setminus\{x,z\}$.

\begin{lemma}\label{lemma-genladeg-norede1}
Let $k$ and $r$ be non-negative integers.  If $(\g,z,x,\psi)$ is a minimal $(k,r)$-counter\-example, then
either $\tau(v)>0$ for every $v\in V(\g)\setminus\{x,z\}$ and $\psi$ directs all edges not incident with $x$ away from $z$,
or $\tau(v)<0$ for every $v\in V(\g)\setminus\{x,z\}$ and $\psi$ directs all edges not incident with $x$ towards $z$.
\end{lemma}
\begin{proof}
Let $\g=(G,\beta)$.  Recall that the canvas $(\g,z)$ is $\psi$-critical by Lemma~\ref{lemma-genladeg-psi}.
By Lemma~\ref{lemma-genladeg-norede2}, we can by symmetry assume that $\tau(v)>0$ for every $v\in V(\g)\setminus\{x,z\}$.
Suppose for a contradiction that $\psi$ directs an edge $e=uz$ of $\g$ with $u\neq x$ towards $z$.
Let $\g'=(G-e,\beta')$, where $\beta'(y)=\beta(y)$ for $y\in V(G)\setminus\{z,v\}$, $\beta'(z)=\beta(z)+1$, and $\beta'(u)=\beta(u)-1$.
Let $\psi'$ be obtained from $\psi$ by removing the edge $e$.  As in the proof of Lemma~\ref{lemma-sameor}, we argue that $(\g',z)$ is
a $\psi'$-critical tame canvas, and since $\deg_{\g'}(z)\le \deg_{\g}(z)-1\le k$, we have $(\g',z)\in \GG_k$.
Moreover, since $\deg_{\g'}(x)=\deg_{\g}(x)$, it follows that $(\g',z,x,\psi')\in\GG_{k,r}$.

Note that $(\g,z)$ is a tip-alteration of $(\g',z)$.
Since $o(\g',z)=o(\g,z)$, the canvas $(\g,z)$ is $x$-homogeneous, and $\deg_{\g'}(z)<\deg_{\g}(z)$,
we have $(\g',z,x)\prec (\g,z,x)$.  Therefore, $(\g,z,x,\psi)$ satisfies Theorem~\ref{thm-genladeg}(ADD), which is a contradiction.
\end{proof}
\subsection{Step 3: There are no $(k,r)$-minimal-counterexamples}\label{subsec:easelfinisher}
It is now easy to finish the proof of the $(k,r)$-tall easel generation theorem.

\begin{proof}[Proof of Theorem~\ref{thm-genladeg}]
Let $k$ and $r$ be non-negative integers and suppose for a contradiction that there exists a minimal $(k,r)$-counterexample $(\g,z,x,\psi)$.
By Lemma~\ref{lemma-genladeg-norede1} and symmetry, we can assume that $\tau(v)>0$ for every $v\in V(\g)\setminus\{x,z\}$ and that $\psi$ directs all edges not incident with $x$ away from $z$.
Since $(x,\psi)$ is a witness of $(k,r)$-tallness of $(\g,z)$, it follows that $\deg(z) \le k$ or $\deg(x) >k-2-r$.

By Theorem~\ref{thm-deg}, we have $\deg(z)\ge \deg(x)+2$,
and thus there exist distinct edges $e=zv$ and $e_0$ between $z$ and $V(\g)\setminus \{x,z\}$.
Let $\g'$ be the $\Z$-boundaried graph obtained from $\g$ by adding
an edge $e'$ parallel to $z$ and let $\psi'$ be the preflow around $z$ in $\g$ obtained from $\psi$ by
directing $e$ and $e'$ towards $z$.  As in the proof of Theorem~\ref{thm:tall}, we argue that $(\g',z)$ is $\psi'$-critical and tame.
Moreover, note that $\psi'$ directs the edges $e$ and $e_0$ between $z$ and $V(\g')\setminus \{x,z\}$ in the opposite ways.
If $\deg_{\g}(z)\le k$, then $\deg_{\g'}(z)=\deg_{\g}(z)+1\le k+1$, and since $\deg_{\g'}(x)=\deg_{\g}(x)\ge k-2-r$,
we have $(\g',z,x,\psi')\in\GG_{k,r}$.
If $\deg_{\g}(z)=k+1$, then recall that $\deg(x) >k-2-r$; consequently $\deg_{\g'}(z)=\deg_{\g} (z)+1=k+2$ and
$\deg_{\g'}(x)=\deg_{\g}(x)\ge (k+1)-2-r$, and $(\g',z,x,\psi')\in\GG_{k+1,r}$.

By Lemma~\ref{lemma-2con}, $v$ has a neighbour $u\neq x$.  Since we changed the parity of the degree of $v$, we have $\tau_{\g'}(v)<0$, while $\tau_{\g'}(u)=\tau_{\g}(u)>0$.
Hence, $(\g',z)$ is not $x$-homogeneous.  Since $o(\g',z)=o(\g,z)$ and $(\g,z)$ is $x$-homogeneous, it follows that $(\g',z,x)\prec (\g,z,x)$.
Therefore, $(\g,z,x,\psi)$ satisfies Theorem~\ref{thm-genladeg}(REM), which is a contradiction.
\end{proof}
\section{Bounding the censuses: Proving Theorem \ref{thm-degbetter}}
Our goal in this section is to prove Theorem \ref{thm-degbetter}. We restate the theorem for the ease of the reader.

\thmdegbetter*

To prove Theorem \ref{thm-degbetter} we simply apply our easel generation
theorem to a minimal counterexample and examine the possible outcomes. In case
(EXPB), we find that several small graphs can appear, but we show that
they are not flow-critical.  We start by presenting observations that are useful in this second
part.
\subsection{Small graphs without nowhere zero flows}
We make the following useful observation.
\begin{observation}\label{obs-sumdeg}
If $(\g,z)$ is a flow-critical canvas and every vertex of $\g$ has degree at least five, then
$$\sum_{d \in C(\g,z)} d \leq \deg z+|C(\g,z)|(|C(\g,z)|-1).$$
Moreover, if equality holds, then $\g-z$ is a complete graph.
\end{observation}
\begin{proof}
Since every vertex of $\g$ has degree at least five, we have $|C(\g,z)|=|V(\g-z)|$
and
$$\sum_{d \in C(\g,z)}d=\sum_{v\in V(\g)\setminus\{z\}} \deg_{\g} v=\deg z+2|E(\g-z)|.$$
The claim follows, since by Lemma~\ref{lemma-simple}, $\g-z$ does not have more edges than the complete graph.
\end{proof}
We will need several standard results on flows in small graphs.  Recall that a graph $G$ is \emph{$\Z$-connected} if $G$ is connected and $(G,\beta)$ has a nowhere-zero flow for every $\Z$-boundary $\beta$.
Let us note the following standard observation.
\begin{observation}
If $G$ is a spanning subgraph of a graph $G'$ and $G$ is $\Z$-connected, then $G'$ is $\Z$-connected as well.
\end{observation}
We say that a graph $G$ is \emph{collapsible} if either $|V(G)|=1$, or $G$ has an edge $uv$ of multiplicity greater than one
and $G/\{u,v\}$ is collapsible. By \emph{suppressing} a vertex of degree two, we mean splitting off its incident edges and deleting the vertex.
\begin{observation}\label{obs-collaps}
Let $G$ be a graph.
\begin{itemize}
\item If $G$ is collapsible, then $G$ is $\Z$-connected.
\item If $v\in V(G)$ is a vertex of degree two such that the graph $G'$ obtained from $G$ by suppressing $v$ is collapsible, then $(G,\beta)$ has a nowhere-zero flow for every $\Z$-boundary $\beta$ such that $\beta(v)=0$.
\item If $v\in V(G)$ is a vertex of degree three and $e$ is an edge incident with $v$ such that the graph obtained from $G-e$ by suppressing $v$ is collapsible, then $(G,\beta)$ has a nowhere-zero flow for every $\Z$-boundary $\beta$ such that $\beta(v)\neq 0$.
\end{itemize}
\end{observation}
\begin{proof}
The first point follows analogously to Lemma~\ref{lemma-simple}.  For the second point, since $G'$ is collapsible,
the $\Z$-bordered graph $(G',\beta\restriction V(G'))$ has a nowhere-zero flow. We obtain a nowhere-zero flow in $(G,\beta)$ by directing both edges incident with $v$ according
to the corresponding edge of $G'$.  For the third point, suppose by symmetry that $\beta(v)=1$.  Let $\beta'(v)=0$, $\beta'(u)=\beta(u)+1$
for the other endpoint $u$ of $e$, and let $\beta'(x)=\beta(x)$ for every other vertex $x$ of $G$.  The second point applied to $(G-e,\beta')$
gives a nowhere-zero flow which extends to a nowhere-zero flow in $(G,\beta)$ by directing the edge $e$ towards $u$.
\end{proof}

Using this observation, it is easy to derive the following facts, illustrated in Figure \ref{fig:lemma-small}.

\begin{lemma}\label{lemma-small}
Let $(G,\beta)$ be a $\Z$-bordered graph without a nowhere-zero flow and with no edges of multiplicity greater than one.
\begin{itemize}
\item If $G=K_4$, then $\beta(v)=0$ for every $v\in V(G)$.
\item If $|V(G)|=5$ and $G$ has at most two non-edges, then $G$ consists of a copy $H$ of $K_4$ together with a vertex $v$ of degree two with neighbours $v_1,v_2 \in V(H)$,
$\beta(v)=\beta(v_1)=\beta(v_2)\neq 0$, and $\beta(x)=0$ for each $x\in V(G)\setminus \{v,v_1,v_2\}$.
\item If $|V(G)|=6$ and $G$ has at most four non-edges and minimum degree at least two, then $G$ consists of two copies of $K_4$ whose intersection is $K_2$ and $\beta(x)=0$ for every $x\in V(G)$.
\end{itemize}
\end{lemma}

\begin{proof}
If $G=K_4$, then deleting any edge and suppressing a resulting vertex of degree two gives a collapsible graph, and thus $\beta(v)=0$ for every $v\in V(G)$
by Observation~\ref{obs-collaps}.

Suppose now that $|V(G)|=5$ and $G$ has exactly two non-edges.  There are two possibilities:
\begin{itemize}
\item The two non-edges form a matching, and thus $G$ consists of a $4$-cycle $K=v_1v_2v_3v_4$ and a vertex $v$ adjacent to all of its vertices.
Deleting any edge of $K$ and suppressing a resulting vertex of degree two gives a collapsible graph, and thus $\beta(v_i)=0$ for $i\in\{1,\ldots,4\}$
by Observation~\ref{obs-collaps}.  It follows that $\beta(v)=0$, since $\beta$ is a $\Z$-boundary.  However, then it is easy to construct a nowhere-zero flow in $G$ by orienting all edges incident with $u \in \{v_1,v_3\}$ towards $u$, and all edges incident with $u \in \{v_2,v_4\}$ away from $u$.
Hence, this graph $G$ is $\Z$-connected.

Moreover, note that this graph is contained as a subgraph of every simple graph with $5$ vertices and at most one non-edge,
and thus all such graphs are $\Z$-connected as well. 
\item The two non-edges do not form a matching.  It follows that $G$ consists of a complete graph on vertices $v_1$, \ldots, $v_4$ together with a vertex $v$ and edges $vv_1$ and $vv_2$.
Suppressing $v$ results in a collapsible graph, and thus by Observation~\ref{obs-collaps}, we have $\beta(v)\neq 0$. Suppose by symmetry that $\beta(v)=1$.  Deleting the edge $v_3v_4$ and suppressing $v_3$ or $v_4$
results in a collapsible graph, and thus again by Observation~\ref{obs-collaps}, we have $\beta(v_3)=\beta(v_4)=0$.
Since $\beta$ is a $\Z$-boundary, we have $(\beta(v_1),\beta(v_2))\in \{(-1,0),(0,-1),(1,1)\}$.
The first two options lead to a $\Z$-bordered graph with a nowhere-zero flow, and thus we conclude that $\beta(v_1)=\beta(v_2)=\beta(v)=1$.
\end{itemize}

Suppose now that $|V(G)|=6$ and $G$ has exactly four non-edges and minimum degree at least two.  If $G$ contains a vertex $v$ of degree two,
then $G-v$ has only one non-edge and by the previous point $G-v$ is $\Z$-connected.  Thus, we can orient the edges incident with $v$ to match the boundary
$\beta(v)$ and extend the flow to a nowhere-zero flow in $G$.

Hence, $G$ has minimum degree three (since it has four non-edges, it cannot have minimum degree at least four).
Let $v$ be a vertex of $G$ of degree three; then $G-v$ has two non-edges and analogously to the previous case, we can assume that $G-v$ is not $\Z$-connected.
By the previous point,
$G-v$ consists of a complete graph on vertices $v_1$, \ldots, $v_4$ and a vertex $v'$ adjacent to $v_1$ and $v_2$.
Since $G$ has minimum degree at least three, $v$ is adjacent to $v'$.  It follows that every vertex of $G$ of degree three has a neighbour of degree three. Up to symmetry, there are the following possibilities for the neighbours of $v$ in $\{v_1,\ldots,v_4\}$:
\begin{itemize}
\item If $v$ is adjacent to $v_2$ and $v_3$, then $v_4$ is a vertex of degree three with all neighbours of larger degree, a contradiction.
\item If $v$ is adjacent to $v_3$ and $v_4$, then since deleting the edge $vv'$ and suppressing $v$ or $v'$ gives a collapsible graph,
we have by Observation~\ref{obs-collaps} that $\beta(v)=\beta(v')=0$.  By symmetry, we can assume that $\beta(v_4)\neq -1$.  Let $\beta'(v_4)=\beta(v_4)+1$, $\beta'(v_3)=\beta(v_3)+1$,
$\beta'(v_2)=\beta(v_2)-1$ and $\beta'(v_1)=\beta(v_1)-1$.  Since $\beta'(v_4)\neq 0$, $(G[\{v_1,\ldots,v_4\}],\beta')$ has a nowhere-zero flow,
as we have observed at the beginning of the lemma.  Orienting the edges incident with $v$ towards $v$ and those incident with $v'$ away from $v'$
extends this to a nowhere-zero flow in $(G,\beta)$.
\item Hence, $v$ is adjacent to $v_1$ and $v_2$.  Observation~\ref{obs-collaps} applied to edges $vv'$ and $v_3v_4$ gives $\beta(v)=\beta(v')=\beta(v_3)=\beta(v_4)=0$.
It follows that $\beta(v_1)=-\beta(v_2)$.  If $\beta(v_1)\neq 0$, then $(G,\beta)$ has a nowhere-zero flow, and thus $\beta(v_1)=\beta(v_2)=0$.
\end{itemize}
Finally, if $|V(G)|=6$ and $G$ has at most three non-edges, then we can delete edges from $G$ so that the resulting graph $G'$ has exactly four non-edges, minimum degree at least two,
and it is not the union of two $K_4$'s sharing an edge. By the previous analysis, this implies that $G'$ is $\Z$-connected, and thus $G$ is $\Z$-connected as well and $(G,\beta)$ has a nowhere-zero flow.
\end{proof}
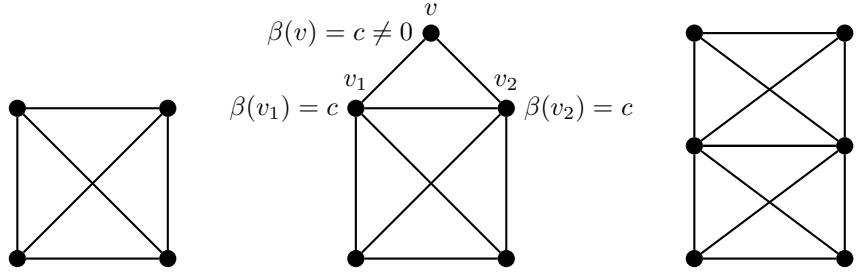
\begin{figure}
\begin{center}
    \begin{tikzpicture}
        \node[blackvertexv2] at (0,0) (v1) {};
        \node[blackvertexv2] at (2,0) (v2) {};
        \node[blackvertexv2] at (2,2) (v3) {};
        \node[blackvertexv2] at (0,2) (v4) {};
        \draw[thick,black] (v1)--(v2)--(v3)--(v4)--(v1);
        \draw[thick,black] (v1)--(v3);
        \draw[thick,black] (v2)--(v4);

        \begin{scope}[xshift = 4.5cm]
        \node[blackvertexv2] at (0,0) (v1) {};
        \node[blackvertexv2] at (2,0) (v2) {};
        \node[blackvertexv2] at (2,2) (v3) [label = above:$v_{2}$] [label = right:$\beta(v_{2}) \equal c$]{};
        \node[blackvertexv2] at (0,2) (v4) [label = above:$v_{1}$] [label = left:$\beta(v_{1}) \equal c$]{};
        \node[blackvertexv2] at (1,3) (v5) [label = above:$v$] [label = left:$\beta(v) \equal c \neq 0$]{};
        \draw[thick,black] (v1)--(v2)--(v3)--(v4)--(v1);
        \draw[thick,black] (v1)--(v3);
        \draw[thick,black] (v2)--(v4);
        \draw[thick,black] (v5)--(v3);
        \draw[thick,black] (v5)--(v4);
        \end{scope}

        \begin{scope}[xshift = 9cm]
        \node[blackvertexv2] at (0,0) (v1) {};
        \node[blackvertexv2] at (2,0) (v2) {};
        \node[blackvertexv2] at (2,1.5) (v3) {};
        \node[blackvertexv2] at (0,1.5) (v4) {};
        \node[blackvertexv2] at (0,3) (v5) {};
        \node[blackvertexv2] at (2,3) (v6) {};
        \draw[thick,black] (v1)--(v2)--(v3)--(v4)--(v1);
        \draw[thick,black] (v1)--(v3);
        \draw[thick,black] (v2)--(v4);
        \draw[thick,black] (v5)--(v6)--(v3)--(v4)--(v5);
        \draw[thick,black] (v5)--(v3);
        \draw[thick,black] (v6)--(v4);
        \end{scope}
    \end{tikzpicture}
    \end{center}
    \caption{The pairs $(G,\beta)$ without nowhere-zero flows in Lemma \ref{lemma-small}. Vertices with no boundary shown are assumed to have boundary zero.} \label{fig:lemma-small}
\end{figure}

\subsection{Containing the censuses}

In this subsection, we prove Theorem \ref{thm-degbetter}.
For non-negative integers $k$ and $r$, let
$$\CC_{k,r}=\{(\deg(z), C(\g,z)): (\g,z,x,\psi)\in\GG_{k,r}\}.$$
Let
$$\CC'_{k,0}=\{(k,\{k-2\}),(k+1,\{k-1\}),(k+1,\{k-2,5\}),(k+1,\{k-2,5,5,5\})\}.$$
Theorem~\ref{thm-degbetter} is then a consequence of the following lemma.
Indeed, suppose that $(\g,z)$ is a non-trivial flow-critical tame canvas with $k=\deg(z)\ge 7$ and
there exists a vertex $x\in V(\g)\setminus\{z\}$ of degree $k-2$. Let $\psi$ be a tip-respecting
preflow that does not extend to a nowhere-zero flow in $(\g,z)$; then $(\g,z,x,\psi)\in\GG_{k,0}$.
Moreover, if $(\deg(z), C(\g,z))\in\CC'_{k,0}$, then $C(\g,z)=\{k-2\}$.

\begin{lemma}\label{lemma-degbetter}
For any integer $k\ge 7$, $\CC_{k,0}\subseteq \CC'_{k,0}$.
\end{lemma}
\begin{proof}
Suppose for a contradiction that there exists an integer $k\ge 7$ and an easel $(\g,z,x,\psi)\in \GG_{k,0}$ such that
$(\deg(z), C(\g,z))\not\in \CC'_{k,0}$, and choose such an easel with $(\g,z,x)$ minimal in the $\prec$ ordering;
such an easel exists, since $\prec$ does not have infinite decreasing chains by Observation~\ref{obs-precord}.
We have $\deg(x) \ge k-2$, and thus $\deg(z)\ge k$ by Theorem~\ref{thm-deg}.  Moreover, $\deg(z)\le k+1$
since $(\g,z,x,\psi)\in \GG_{k,0}$.

We split into cases depending on the outcomes of Theorem~\ref{thm-genladeg} applied to $(\g,z,x,\psi)$. 
\begin{claim*}
The outcome (SMALL) does not occur.
\end{claim*}
\begin{subproof}
Suppose it does, and in this case let $V(\g)=\{z,x,v\}$.  Since $(\g,z)$ is tame, we have $\deg(v) \ge 4$, and since $\deg(x)>\deg(z)-4$,
the vertices $x$ and $v$ must be adjacent.  By Lemma~\ref{lemma-simple}, $x$ and $v$ are joined only by one edge, and thus
$$(\deg(z),\deg(x),\deg(v))\in \{(k,k-2,4),(k+1,k-1,4),(k+1,k-2,5)\}.$$  In any of these cases, we have $(\deg(z),C(\g,z))\in \CC'_{k,0}$,
which is a contradiction.
\end{subproof}

Let $(\g',z,x',\psi')$ be the easel that appears in the rest of the outcomes of Theorem~\ref{thm-genladeg} with $(\g',z,x')\prec (\g,z,x)$.
By the minimality of $(\g,z,x,\psi)$, if $(\g',z,x',\psi')\in \GG_{k,r}$, then $(\deg_{\g'}(z),C(\g',z))\in \CC'_{k,0}$,
and if $(\g',z,x',\psi')\in \GG_{k+1,0}$, then $(\deg_{\g'}(z), C(\g',z))\in \CC'_{k+1,0}$.  If $\deg_{\g'}(x')\ge 5$, then let $a_0$ denote the element of
$C(\g',z)$ corresponding to $\deg_{\g'}(x')$.

\begin{claim*}
The outcome (ADD) does not occur.
\end{claim*}

\begin{subproof}
Suppose it does. Since $(\g,z)$ is a tip-alteration of $(\g',z)$ at a vertex $v\neq x$, $C(\g,z)$ is obtained from $C(\g',z)$
by either choosing an element $a$ different from $a_0$ and replacing it by $a+1$, or (in case that $\deg_{\g'}(v)=4$)
by adding $5$ to $C(\g',z)$.  Moreover, since $\deg_{\g}(z)\le k+1$ and $\deg_{\g'}(x)=\deg_{\g}(x)\ge k-2$,
Theorem~\ref{thm-deg} implies that $\deg_{\g'}(z)=k$, $\deg(x)=k-2$, and $\deg_{\g}(z)=k+1$.
Since $(\deg_{\g'}(z), C(\g',z))\in \CC'_{k,0}$, it follows that $C(\g',z)=\{k-2\}$ and $C(\g,z)=\{k-2,5\}$.
Let us remark that the case $C(\g,z)=\{k-1\}$ is not possible, since $a$ is different from $a_0$.
However, then $(\deg_{\g}(z),C(\g,z))\in \CC'_{k,0}$, which is a contradiction.
\end{subproof}

\begin{claim*}
The outcome (REM) does not occur.
\end{claim*}

\begin{subproof}
Suppose it does. Since $(\g,z)$ is a tip-reduction of $(\g',z)$
at a vertex other than $x$ and all vertices of $\g$ other than $x$ have degree at least $5$,
$C(\g,z)$ is obtained from $C(\g',z)$ by choosing an element $a\ge 6$ different from $a_0$ and replacing it by $a-1$.
However, $(\deg_{\g'}(z),C(\g',z))\in \CC'_{k,0}\cup \CC'_{k+1,0}$, and the inspection of the definition of $\CC'_{k,0}$ shows
that no such element $a$ exists.  This is a contradiction.
\end{subproof}

For the remaining outcomes ((EXP), (EXPA), (EXPX), and (EXPB)), let $(\g_0,z)$ be the canvas (equal to $(\g',z)$, or to
a $2$-alteration, an $x'$-alteration, or a $1$-alteration of $(\g',z)$) such that $(\g,z)$ is a $(\GG_k,x\to x',\GG_{k,0},Y)$-expansion of $(\g_0,z)$ for a set $Y$.
Let $\PP$ be the corresponding tip-respecting $\GG_k$-partition such that $(\g/\PP, z)=(\g_0,z)$.
We say that a vertex $v\in V(\g_0)\setminus\{z\}$ (or $v\in V(\g')\setminus\{z\}$) \emph{contributes} degrees $d_1$, \ldots, $d_m$
to $C(\g,z)$ if the part $P\in\PP$ contracted to $v$ satisfies
$$\{\deg_{\g}(u):u\in P,\deg_{\g}(u)\ge 5\}=\{d_i:i\in\{1,\ldots,m\}, d_i\ge 5\},$$
where both sides of the equality are multisets.
The following observations will be sufficient to bound the census of $(\g,z)$ based on the census of $(\g_0,z)$:
\begin{itemize}
\item[(i)] If a vertex $v\in V(\g_0)\setminus \{z\}$ has degree at most five in $\g_0$, then Observation~\ref{obs-subcrit} and Theorem~\ref{thm-deg}
imply that the part of $\PP$ contracted to $v$ has size one, and thus it contributes $\deg_{\g_0}(v)$ to $C(\g,z)$.
\item[(ii)] If a vertex $v\in V(\g_0)\setminus \{z\}$ has degree six in $\g_0$ and the part $P\in \PP$ contracted to $v$ has size at least two,
then by Observation~\ref{obs-subcrit} and Theorem~\ref{thm-deg} all vertices in $P$ have degree four in $\g$, and thus $v$
does not contribute anything to $C(\g,z)$.
\item[(iii)] Suppose a vertex $v\in V(\g_0)\setminus \{z\}$ has degree seven in $\g_0$ and the part $P\in \PP$ contracted to $v$ has size at least two.
Let $(\g_1,b)=(\g,z) \restriction P$.  By Observation~\ref{obs-subcrit}, Theorem~\ref{thm-deg} and parity, there exists a vertex $x_1\in P$ such that $\deg_{\g}(x_1)=5$.  Consequently, letting
$\psi_1$ be any tip preflow that does not extend to a nowhere-zero flow in $(\g_1,b)$, we have
$(\g_1,b,x_1,\psi_1)\in\GG_{7,0}$. Clearly $o(\g_1,b)<o(\g,z)$ and $(\g_1,b,x_1)\prec (\g,z,x)$.
Therefore, we have $(\deg(b),C(\g_1,b))\in \CC'_{7,0}$, and since $\deg(b)=7$, it follows that $C(\g_1,b)=\{5\}$.
Hence, $v$ contributes $5$ to $C(\g,z)$.
\item[(iv)] Finally, let us consider a contribution of the vertex $x'$ to $(\g,z)$.  Let $P\in\PP$ be the part contracted to the vertex $x'$;
we have $x\in P$.  If $|P|=1$, then $x'$ contributes $\deg_{\g_0} x'=\deg_{\g} x$ to $(\g,z)$.

If $|P|\ge 2$, then let $(\g_1,b)=(\g,z) \restriction P$ and note that $\deg b=\deg_{\g_0} x'$.
Since $(\g,z)$ is a $(\GG_k,x\to x',\GG_{k,0},Y)$-expansion of $(\g_0,z)$, there exists a tip preflow $\psi_1$
such that $(\g_1,b,x,\psi_1)\in\GG_{k,0}$.  Clearly $o(\g_1,b)<o(\g,z)$ and $(\g_1,b,x)\prec (\g,z,x)$.
Therefore, we have $(\deg(b),C(\g_1,b))\in \CC'_{k,0}$.  We conclude that
\begin{itemize}
\item $\deg_{\g_0}(x')=k$ and $x'$ contributes only $k-2=\deg_{\g}(x)=\deg_{\g_0}(x')-2$ to $C(\g,z)$; or,
\item $\deg_{\g_0}(x')=k+1$ and $x'$ contributes only $k-1=\deg_{\g}(x)=\deg_{\g_0}(x')-2$ to $C(\g,z)$; or,
\item $\deg_{\g_0}(x')=k+1$ and $x'$ contributes only $k-2=\deg_{\g}(x)=\deg_{\g_0}(x')-3$ and one or three $5$'s to $C(\g,z)$.
\end{itemize}
\end{itemize}

Let us now discuss each of the conclusions separately.  Note that in all the cases, we have $(\deg_{\g'}(z),C(\g',z))\in \CC'_{k,0}$,
and $C(\g',z)$ consists of $\deg_{\g'}(x')\le k-1$ and possibly several fives.
\begin{claim*}
The case (EXP) does not occur.
\end{claim*}

\begin{subproof}
Suppose it does. Note that $C(\g_0,z)=C(\g',z)\in \CC'_{k,0}$ and $\deg_{\g_0}(x')=\deg_{\g'}(x')\le k-1$. By (i) and (iv), we have $C(\g,z)=C(\g',z)$, and thus $(\deg_{\g}(z), C(\g,z))\in \CC'_{k,0}$.
This is a contradiction.
\end{subproof}

\begin{claim*}
Neither (EXPA) nor (EXPX) occur.
\end{claim*}

\begin{subproof}
In this case, $(\g_0,z)$ is a $2$-alteration of $(\g',z)$ on some $y_1$ and $y_2$ or an $x'$-alteration of $(\g',z)$ on some $y_1$.
In the latter case, let $y_2=x'$, so that $Y=\{y_1,y_2\}$ in both cases.

If $z\in Y$, then we would be in the (EXPA) case and we would have $\deg_{\g'}(z)=\deg_{\g_0}(z)-2=\deg_{\g}(z)-2\le k-1$, which is not possible since $(\deg_{\g'}(z),C(\g',z))\in \CC'_{k,0}$.
It follows that $z\not\in Y$, and $\deg_{\g'}(z)=\deg_{\g_0}(z)=\deg_{\g}(z)$.  
Moreover, $C(\g_0,z)$ is obtained from $C(\g',z)$
by replacing $\deg_{\g'}(y)$ by $\deg_{\g'}(y)+2$ for each vertex $y\in Y$ (or adding $\deg_{\g'}(y)+2=6$
to $C(\g',z)$ when $\deg_{\g'}(y)=4$). 

Consider any vertex $v\in V(\g')\setminus \{z,x'\}$.  If $v\not\in Y$, then $\deg_{\g_0}(v)=\deg_{\g'}(v)\in\{4,5\}$
and by (i), $v$ contributes $\deg_{\g'}(v)$ to $C(\g,z)$.  If $v\in Y$, then let $X_v$ be the part of $\PP$ contracted into $v$
and note that $|X_v|\ge 2$.
\begin{itemize}
\item If $\deg_{\g'}(v)=4$, then $\deg_{\g_0}(v)=6$ and by (ii) $v$ does not contribute anything to $C(\g,z)$.
\item If $\deg_{\g'}(v)=5$, then $\deg_{\g_0}(v)=7$ and by (iii) $v$ contributes $5$ to $C(\g,z)$.
\end{itemize}
In either case, we again conclude that $v$ contributes $\deg_{\g'}(v)$ to $C(\g,z)$.

Since $(\deg_{\g'}(z),C(\g',z))\in\CC'_{k,0}$, $\deg_{\g}(z)=\deg_{\g'}(z)$, and $(\deg_{\g}(z),C(\g,z))\not\in\CC'_{k,0}$,
we have $C(\g,z)\neq C(\g',z)$.  It follows that $x'$ contributes something else than $\deg_{\g'}(x')$ to $C(\g,z)$.
If $x'\not\in Y$, then $\deg_{\g_0}(x')=\deg_{\g'}(x')\le k-1$ and $x'$ would only contribute $\deg_{\g'}(x')$ to $C(\g,z)$ by (iv).
Therefore, $x'\in Y$, and in particular the part of $\PP$ contracted to $x'$ has size at least two and $\deg_{\g_0}(x')=\deg_{\g'}(x')+2$.

Since $x'$ does not contribute only $\deg_{\g'}(x')=\deg_{\g_0}(x')-2$ to $C(\g,z)$, (iv) implies that
$\deg_{\g_0}(x')=k+1$ and $x'$ contributes only $k-2$ and one or three $5$'s to $C(\g,z)$.
Note that $\deg_{\g'}(x')=k-1$.  Recall that $(\deg_{\g'}(z),C(\g',z))\in \CC'_{k,0}$, and thus we have $(\deg_{\g'}(z),C(\g',z))=(k+1,\{k-1\})$.
As we have argued above, this implies that no vertex of $V(\g')\setminus\{z,x'\}$ contributes anything to $C(\g,z)$, and thus
$$(\deg_{\g}(z),C(\g,z))\in\{(k+1,\{k-2,5\}),(k+1,\{k-2,5,5,5\})\}\subset \CC'_{k,0};$$ this is a contradiction.
\end{subproof}

It follows that (EXPB) holds. In particular, $(\g_0,z)$ is obtained from $(\g',z)$ by adding an edge $y_1y_2$ not incident with $z$ (and adjusting the boundary),
and thus $C(\g_0,z)$ is obtained from $C(\g',z)$ by replacing $\deg_{\g'}(y)$ by $\deg_{\g'}(y)+1$ for each vertex $y\in \{y_1,y_2\}$
(or adding $\deg_{\g'}(y)+1=5$ in case that $\deg_{\g'}(y)=4$).  By (ii), in case that the new element is $6$, the vertex $y$ can contribute either
$6$ or nothing to $C(\g,z)$.

Consider now the vertex $x'$, and let $X$ be the part of $\PP$ contracted to $x'$ (and containing $x$).
Since $(\g',z,x',\psi')\in\GG_{k,0}$, we have $\deg_{\g'}(x')\in\{k-2,k-1\}$.
\begin{itemize}
\item If $x'\not\in \{y_1,y_2\}$, then $\deg_{\g_0}(x')=\deg_{\g}(x)\le k-1$, and by (iv),
the vertex $x'$ contributes only $\deg_{\g'}(x)=\deg(x)$ to $C(\g,z)$.
\item If $x'\in \{y_1,y_2\}$, then we have $\deg_{\g'}(x')+1=\deg_{\g_0}(x')=\deg_{\g}(X)\in\{k-1,k\}$.
\begin{itemize}
\item If $|X|=1$, then $\deg_{\g_0}(x')=\deg_{\g}(x)$, and since $\deg_{\g}(x)\le k-1$,
it follows that $\deg_{\g'}(x')=k-2$ and $x'$ contributes only $\deg(x)=k-1=\deg_{\g'}(x')+1$ to $C(\g,z)$.
\item If $|X|\ge 2$, then note that Theorem~\ref{thm-deg} for $(\g,z)\restriction X$ implies that
$k-2\le \deg_{\g}(x)\le \deg_{\g}(X)-2\le k-2$, and thus
$\deg_{\g_0}(x')=\deg_{\g}(X)=k$, $\deg_{\g'}(x')=k-1$, and $\deg_{\g}(x)=k-2$.  By (iv), $x'$ contributes only $\deg(x)=k-2=\deg_{\g'}(x')-1$ to $C(\g,z)$.
\end{itemize}
\end{itemize}
Thus, $C(\g,z)$ is obtained from $C(\g',z)\in \CC'_{k,0}$ by choosing some number $m\le 2$ of distinct elements,
increasing or decreasing them by one (and removing them if the result is $4$), and adding $2-m$ fives.
The inspection of the definition of $\CC'_{k,0}$ gives us the following possibilities for $(\deg(z),C(\g,z))$,
taking into the account that every vertex in $V(\g)\setminus\{z\}$ has degree
at most $\deg(z)-2$ by Theorem~\ref{thm-deg}, that $\deg_{\g}(x)\in \{k-2,k-1\}$, and that $(\deg(z),C(\g,z))\not\in \CC'_{k,0}$:
\begin{align*}
(\deg(z),C(\g,z))\in\{&(k,\{k-2,5,5\}),(k+1,\{k-1,5,5\}), (k+1,\{k-2,6,5\}),\\
&(k+1,\{k-1,6\}), (k+1,\{k-2,5,5,5,5,5\}), (k+1,\{k-2,6,5,5,5\}),\\
&(k+1,\{k-2,6,6,5\}),(k+1,\{k-1,5,5,5,5\}), (k+1,\{k-1,6,5,5\})\}
\end{align*}
Theorem~\ref{thm-genladeg} guarantees that in (EXPB), every vertex other than $x$ has degree at least five, and since $k\ge 7$, so does $x$.
Observation~\ref{obs-sumdeg} excludes all the possibilities for $(\deg(z), C(\g,z))$
except for
$$(k+1,\{k-2,5,5,5,5,5\}), (k+1,\{k-2,6,5,5,5\}),\text{ and }(k+1,\{k-1,5,5,5,5\}).$$
Let $\g=(G,\beta)$, and let $\beta'$ be the $\Z$-boundary for $G-z$ such that $\beta'(v)=\beta(v)-\deg^+_\psi(v)+\deg^-_\psi(v)$ for each $v\in V(G-z)$.
Since $\psi$ does not extend to a nowhere-zero flow in $\g$, $(G-z,\beta')$ does not have a nowhere-zero flow.
\begin{itemize}
\item If $\deg(z)=k+1$ and $C(\g,z)=\{k-2,5,5,5,5,5\}$, then $|V(G-z)|=6$ and $G-z$ has four non-edges.  By Lemma~\ref{lemma-small},
$G-z$ consists of two copies of $K_4$ whose intersection is $K_2$ and $\beta'(v)=0$ for $v\in V(G-z)$.  Note that $G-z$ has two vertices of degree five; let $v$ be one of these degree-five vertices different from the vertex of $G$ of degree $k-2$.  Then $v$ also has degree five in $\g$, $v$ is not adjacent to $z$,
and $\beta(v)=\beta'(v)=0$.  This is a contradiction, since $\deg(v)<4+|\tau(v)|=7$.
\item If $\deg(z)=k+1$ and $C(\g,z)=\{k-2,6,5,5,5\}$ or $C(\g,z)=(k+1,\{k-1,5,5,5,5\})$, then $|V(G-z)|=5$ and $G-z$ has only one non-edge.
However, then $(G-z,\beta')$ has a nowhere-zero flow by Lemma~\ref{lemma-small}.
\end{itemize}
In any of the cases, we obtain a contradiction.
\end{proof}

Although this is mostly a technicality forced on us by our proof method,
Lemma~\ref{lemma-degbetter} also speaks about some canvases with tip of degree $k+1$
and another vertex of degree $k-2$.  Thus, in addition to Theorem~\ref{thm-degbetter},
we obtain the following consequence.
\begin{corollary}\label{cor-db3}
Let $(\g,z)$ be a non-trivial flow-critical tame canvas such that $\deg(z)\ge 8$
and there exists a vertex $x\in V(\g)\setminus\{z\}$ of degree $\deg(z)-3$.
Let $\psi$ be a tip preflow $\psi$ that does not extend to a nowhere-zero flow in $\g$.
If $\psi$ does not orient all edges between $z$ and $V(\g)\setminus\{z,x\}$ in the same direction,
then $C(\g,z)$ is either $\{\deg(z)-3,5\}$ or $\{\deg(z)-3,5,5,5\}$.
\end{corollary}

We believe the assumption of the existence of the tip preflow $\psi$ can be dropped.
\begin{conjecture}\label{conj-db3}
Let $(\g,z)$ be a non-trivial flow-critical tame canvas such that $\deg(z)\ge 8$.
If $\g$ has a vertex of degree $\deg(z)-3$, then $C(\g,z)$ is either $\{\deg(z)-3,5\}$ or $\{\deg(z)-3,5,5,5\}$.
\end{conjecture}
In particular, together with Theorem~\ref{thm-degbetter}, this would imply that
if $(\g,z)$ is a non-trivial flow-critical tame canvas and $\deg(z)=8$,
then $C(\g,z)$ is $\{6\}$, $\{5,5\}$, $\{5,5,5,5\}$, or $\emptyset$.
All of these censuses are indeed possible; e.g., $C(\g,z)=\{5,5,5,5\}$
when $\g-z$ is the clique $K_4$ and $z$ is joined by a double edge to all the vertices of this clique.

Let us remark that it is in principle possible to prove Conjecture~\ref{conj-db3} along the lines
of the proof of Lemma~\ref{lemma-degbetter} by characterizing the censuses of all easels in $\GG_{k,1}$.
An issue that prevented us from doing so is as follows.  In addition to the desired easels $(\g,z,x,\psi)$
with $\deg(z)=k$ and $\deg(x)=k-3$, the class $\GG_{k,1}$ also contains the easels with
with $\deg(z)=k+1$ and $\deg(x)=k-3$ for which $(x,\psi)$ is a witness of $(k,1)$-tallness.
There exists such an easel with $\deg(z)=k+1$ and $C(\g,z)=\{k-3, 6, 6\times 5\}$, consisting of $K_{3,5}$,
an edge between distinct vertices $x$ and $y$ of the part of size three, and the vertex $z$ joined to the vertices
in the part of size five by double edges and to $x$ by an edge of multiplicity $k-9$; it is easy to see that
$(\g,z)$ is $\psi$-critical for a suitable choice of the boundary and the tip preflow $\psi$, using the argument from~\cite{li20223}.
Since there also exists a flow-critical tame canvas $(\g_1,z_1)$ with $\deg(z_1)=8$ and $C(\g_1,z_1)=\{5,5,5,5\}$,
in the (EXPA) case (with $(\g,z)$ playing the role of $(\g',z)$),
there arises the possibility of an existence of a flow-critical tame canvas $(\g_2,z_2)$
with $\deg(z_2)=k+1$ and $C(\g_2,z_2)=\{k-3, 10\times 5\}$ (we believe no such canvas actually exists, but
this possibility does not seem easy to exclude).  Finally, in the (EXPB) case, this
would force us to exclude the existence of a flow-critical tame canvas $(\g_3,z_3)$
with $\deg(z_3)=k+1$ and $C(\g_3,z_3)=\{k-3, 12\times 5\}$.  Since we can also assume that $\g_3$
has minimum degree five, this is only a finite problem, which could be dealt by enumerating
simple graphs on $13$ vertices and all their possible boundaries (viewed as already adjusted for
the tip preflow), then testing them for flow-criticality.  However, even if the constraints on
degrees are taken into account, the number of such graphs seems too large (much more than $10^9$)
to do so easily.  A more promising approach is to enumerate only the flow-critical canvases by using Theorem~\ref{thm-gen};
however, actually implementing the algorithm based on this theorem would require non-trivial extra effort
and we leave it for a future project.

\section{Bounding the density of flow-critical graphs: Proving Theorem \ref{thm:introweakerversion}}

We start by giving a reduction that relates the density of flow-critical $\Z$-bordered graphs and tame flow-critical canvases. 

\begin{observation}\label{obs-to-tame}
Let $(G,\beta)$ be a flow-critical $\Z$-bordered graph with $|V(G)|\ge 3$.  For $d\ge 2$, let $n_d$ be the number of vertices of $G$ of degree $d$.
There exists a tame flow-critical canvas $((G',\beta'),z)$ such that $G=G'-z$, $\deg(z)=2n_2+2n_3+n_4+n_5$,
and $\deg_{G'}(v)\le \max(\deg_G(v), 6)$ for every $v\in V(G)$.  Moreover,
$$|E(G)|=3|V(G)| - \frac{1}{2}\Bigl(\deg(z) -\sum_{v\in V(G)} (\deg_{G'}(v)-6)\Bigr).$$
\end{observation}
\begin{proof}
The graph $G'$ is obtained from $G$ by joining the vertex $z$ to each vertex of $G$ of degree two or three by a double edge and to each vertex of $G$ of degree four or five by a single edge.
Let us define the $\Z$-boundary $\beta'$ and a tip preflow $\psi$ as follows:
\begin{itemize}
\item If $v\in V(G)$ has degree two, then let $\beta'(v)=0$ and choose $\psi$ on the edges between $v$ and $z$
so that $\deg^+_\psi(v)-\deg^-_\psi(v)\equiv \beta(v)\pmod 3$.
\item If $v\in V(G)$ has degree three, then let $\beta'(v)=1$ and choose $\psi$ on the edges between $v$ and $z$
so that $\deg^+_\psi(v)-\deg^-_\psi(v)\equiv \beta(v)-1\pmod 3$.
\item If $v\in V(G)$ has degree four, then choose $\beta'(v)\neq 0$ different from $\beta(v)$,
and choose $\psi$ on the edge between $v$ and $z$ so that $\deg^+_\psi(v)-\deg^-_\psi(v)\equiv \beta(v)-\beta'(v)\pmod 3$.
\item If $v\in V(G)$ has degree five, then choose $\beta'(v)\neq \beta(v)$ arbitrarily,
and choose $\psi$ on the edge between $v$ and $z$ so that $\deg^+_\psi(v)-\deg^-_\psi(v)\equiv \beta(v)-\beta'(v)\pmod 3$.
\item If $v\in V(G)$ has degree at least six, then let $\beta'(v)=\beta(v)$.
\end{itemize}
Finally, we let $\beta'(z)\equiv \deg^+_\psi(z)-\deg^-_\psi(z)\pmod 3$.
The choice of $G'$ and $\beta'$ implies that the canvas $((G',\beta'),z)$ is tame.
Moreover, since $(G,\beta)$ is flow-critical, it is easy to see that the canvas $((G',\beta'),z)$ is $\psi$-critical.
Since
$$\deg(z) + \sum_{v\in V(G)} \deg_{G'}(v)=2|E(G')|=2|E(G)|+2\deg(z),$$
we have
\begin{align*}
|E(G)|&=-\frac{1}{2}\Bigl(\deg(z)-\sum_{v\in V(G)} \deg_{G'}(v)\Bigr)\\
&=3|V(G)| - \frac{1}{2}\Bigl(\deg(z) -\sum_{v\in V(G)} (\deg_{G'}(v)-6)\Bigr).
\end{align*}
\end{proof}

In combination with Theorem~\ref{thm-degbetter}, we now prove Theorem \ref{thm:introweakerversion}, which we restate for ease of the reader.

\thmintroweakerversion*

\begin{proof}
Let $((G',\beta'),z)$ be the tame flow-critical canvas obtained in Observation~\ref{obs-to-tame} for $G$ with zero boundary.

Let us first consider the case that $\Delta(G) \geq 6$.
Let $v_0$ be a vertex of maximum degree in $G$. Note that by assumption, all other vertices have degree at most six.  Moreover, by Lemma~\ref{lemma-simple},
the graph $G$ is simple, and thus $|V(G)|\ge 7$.  The conclusions of Observation~\ref{obs-to-tame}
imply that $\deg_{G'}(v_0)=\deg(v_0)$ and together with the assumptions of the theorem, that $\deg_{G'} (v)\le 6$ for every $v\in V(G')\setminus\{z,v_0\}$.
By Theorem~\ref{thm-deg}, we have $\deg(z)\ge \deg(v_0)+2$. 
\begin{itemize}
\item If $\deg(z)=\deg(v_0)+2$, then Theorem~\ref{thm-degbetter} implies that
all vertices in $V(G')\setminus\{z,v_0\}$ have degree four and 
\begin{align*}
    \deg(z)-\sum_{v\in V(G)} (\deg_{G'}(v)-6)&=\deg(z)-(\deg(z)-8)+2(|V(G)|-1)\\[-0.2cm]
    & =2|V(G)|+6\ge 20.   
\end{align*}

\item On the other hand, if $\deg(z)\ge \deg(v_0)+3$, then since all vertices of $V(G')\setminus\{z,v_0\}$ have degree at most $6$, we have that
$$\deg(z)-\sum_{v\in V(G)} (\deg_{G'}(v)-6)\ge \deg(z)-(\deg(z)-9)=9.$$
\end{itemize}
Therefore, by the last equality from the statement of Observation~\ref{obs-to-tame},
we have 
$$|E(G)|\le 3|V(G)| - \lceil 9/2\rceil=3|V(G)|-5,$$
as desired.
Next, let us consider the case that $\Delta(G) \leq 5$, and thus all vertices of $G'$ except for $z$ have degree at most six.
Since $G$ is flow-critical, we have $|V(G)|\ge 4$, and Theorem~\ref{thm-deg} implies that $\deg(z)\ge 6$.
\begin{itemize}
\item
If $\deg(z)\ge 9$, then 
$$\deg(z)-\sum_{v\in V(G)} (\deg_{G'}(v)-6)\ge \deg(z)\ge 9.$$
\item If $\deg(z)=8$, then Theorem~\ref{thm-degbetter} implies that $V(G')\setminus\{z\}$ contains at most one vertex of degree six,
and thus
$$\deg(z)-\sum_{v\in V(G)} (\deg_{G'}(v)-6)\ge \deg(z)+(|V(G)|-1)\ge 11.$$
\item If $\deg(z)=7$, then Theorem~\ref{thm-degbetter} implies that $V(G')\setminus\{z\}$ consists of a vertex of degree five and all other
vertices have degree four, and
$$\deg(z)-\sum_{v\in V(G)} (\deg_{G'}(v)-6)\ge \deg(z)+2(|V(G)|-1)+1\ge 14.$$
\item Finally, if $\deg(z)=6$, then Theorem~\ref{thm-deg} implies that $V(G')\setminus\{z\}$ consists of vertices of degree four, and
$$\deg(z)-\sum_{v\in V(G)} (\deg_{G'}(v)-6)\ge \deg(z)+2|V(G)|\ge 14.$$
\end{itemize}
Therefore, the last equality from the statement of Observation~\ref{obs-to-tame} again implies
$$|E(G)|\le 3|V(G)| - \lceil 9/2\rceil=3|V(G)|-5.$$
\end{proof}

\section{Conclusion}
To recap the paper, we proved that every non-trivial flow-critical tame canvas
$(\g,z)$ has every vertex other than $z$ having degree at most $\deg(z) -2$.
More generally, we gave generation theorems for flow-critical canvases and
$(k,r)$-tall easels and used this to make partial progress towards
understanding the density of flow-critical graphs. It would be quite
interesting to know what the correct density of flow-critical graphs is. We
recall the conjecture from the introduction:

\conjdensitycrit*

It would be interesting to resolve this even in the easier situation of
flow-critical graphs. Note that this would imply that graphs with three pairwise
edge-disjoint spanning trees are $\mathbb{Z}_{3}$-connected, which is already an
interesting result in its own right. As the main point of the paper was to
understand tame flow-critical canvases, let us recall two conjectures on
tame flow-critical canvases that we have made and left unanswered:  

\conjfewlarge*

\conjcensusplus*

Let us remark that in principle, we could push the proof method of Theorem~\ref{thm-degbetter}
further, proving these conjectures for graphs with difference $3$, $4$, \ldots, between $\deg z$
and the maximum degree of the rest of the graph; essentially, the only difficulty (beyond the
increasing unwieldiness of the possible degree sets) is the enumeration of the small graphs
in the case (EXPB).  Unfortunately, the lack of control over the increase of the degrees
in (EXPB) prevents us from proving these conjectures in general.
\begin{ack}
    We thank the referees for their helpful comments. 
\end{ack}

\bibliographystyle{acm}
\bibliography{refs.bib}

\begin{thebibliography}{10}

\bibitem{fourtriangles}
{\sc Borodin, O.~V., Dvořák, Z., Kostochka, A.~V., Lidický, B., and Yancey,
  M.}
\newblock Planar 4-critical graphs with four triangles.
\newblock {\em European Journal of Combinatorics 41\/} (2014), 138--151.

\bibitem{trfree2}
{\sc Dvo\v{r}\'ak, Z., Kr\'al', D., and Thomas, R.}
\newblock Three-coloring triangle-free graphs on surfaces {II}. $4$-critical
  graphs in a disk.
\newblock {\em Journal of Combinatorial Theory, Series B 132\/} (2018), 1--46.

\bibitem{trfree3}
{\sc Dvo\v{r}\'ak, Z., Kr\'al', D., and Thomas, R.}
\newblock Three-coloring triangle-free graphs on surfaces {III}. {G}raphs of
  girth five.
\newblock {\em Journal of Combinatorial Theory, Series B 145\/} (2020),
  376--432.

\bibitem{trfree4}
{\sc Dvo\v{r}\'ak, Z., Kr\'al', D., and Thomas, R.}
\newblock Three-coloring triangle-free graphs on surfaces {IV}. {B}ounding face
  sizes of $4$-critical graphs.
\newblock {\em Journal of Combinatorial Theory, Series B 150\/} (2021),
  270--304.

\bibitem{col8cyc}
{\sc Dvo\v{r}\'ak, Z., and Lidick\'y, B.}
\newblock 3-coloring triangle-free planar graphs with a precolored 8-cycle.
\newblock {\em J. Graph Theory 80\/} (2015), 98--111.

\bibitem{bojanzdenek}
{\sc Dvo\v{r}\'{a}k, Z., and Mohar, B.}
\newblock On density of $\mathbb{Z}_3$-flow-critical graphs.
\newblock {\em SIAM Journal on Discrete Mathematics 37}, 2 (2023), 699--717.

\bibitem{torodialcharacterization}
{\sc Dvo\v{r}\'{a}k, Z., and Pek\'{a}rek, J.}
\newblock Characterization of 4-critical triangle-free toroidal graphs.
\newblock {\em J. Comb. Theory Ser. B 154\/} (may 2022), 336–369.

\bibitem{havelpaper}
{\sc Dvořák, Z., Král', D., and Thomas, R.}
\newblock Three-coloring triangle-free graphs on surfaces {V}. coloring planar
  graphs with distant anomalies.
\newblock {\em Journal of Combinatorial Theory, Series B 150\/} (2021),
  244--269.

\bibitem{trfree7}
{\sc Dvořák, Z., Král', D., and Thomas, R.}
\newblock Three-coloring triangle-free graphs on surfaces {VII}. a linear-time
  algorithm.
\newblock {\em Journal of Combinatorial Theory, Series B 152\/} (2022),
  483--504.

\bibitem{DVORAK2024517}
{\sc Dvořák, Z., Král', D., and Thomas, R.}
\newblock Three-coloring triangle-free graphs on surfaces vi. 3-colorability of
  quadrangulations.
\newblock {\em Journal of Combinatorial Theory, Series B 164\/} (2024),
  517--548.

\bibitem{dvořák2024sparsity}
{\sc Dvořák, Z., and Norin, S.}
\newblock Sparsity of 3-flow critical graphs, 2024.

\bibitem{grotzsch1959}
{\sc Gr\"{o}tzsch, H.}
\newblock A three-color set for three-circle-free nets on the sphere.
\newblock {\em science Z. Martin Luther Univ. Halle-Wittenberg, Math. Nat. Line
  8\/} (1959), 109--120.

\bibitem{circularflows}
{\sc JAEGER, F.}
\newblock On circular flows in graphs.
\newblock In {\em Finite and Infinite Sets}, A.~HAJNAL, L.~LOVÁSZ, and
  V.~SÓS, Eds. North-Holland, 1984, pp.~391--402.

\bibitem{KOCHOL}
{\sc Kochol, M.}
\newblock An equivalent version of the 3-flow conjecture.
\newblock {\em Journal of Combinatorial Theory, Series B 83}, 2 (2001),
  258--261.

\bibitem{Shortproof}
{\sc Kostochka, A., and Yancey, M.}
\newblock {O}re’s conjecture for $k=4$ and {G}r{\"o}tzsch’s {T}heorem.
\newblock {\em Combinatorica 34\/} (2014), 323--329.

\bibitem{torodial3flows}
{\sc Li, J., Ma, Y., Miao, Z., Shi, Y., Wang, W., and Zhang, C.-Q.}
\newblock Nowhere-zero 3-flows in toroidal graphs.
\newblock {\em Journal of Combinatorial Theory, Series B 153\/} (2022), 61--80.

\bibitem{li20223}
{\sc Li, J., Ma, Y., Shi, Y., Wang, W., and Wu, Y.}
\newblock On 3-flow-critical graphs.
\newblock {\em European Journal of Combinatorics 100\/} (2022), 103451.

\bibitem{ltwz}
{\sc Lov{\'a}sz, L.~M., Thomassen, C., Wu, Y., and Zhang, C.}
\newblock Nowhere-zero 3-flows and modulo $k$-orientations.
\newblock {\em J. Comb. Theory, Ser. {B} 103\/} (2013), 587--598.

\bibitem{densityboundtrianglefree}
{\sc Moore, B., and Smith-Roberge, E.}
\newblock A density bound for triangle-free 4-critical graphs.
\newblock {\em Journal of Graph Theory 103}, 1 (2023), 66--111.

\bibitem{THOMASSEN2003189}
{\sc Thomassen, C.}
\newblock A short list color proof of {G}rötzsch's theorem.
\newblock {\em Journal of Combinatorial Theory, Series B 88}, 1 (2003),
  189--192.

\bibitem{weak3flow}
{\sc Thomassen, C.}
\newblock The weak 3-flow conjecture and the weak circular flow conjecture.
\newblock {\em Journal of Combinatorial Theory, Series B 102}, 2 (2012),
  521--529.

\bibitem{gimbelthomassentheorem}
{\sc Thomassen, C., and Gimbel, J.}
\newblock Coloring graphs with fixed genus and girth.
\newblock {\em Transactions of the American Mathematical Society 349\/} (1997),
  4555--4564.

\bibitem{tutte1954contribution}
{\sc Tutte, W.~T.}
\newblock A contribution to the theory of chromatic polynomials.
\newblock {\em Canadian journal of mathematics 6\/} (1954), 80--91.

\end{thebibliography}

\end{document}